\documentclass[reqno]{amsart}
\usepackage{cite}

\usepackage{newtxtext}
\usepackage{txfonts}
\DeclareMathAlphabet{\mathcal}{OMS}{cmsy}{m}{n}
\DeclareSymbolFont{largesymbols}{OMX}{cmex}{m}{n}

\usepackage[unicode, colorlinks]{hyperref}
\hypersetup{
	linkcolor = blue,
	citecolor = red,
	urlcolor = teal,
	bookmarksnumbered = true,
}

\usepackage{amsmath, amsthm, amssymb, bm, graphicx, mathrsfs}
\usepackage{cases}

\newtheorem{theorem}{Theorem}[section]
\newtheorem{definition}[theorem]{Definition} 
\newtheorem{proposition}[theorem]{Proposition}
\newtheorem{lemma}[theorem]{Lemma}

\newcommand{\e}{{\rm e}}

\begin{document}

\title{Low-regularity Schr\"odinger map flow on high-dimensional periodic domains}

\author{Li Tu}
\address{School of Mathematical Sciences, Fudan University, Shanghai 200433, China.}
\email{ltu23@m.fudan.edu.cn}

\author{Yi Zhou}
\address{School of Mathematical Sciences, Fudan University, Shanghai 200433, China.}
\email{yizhou@fudan.edu.cn}


\begin{abstract}
    We study the initial-value problem for the Schr\"odinger map flow from flat torus $\mathbb{T}^d$ into compact K\"ahler manifold $\mathcal{N}$.
    When $d \geq 3$ and $\mathcal{N} = \mathbb{S}^2$, we establish local well-posedness in $H^{\sigma}_x$ with $\sigma > d/2 + 1/2$.
    In this case, the evolution equation for the gradient of the solution reduces to a certain semilinear nonlinear Schr\"odinger equation (also known as modified Schr\"odinger map flow) when formulated in orthonormal frames.
    For general compact K\"ahler targets, we only obtain local well-posedness in $H^{\sigma}_x$ with $ \sigma > d/2 + 5/6$ due to the quasilinear nature of the flow, but in all dimensions $d \geq 2$.
    To the best of our knowledge, this is the first low-regularity local well-posedness result for Schr\"odinger map flow in the periodic setting, which yields a gain of $1/2$ derivatives for $\mathbb{S}^2$ targets and $1/6$ derivatives for general K\"ahler targets compared to the classical results \cite{DW,M}.
    The key ingredients of our method are an $L_{t, x}^2$ bilinear estimate for the first case and an \emph{a priori} $L_t^6L_x^{\infty}$ estimate for the second case, which are both achieved by combining the mass/energy and momentum balance laws of the equation with a new type of div-curl lemma introduced by the second author.
\end{abstract}

\maketitle

\setcounter{tocdepth}{1}
\tableofcontents
\thispagestyle{empty}
\setcounter{page}{1}

\section{Introduction}\label{intro}

\subsection{Background and main results}

Suppose $d \geq 2$. Let $\mathcal{M}^d$ denotes Euclidean space $\mathbb{R}^d$ or flat torus $\mathbb{T}^d$ and $\mathcal{N}$ be a complete K\"ahler manifold with a compatible almost complex structure $J$.
The Schr\"odinger map flow (SMF) is a map $u\colon \mathbb{R}_+ \times \mathcal{M}^d \rightarrow \mathcal{N}$ which satisfies
\begin{equation}\label{SMF}
    \partial_tu = J(u)D_k\partial_ku,
\end{equation}
where $D$ denotes the covariant derivatives along the curve $u$ and $\partial$ denotes the usual derivatives.
Throughout this paper the repeated indices imply summation. The SMF and its various forms arise from solid-state physics and ferromagnetic materials \cite{LL}. 
An important example is the case $\mathcal{N} = \mathbb{S}^2$, where \eqref{SMF} is equivalent to the Landau-Lifshitz equation
\begin{equation}\label{LL}
    \partial_tu = u \times \Delta_x u
\end{equation}
describing Heisenberg spin chain system. Here, $\times$ denotes the cross product in $\mathbb{R}^3$.

The solution to \eqref{SMF} obeys energy conservation 
\begin{equation*}
    \mathcal{E}(u(t)) = \frac{1}{2}\int_{\mathcal{M}^d}|\nabla_xu(t)|^2 \,{\rm d}x = \mathcal{E}(u(0)), \qquad \nabla_x = (\partial_1, \cdots, \partial_d).
\end{equation*}
When $\mathcal{M}^d = \mathbb{R}^d$, it is easy to see \eqref{SMF} is invariant under scaling transform 
\begin{equation}\label{scaling}
    u(t, x) \mapsto u(\lambda^2t, \lambda x),
\end{equation}
which implies the critical index $\sigma_c = d/2$ since the norm $\Vert \cdot\Vert_{\dot H^{\sigma_c}(\mathbb{R}^d)}$ is invariant under \eqref{scaling}.

The primary objective of this paper is the \emph{low-regularity problem} for SMF \eqref{SMF}, which amounts to lowering the regularity index for initial data as much as possible, while ensuring the original initial-value problem remains well-posed. 
Restricting our attention to local-in-time solutions, we interpret local well-posedness in the following sense:
\begin{itemize}
    \item \emph{Existence}. For any initial data in $H_x^{\sigma}$, \eqref{SMF} admits a solution belonging to the class $C([0, T]; H^\sigma)$ with $T$ depending only on the $H_x^{\sigma}$ size of initial data;
    \item \emph{Uniqueness}. The solution is uniquely determined in the sense that it is the unique limit of smooth solutions;
    \item \emph{Weak Lipschitz dependence}. For two solutions $u_0(t)$ and $u_1(t)$ with $u_0(0) = \phi_0$ and $u_1(t) = \phi_1$,
    \begin{equation}\label{WLD}
        \Vert u_0(t) - u_1(t)\Vert_{H_x^{\sigma - 1}} \lesssim \Vert \phi_0 - \phi_1\Vert_{H_x^{\sigma - 1}}, \qquad t \in [0, T].
    \end{equation} 
\end{itemize}

The difficulties of establishing well-posedness for SMF \eqref{SMF} stem from its quasilinear nature, namely, the presence of almost complex structure $J(u)$ acting on the highest-order term $D_k\partial_ku$, as well as the periodicity of the domain leading to lack of dispersion.
In contrast to general SMF, for the solution to \eqref{LL}, the evolution equation of its gradient favorably reduces to a certain semilinear NLS (nonlinear Schr\"odinger equation) when formulated in the notations of orthonormal frames.
This semilinear NLS is also known as modified Schr\"odinger map flow \cite{CSU,NSU1,KN,BIKT2}.

The first result of our paper concerns the local well-posedness for Landau-Lifshitz equation \eqref{LL} in dimensions $d \geq 3$.
We use energy method and mainly adpot the notations in \cite{BIKT2}. Specifically, we aim to investigate NLS with magnetic nonlinearity
\begin{equation}\label{Mag}
    (\sqrt{-1}\partial_t + \Delta_x)\psi_k = \sqrt{-1}\bm{A} \cdot \nabla_x\psi_k, \qquad k = 1, \cdots, d,
\end{equation}
where $\bm{\psi} = (\psi_1, \cdots, \psi_d)\colon \mathbb{R}_+ \times \mathbb{T}^d \rightarrow \mathbb{C}^d$, $\bm{A} = \nabla_x^{-1}(\psi_k\overline{\bm{\psi}})$ satisfies the Coulomb gauge condition $\nabla_x \cdot \bm{A} = 0$ and $\nabla_x^{-1}$ is a Fourier multiplier defined by $\xi \mapsto |\xi|^{-1}$.

To handle different frequency interactions in the nonlinearity, we first provide an $L_{t, x}^2$ bilinear estimate to deal with the high-low (or low-high) interaction from Bony paraproduct decomposition.
This $L_{t, x}^2$ bilinear estimate will be proved in a physical space approach via a novel div-curl lemma first introduced by the second author, combinig with the energy and momentum balance laws of the equation.
For the high-high interaction case, we follow the work by Colliander-Keel-Staffilani-Takaoka-Tao \cite{C} to apply the interaction Morawetz estimates,  which yield $L_{t, x}^4$ control for (the Littlewood-Paley blocks of) the solution.
However, the classical interaction Morawetz estimates rely ultimately on the fact that $-\Delta\Delta|x|$ is a well-defined and nonnegative termpered distribution in $\mathbb{R}^d$ when $d \geq 3$, and cannot be applied to the periodic setting directly.
To address this we employ an appropriate spatial cutoff function.
This approach enables us to establish a ``localized'' version of interaction Morawetz estimate, which is sufficient for our purpose.
Finally, we note that although there exists one derivative loss from the term $\nabla_x\psi_k$, the pure imaginary part $\sqrt{-1}$ and Coulomb gauge condition enable us to remove the derivatives at high frequency using integration by parts and commutator estimates, and then close the energy estimates.

For $\sigma \geq 0$ we define extrinsic Sobolev spaces 
\begin{equation*}
    H^{\sigma}(\mathbb{T}^d; \mathbb{S}^2) := \{u \in H^{\sigma}(\mathbb{T}^d; \mathbb{R}^3)\colon |u| \equiv 1\ {\rm a.e.}\},
\end{equation*}
where $H^{\sigma}(\mathbb{T}^d; \mathbb{R}^3)$ denotes usual Sobolev spaces.
Our first result reads 

\begin{theorem}\label{LWP2}
    Suppose $d \geq 3$ and $\sigma > d/2 + 1/2$. Then \eqref{LL} is locally well-posed for initial data $\phi \in H^\sigma(\mathbb{T}^d; \mathbb{S}^2)$.
\end{theorem}

The improvments of Theorem \ref{LWP2} are substantial, yielding a gain of $1/2$ derivatives compared to the classical local well-posedness results in $H_x^{(d/2 + 1)+}$\footnote{
    In this article, for $\sigma \in \mathbb{R}$ we use $H^{\sigma+}$ to denote $H^{\sigma + \varepsilon}$ for arbitrarily small $\varepsilon > 0$.
} \cite{DW,M}.
Moreover, our analysis also implies that \eqref{Mag} is locally well-posed in $H_x^{(d/2 - 1/2)+}$.
For recent works concerning well-posedness for NLS on periodic domains we refer to \cite{LZ1,LZ2,HK,HK2}.

We next consider the local well-posedness for \eqref{SMF} with general compact K\"ahler targets, where we retain the quasilinear feature of the original equation and apply the standard energy method.
Roughly speaking, one needs to control the $L_x^{\infty}$ norm for the gradient of the solution during energy estimates, which, by Sobolev embeddings, implies local well-posedness in $H_x^{(d/2 + 1)+}$.
To overcome this regularity barrier we adopt the framework of \cite{WZ}, based on the balance laws of SMF and the div-curl lemma to obtain some nonlinear estimates that close the energy estimates (here, we use Littlewood-Paley decomposition additionally).
Following this approach, in this paper we prove an enhanced \emph{a priori} $L_t^6L_x^{\infty}$ estimate for the gradient of the solution, which enables us to obtain a low-regularity solution by energy method and commutator estimates.
Consequently, we show that \eqref{SMF} is locally well-posed in $H_x^{(d/2 + 5/6)+}$, gaining $1/6$ derivatives in all dimensions $d \geq 2$. 

Before stating our second result we introduce several notations.
Suppose $\mathcal{N}$ is a compact K\"ahler manifold isometrically embedded into $\mathbb{R}^N$, where $N$ is an integer sufficiently large.
Under this embedding, we may consider points in $\mathcal{N}$ as vectors in $\mathbb{R}^N$ and define extrinsic Sobolev spaces
\begin{equation*}
    H^{\sigma}(\mathbb{T}^d; \mathcal{N}) := \{u \in H^{\sigma}(\mathbb{T}^d; \mathbb{R}^N)\colon u \in \mathcal{N}\ {\rm a.e.}\},
\end{equation*}
where $\sigma \geq 0$ and $H^{\sigma}(\mathbb{T}^d; \mathbb{R}^N)$ denotes usual Sobolev spaces.
Note that $H^{\sigma}(\mathbb{T}^d; \mathcal{N}) \subseteq L^2(\mathbb{T}^d; \mathcal{N})$ since $\mathbb{T}^d$ is compact.
Moreover, we assume
\begin{itemize}
    \item $J$ is an $N \times N$ anti-symmetric matrix satisfying $J \in C^{\infty}, J' \in {\rm Lip}$ and $J'(0) = 0$;
    \item The unit normals, denoted by $n_i = n_i(u)$, of $\mathcal{N}$ exist everywhere. Moreover, $n_i \in C^{\infty}, n_i' \in {\rm Lip}$ and $n_i'(0) = 0$.
\end{itemize}
The covariant derivatives can be expressed as  
\begin{equation}\label{DV}
        DV = \partial V - \langle \partial V, n_i(u)\rangle n_i(u), \qquad V \in \mathbb{R}^N.
\end{equation}
where $\langle \cdot, \cdot\rangle$ stands for standard Euclidean inner product.

We shall prove 

\begin{theorem}\label{LWP1}
    Let $\sigma > d/2 + 5/6$. Then \eqref{SMF} is locally well-posed for initial data $\phi \in H^\sigma(\mathbb{T}^d; \mathcal{N})$.
    Moreover, the solution $u$ satisfies 
    \begin{equation}\label{Lwuqiong}
        \nabla_xu \in L^6(0, T; L^{\infty}(\mathbb{T}^d; \mathcal{N})).
    \end{equation}
\end{theorem}

The desired solutions in Theorem \ref{LWP2} and \ref{LWP1} will be constructed by compactness arguments. However, the usual smooth approximation procedure for initial data cannot be applied directly, as the approximating sequence should take values in the target manifolds.
To preserve this geometric constraint, we regularize the initial data by evolving it along the heat flow of harmonic map for a short time interval $[0, \varepsilon]$.
By taking the heat time $\varepsilon$ as the approximation parameter, the parabolic smoothing effect guarantees a smooth sequence that satisfies the target manifold constraint, and naturally converges to the original initial data as $\varepsilon \rightarrow 0$.

\subsection{Priori works on well-posedness}

We briefly recall the non-exhaustive works concerning the well-posedness for SMF \eqref{SMF} and the Landau-Lifshitz equation \eqref{LL}.
The local theory for SMF \eqref{SMF} with compact K\"ahler targets was developed by Ding-Wang \cite{DW} and McGahagan \cite{M} using different approximation schemes.
For the one-dimensional domain case, Rodnianski-Rubinstein-Staffilani \cite{RRS} proved the global well-posedness for SMF from $\mathbb{R}$ and $\mathbb{T}$ into compact K\"ahler manifolds and Riemann surfaces, respectively, in $H^{k}_x$ for integers $k > 2$, which partially solves the Wei-Yue Ding's Conjecture \cite{D}.
The main ideas in \cite{RRS} is to find a global frame using parallel transport, and then translate SMF into NLS where classical methods for dispersive equations can be applied, such as Strichartz estimates.
However, since $\mathbb{T}$ is not simply-connected, the use of parallel transport introduces holonomy and the discussions in \cite{RRS} become more delicate. 

From a different perspective, recently Wang and the second author \cite{WZ} proved that SMF from $\mathbb{T}$ into compact K\"ahler manifolds is globally well-posed in $H_x^{k}$ for integers $k > 2$.
The novelty of their method is to fully exploit the balance laws of SMF, combining with a new type of div-curl lemma which was first introduced by the second author \cite{Z1}, to get some nonlinear estimates that close the energy estimates.
Neither global frames nor dispersive estimates were used in their approach. Moreover, their method can be also applied to the SMF from $\mathbb{R}$. 

The Landau-Lifshitz equation \eqref{LL} on $\mathbb{R}^d$ has been studied using modified Schr\"odinger map flow (mSMF); see \cite{BIKT2,CSU,NSU1,KN}\footnote{
    The estimates on $\psi_k$ (the solutions for mSMF) can be transferred back to the original solution $u$ using the relations between $\psi_k$ and $u$, Sobolev embeddings and an inductive argument; see Section \ref{1.3} and also \cite{BIKT2} for detailed discussions.
}.
When $d \geq 2$, Bejenaru \cite{B1} and Ionescu-Kenig \cite{IK1} established the small-data local well-posedness in subcritical Sobolev spaces $H_x^{(d/2)+}$, 
where they use stereographic projection to translate \eqref{LL} into a derivative NLS, and analyze it perturbatively.
Using this stereographic model, the small data global well-posedness in critical Besov spaces was obtained in \cite{B2,IK2}.
However, it turns out that this stereographic model becomes less effective when considering well-posedness in critical Sobolev spaces.
In \cite{BIKT2}, the authors continued the study of mSMF, and proved the small data global well-posdeness for \eqref{LL} in critical Sobolev spaces with dimensions $d \geq 4$ under Coulomb gauge. 
This result was later improved to all dimensions $d \geq 2$ by Bejenaru-Ionescu-Kenig-Tataru \cite{BIKT} in the caloric gauge setting.
Very recently, Li \cite{Li1,Li2} generalized the results in \cite{BIKT} to general compact K\"ahler targets.

\subsection{Outline of the proof for Theorem \ref{LWP2}}\label{1.3}

To construct the approximating sequence for given initial data $\phi \in H^{\sigma}(\mathbb{T}^d; \mathbb{S}^2)$, we consider the heat flow of harmonic map
\begin{equation*}
    \begin{cases}
        \partial_s\widetilde{u} = D_k\partial_k\widetilde{u} = \Delta_x\widetilde{u} + |\nabla_x\widetilde{u}|^2\widetilde{u}, \\ 
        s = 0\colon \widetilde{u} = \phi,
    \end{cases}
\end{equation*}
where we note that $\widetilde{u}$ is the unit normal to the tangent space $T_{\widetilde{u}}\mathbb{S}^2$, and hence by \eqref{DV}
\begin{equation*}
    D_k\partial_k\widetilde{u} = \Delta_x\widetilde{u} - \langle \Delta_x\widetilde{u}, \widetilde{u}\rangle \widetilde{u} = \Delta_x \widetilde{u} + |\nabla_x \widetilde{u}|^2\widetilde{u}.
\end{equation*}
For $0 < \varepsilon \ll 1$ we set $\phi^{\varepsilon}(x) := \widetilde{u}(\varepsilon, x)$, then $\phi^{\varepsilon} \in C^{\infty}(\mathbb{T}^d; \mathbb{S}^2)$ and $\phi^{\varepsilon} \rightarrow \phi$ in $H^\sigma(\mathbb{T}^d; \mathbb{S}^2)$ as $\varepsilon \rightarrow 0$.

Let $u^{\varepsilon}(t, x)$ be the solution to \eqref{SMF} with initial data $\phi^{\varepsilon}$, where $0 \leq t \leq T = T(\varepsilon)$.
In the following we aim to prove that, for $\sigma > d/2 + 1/2$ there exists $T^* = T^*(\Vert \phi\Vert_{H_x^\sigma}) > 0$, such that for all $t \in [0, T^*]$ 
\begin{equation}\label{bound}
    \Vert u^\varepsilon(t)\Vert_{H_x^\sigma} \lesssim \Vert \phi^\varepsilon\Vert_{H_x^\sigma}.
\end{equation}
This uniform bound ensures that the approximation solutions $u^{\varepsilon}$ exist at least up to the time $T^*$.
Therefore, we can restrict $u^{\varepsilon}$ to the time interval $[0, T^*]$ and then the desired solution follows from compactness method.

As previously mentioned, it is more convenient to work on the mSMF, namely, the equation satisfied by $\partial u$ in terms of orthonormal frame, which is constructed as follows:
Given $Q \in \mathbb{S}^2$, we choose $0 \neq e_1^Q \in T_Q\mathbb{S}^2$ satisfying $|e_1^Q| = 1$, and set $e_2^Q = Q \times e_1^Q$.
We choose $O \in {\rm SO}(3)$ such that $O(t, x)Q = u(t, x)$ and then set 
\begin{equation}\label{frame}
    \begin{cases}
        e_1(t, x) = O(t, x)e_1^Q, \\ 
        e_2(t, x) = O(t, x)e_2^Q.
    \end{cases}
\end{equation}

For $\alpha = 0, 1, \cdots, d$, we introduce differential fields
\begin{equation*}
    \psi_{\alpha}(t, x) := \langle \partial_{\alpha}u(t, x), e_1(t, x)\rangle + \sqrt{-1}\langle \partial_{\alpha}u(t, x), e_2(t, x)\rangle
\end{equation*}
and connection coefficients
\begin{equation*}
    A_{\alpha}(t, x) := \langle \partial_{\alpha}e_1(t, x), e_2(t, x)\rangle,
\end{equation*}
where $\partial_0 = \partial_t$. In terms of the definitions above, we have 
\begin{equation}\label{A}
    \partial_{\alpha}u = \Re(\psi_{\alpha})e_1 + \Im(\psi_{\alpha})e_2
\end{equation}
and 
\begin{equation}\label{B}
    \begin{cases}
        \partial_{\alpha}e_1 = -\Re(\psi_{\alpha})u + A_{\alpha}e_2, \\ 
        \partial_{\alpha}e_2 = -\Im(\psi_{\alpha})u - A_{\alpha}e_1
    \end{cases}
\end{equation}
due to the orthogonality of $u, e_1, e_2$. We define $D_{\alpha}^{\bm{A}} := \partial_{\alpha} + \sqrt{-1}A_{\alpha}$. Direct computations show ($\alpha, \beta = 0, 1, \cdots, d$)
\begin{equation}\label{C1}
    D_{\alpha}^{\bm{A}}\psi_{\beta} = D_{\beta}^{\bm{A}}\psi_{\alpha}, 
\end{equation}
and 
\begin{equation}\label{C}
    \partial_{\alpha}A_{\beta} - \partial_{\beta}A_{\alpha} = \Im(\psi_{\alpha}\overline{\psi_{\beta}}).
\end{equation}

We note that the system $(\psi_{\alpha}, A_{\beta})$ is not uniquely defined yet due to the gauge freedom 
\begin{equation}\label{gauge1}
    \psi_{\alpha} \mapsto \e^{\sqrt{-1}\theta}\psi_{\alpha}, \quad A_{\beta} \mapsto A_{\beta} + \partial_{\beta}\theta,
\end{equation}
where $\theta$ is any real-valued function. To fix the gauge, we choose $\theta = \theta(t, x)$ to satisfy 
\begin{equation}\label{gauge}
    \begin{cases}
        \Delta_x\theta = -\nabla_x \cdot \bm{A}, \qquad \bm{A} = (A_1, \cdots, A_d), \\ 
        \displaystyle\partial_t\theta = -\frac{1}{|\mathbb{T}^d|}\int_{\mathbb{T}^d}A_0(t, x) \,{\rm d}x.
    \end{cases}
\end{equation}
Such $\theta$ always exists; for example, 
\begin{equation*}
    \theta(t, x) = \nabla_x^{-1}\sum_{k = 1}^d\mathcal{R}_kA_k(t, x) + C_{\theta}(t),
\end{equation*}
where $C_{\theta}$ satisfies
\begin{equation*}
    C_{\theta}'(t) = -\partial_t\nabla_x^{-1}\sum_{k = 1}^d\mathcal{R}_kA_k(t, x) - \frac{1}{|\mathbb{T}^d|}\int_{\mathbb{T}^d}A_0(t, x) \,{\rm d}x.
\end{equation*}
Here the operator $\nabla^{-1}_x$ is defined by the Fourier multiplier $\xi \mapsto |\xi|^{-1}$, and $\mathcal{R}_k$ denotes the Riesz-type transform defined by the Fourier multiplier $\xi \mapsto \sqrt{-1}\xi_k/|\xi|$.
It is clear that, after the gauge transform \eqref{gauge1}-\eqref{gauge}, we have 
\begin{equation}\label{Coulomb}
    \begin{cases}
        \nabla_x \cdot \bm{A} = 0, \\ 
        \displaystyle\frac{1}{|\mathbb{T}^d|}\int_{\mathbb{T}^d}A_0(t, x) \,{\rm d}x = 0.
    \end{cases}
\end{equation}

We call the equation satisfied by $\bm{\psi} = (\psi_1, \cdots, \psi_d)$ the mSMF; see \eqref{psi-e}.
To proceed further we need the following lemma.

\begin{lemma}[Quantitive estimates for $\bm{\psi}$]
    For $\phi \in H^{\sigma}(\mathbb{T}^d; \mathbb{S}^2)$ with $\sigma > d/2 + 1/2$ we have  
    \begin{equation}\label{psi}
        \Vert \bm{\psi}(0)\Vert_{H_x^{\sigma - 1}} \leq M(\Vert \phi\Vert_{H_x^\sigma}).
    \end{equation}
    \begin{proof}
        We refer to \cite[Lemma 2.5]{BIKT2} for the proof in the case $\phi \in H^{\sigma}(\mathbb{R}^d; \mathbb{S}^2)$.
        However, the proof for \cite[Lemma 2.5]{BIKT2} can be straightforwardly extended to the periodic case.
    \end{proof} 
\end{lemma}

Working on mSMF, it suffices to prove 
\begin{equation}\label{bound21}
    \Vert \bm{\psi}(t)\Vert_{H_x^{\sigma - 1}} \leq M = M(\Vert \bm{\psi}(0)\Vert_{H_x^{\sigma - 1}}), \quad t \in [0, T^*],
\end{equation}
where $T^* = T^*(\Vert \bm{\psi}(0)\Vert_{H_x^{\sigma - 1}}) > 0$. Once \eqref{bound21} is obtained, from \eqref{A}-\eqref{B} and \eqref{psi}, an inductive argument then shows (see \cite{BIKT2})
\begin{equation*}
    \Vert u(t)\Vert_{H_x^\sigma} \leq M(\Vert \phi\Vert_{H_x^\sigma}), \qquad t \in [0, T^*].
\end{equation*}

The proof for \eqref{bound21} will be carried out through a bootstrap argument.
For simplicity we denote by $u := u^\varepsilon$ and $\phi := \phi^\varepsilon$.
Suppose $\sigma = (d + 1)/2 + \varepsilon$ and we only consider the case $0 < \varepsilon \leq 1/2$\footnote{
When $\varepsilon > 1/2, \sigma > d/2 + 1$, and the local well-posedness for \eqref{LL} follows from \cite{DW,M}. 
}. We choose $\{a_{\lambda}\}_{\lambda \in 2^{\mathbb{N}}} \in \ell^2$ satisfying
\begin{equation}\label{fre1}
    \begin{cases}
        \lambda^{\sigma - 1}\Vert P_{\lambda}\bm{\psi}(0)\Vert_{L_x^2} \lesssim a_{\lambda}, \\ 
        \displaystyle \left(\sum_{\lambda \geq 1}a_{\lambda}^2\right)^{1/2} \lesssim \Vert \bm{\psi}(0)\Vert_{H_x^{\sigma - 1}}, \\ 
        a_{\lambda} \lesssim 2^{|k - k'|\delta}a_{\lambda'} \qquad (\lambda = 2^k, \lambda' = 2^{k'}),
    \end{cases}
\end{equation}
where $P_{\lambda}$ is the Littlewood-Paley operators (see Section \ref{LPP}) and $\delta = \delta(\sigma, d) > 0$ is a constant sufficiently small.
Here, $\{a_{\lambda}\}$ is referred to as the frequency envelope\footnote{
    This notion was originally introduced by Tao \cite{T}.
} associated with $\bm{\psi}(0) \in H_x^{\sigma - 1}$. 

For $a \in \mathbb{T}^d$, we denote by $f^a(x) := f(x - a)$ the translation of a function $f$. 
We formulate the following bootstrap assumptions ($i, k, \ell = 1, \cdots, d$):
\begin{gather}
    \label{11}\Vert P_{\lambda}\bm{\psi}(t)\Vert_{L_x^2} \leq t^{-\frac{\varepsilon}{8}}\lambda^{-\sigma + 1}a_{\lambda}, \\ 
    \label{13}\sup_a\left\Vert \Vert P_{\mu}\psi_k^a\Vert_{L_{\widehat{x}_i}^2}\Vert P_{\lambda, e_i}\psi_{\ell}\Vert_{L_{\widehat{x}_i}^2} \right\Vert_{L_{t, x_i}^2} \leq t^{-\frac{\varepsilon}{8}}\mu^{-\sigma + 1}\lambda^{-\sigma + 1}a_{\mu}a_{\lambda} \quad (\mu \ll \lambda), \\
    \label{20}\Vert P_{\lambda}\bm{\psi}\Vert_{L_{t, y}^4L_z^2} \leq t^{-\frac{\varepsilon}{8}}\lambda^{-\sigma + \frac{3}{2}}a_{\lambda}, 
\end{gather}
where $x = (y, z) \in \mathbb{T}^3 \times \mathbb{T}^{d - 3}$, $\widehat{x}_i = (x_1, \cdots, x_{i - 1}, x_{i + 1}, \cdots, x_d) \in \mathbb{T}^{d - 1}$, $P_{\lambda}u = \sum_{i = 1}^dP_{\lambda, e_i}u$, $\widehat{P_{\lambda, e_i}u}(\xi) = Q_{e_i}(\xi/|\xi|)\widehat{P_{\lambda}u}(\xi)$ and
\begin{equation*}
    {\rm supp}\ Q_{e_i} \subseteq \{\xi = (\xi_1, \cdots, \xi_d) \in \mathbb{Z}^d \colon |\xi_i| \gtrsim |\xi|\}, \qquad i = 1, \cdots, d.
\end{equation*}
Our goal is to improve these bounds to 
\begin{gather}
    \label{apri}\Vert P_{\lambda}\bm{\psi}(t)\Vert_{L_x^2} \lesssim \lambda^{-\sigma + 1}a_{\lambda}, \\ 
    \label{16}\sup_a\left\Vert \Vert P_{\mu}\psi_k^a\Vert_{L_{\widehat{x}_i}^2}\Vert P_{\lambda, e_i}\psi_{\ell}\Vert_{L_{\widehat{x}_i}^2} \right\Vert_{L_{t, x_i}^2} \lesssim \mu^{-\sigma + 1}\lambda^{-\sigma + 1}a_{\mu}a_{\lambda} \quad (\mu \ll \lambda), \\
    \label{17}\Vert P_{\lambda}\bm{\psi}\Vert_{L_{t, x}^4} \lesssim \lambda^{-\sigma + \frac{3}{2}}a_{\lambda},
\end{gather}
where $t \in [0, T^*]$ and $T^* = T^*(\Vert \bm{\psi}(0)\Vert_{H_x^{\sigma - 1}}) > 0$.
The proof for this bootstrap process will be completed in Sections \ref{EnergyEs}-\ref{Morawetzx}.
To prove the weak Lipschitz dependence, namely, to estimate the $H_x^{\sigma - 1}$ norm for the difference of two solutions we use the homotopy method; see Section \ref{34}.

\subsection{Outline of the proof for Theorem \ref{LWP1}}

Similar to the proof for Theorem \ref{LWP2}, we first employ an approximate procedure, which leads us to solve the following heat flow:
\begin{equation*}
    \begin{cases}
        \partial_s\widetilde{u} = D_k\partial_k\widetilde{u}, \\ 
        s = 0\colon \widetilde{u} = \phi,
    \end{cases}
\end{equation*}
where $\phi$ is a smooth initial data to \eqref{SMF}. 
For $0 < \varepsilon \ll 1$ we set $\phi^{\varepsilon}(x) := \widetilde{u}(0, x)$ the approximation of $\phi$ and $u^{\varepsilon}$ as the evolution along \eqref{SMF} with initial data $\phi^\varepsilon$.
With a slight abuse of notation we also set $u := u^{\varepsilon}$ and $\phi := \phi^{\varepsilon}$, then it suffices to construct the desired solution by proving 
\begin{equation}\label{24}
    \Vert u(t)\Vert_{H_x^\sigma} \lesssim \Vert \phi\Vert_{H_x^\sigma}, \qquad \sigma > \frac{d}{2} + \frac{5}{6},\quad t \in [0, T^*],
\end{equation}
where $T^* > 0$ depending only on the $H^{\sigma}_x$ size of $\phi$.

Again, we use a bootstrap argument to prove \eqref{24}. For $\sigma > d/2 + 5/6$ we choose a frequency envelope $\{b_{\lambda}\} \in \ell^2$ associated with $\phi \in H_x^{\sigma}$:
\begin{equation*}
    \begin{cases}
                \lambda^{\sigma}\Vert P_{\lambda}\phi\Vert_{L_x^2} \lesssim b_{\lambda}, \\ 
        \displaystyle \left(\sum_{\lambda \geq 1}b_{\lambda}^2\right)^{1/2} \lesssim \Vert \phi\Vert_{H_x^{\sigma}}, \\ 
        b_{\lambda} \lesssim 2^{|\ell - \ell'|\delta}b_{\lambda'} \qquad (\lambda = 2^{\ell}, \lambda' = 2^{\ell'}),
    \end{cases}
\end{equation*}
where $\delta = \delta(\sigma, d) > 0$ is a small constant. We assume
\begin{gather}
    \label{boot1}\Vert \nabla_xP_{\lambda}u(t)\Vert_{L_x^2} \leq t^{-\frac{1}{24}}\lambda^{-\sigma + 1}a_{\lambda}, \\
    \label{boot2}\Vert \nabla_xP_{\lambda}u\Vert_{L_t^6L_x^{\infty}} \leq t^{-\frac{1}{24}}\lambda^{-\sigma + \frac{d}{2} + \frac{5}{6}}a_{\lambda}.
\end{gather}
Under the assumptions above, in Sections \ref{31}-\ref{Linfty} we will prove
\begin{gather}
    \label{bound1}\Vert \nabla_xP_{\lambda}u(t)\Vert_{L_x^2} \lesssim \lambda^{-\sigma + 1}a_{\lambda}, \\
    \label{bound2}\Vert \nabla_xP_{\lambda}u\Vert_{L_t^6L_x^{\infty}} \lesssim \lambda^{-\sigma + \frac{d}{2} + \frac{5}{6}}a_{\lambda},
\end{gather}
where $t \in [0, T^*]$ and $T^* = T^*(\Vert \phi\Vert_{H_x^\sigma}) > 0$. It is clear that \eqref{24} follows from \eqref{bound1}, and we have the desired solution $u$ which satisfies \eqref{Lwuqiong} by \eqref{bound2}.
The weak Lipschitz dependence can be obtained in the same way as for the Landau-Lifshitz equation \eqref{LL} in Section \ref{34} using homotopy method, and we omit the details.

\section{Notations and Preliminaries}\label{2}

\subsection{Littlewood-Paley theory for periodic functions}\label{LPP}

Let $d \geq 2$ be an integer. For $f \in L^1(\mathbb{T}^d)$ and $\xi \in \mathbb{Z}^d$, the Fourier blocks $\widehat{f}(\xi)$ are defined by 
\begin{equation*}
    \widehat{f}(\xi) := \frac{1}{(2\pi)^d}\int_{\mathbb{T}^d}f(x)\mathrm{e}^{-\sqrt{-1}\xi \cdot x} \,{\rm d}x.
\end{equation*}
For dyadic number $\lambda \in 2^{\mathbb{N}}$, we set 
\begin{equation*}
    A_{\lambda} := \{\xi = (\xi_1, \cdots, \xi_d) \in \mathbb{Z}^d\colon |\xi_i| \leq \lambda, i = 1, \cdots, d\}.
\end{equation*}
The Littlewood-Paley projection operators are defined as (set $A_{1/2} = \varnothing$)
\begin{align*}
    P_{\lambda}f(x) := \sum_{\xi \in A_{\lambda}\setminus A_{\lambda/2}}\widehat{f}(\xi)\mathrm{e}^{\sqrt{-1}\xi \cdot x}, \qquad P_{\leq \lambda} := \sum_{\mu\ {\rm dyadic}, \mu \leq \lambda}P_{\mu}.
\end{align*}
Then we have the (inhomogeneous) Littlewood-Paley decomposition
\begin{equation*}
    f = P_{< 1}f + \sum_{\lambda\ {\rm dyadic}, \lambda > 1}P_{\lambda}f.
\end{equation*}
With a slight abuse of notations we also write $f = \sum_{\lambda \geq 1}P_{\lambda}f$.

The projection operators $P_{\lambda}$ and $P_{\leq \lambda}$ for periodic functions share similar properties with those in the Euclidean setting:

\begin{proposition}\label{prop21}
    Given dyadic numbers $\lambda$ and $\mu$.
    \begin{itemize}
        \item (Boundedness) For $f \in L^p(\mathbb{T}^d), 1 < p \leq \infty$, we have 
        \begin{equation*}
            \Vert P_{\lambda}f\Vert_{L^p} \lesssim \Vert f\Vert_{L^p}, \qquad \Vert P_{\leq \lambda}f\Vert_{L^p} \lesssim \Vert f\Vert_{L^p};
        \end{equation*}
        \item (Bernstein inequalities) Let $\sigma \geq 0, 1 \leq q \leq p \leq \infty$. There holds 
        \begin{align*}
            &\Vert (\sqrt{-\Delta})^{\sigma}P_{\lambda}f\Vert_{L^p} \sim \lambda^\sigma\Vert P_{\lambda}f\Vert_{L^q}, \\ 
            &\Vert P_{\lambda}f\Vert_{L^p} \lesssim \lambda^{d\left(\frac{1}{q} - \frac{1}{p}\right)}\Vert P_{\lambda}f\Vert_{L^q}, \\ 
            &\Vert P_{\leq \lambda}f\Vert_{L^p} \lesssim \lambda^{d\left(\frac{1}{q} - \frac{1}{p}\right)}\Vert P_{\leq \lambda}f\Vert_{L^q}.
        \end{align*}
    \end{itemize}
\end{proposition}

For the proof for Proposition \ref{prop21} and detailed discussions on Littlewood-Paley theory on periodic domains we refer to \cite{DHWX}.

For $\sigma \in \mathbb{R}, 1 \leq p, q \leq \infty$, the Besov norm is defined as 
\begin{equation*}
    \Vert f \Vert_{B^\sigma_{p, q}} := \Vert \lambda^\sigma\Vert P_{\lambda}f\Vert_{L^p}\Vert_{\ell^q}.
\end{equation*}
In particular, for the case $p = q = 2$, the space $B^\sigma_{2, 2}$ can be identified with $H^{\sigma}$, the usual suqare-integrable Sobolev spaces:
\begin{equation*}
    B^\sigma_{2, 2}(\mathbb{T}^d) = H^\sigma(\mathbb{T}^d), \qquad \forall \sigma \in \mathbb{R}.
\end{equation*}
This allows us to represent the $H^\sigma$ norm using dyadic decomposition:
\begin{equation*}
    \Vert f\Vert_{H^\sigma} = \left(\sum_{\lambda \geq 1}\left(\lambda^\sigma\Vert P_{\lambda}f\Vert_{L^2}\right)^2\right)^{\frac{1}{2}}.
\end{equation*}

For the Littlewood-Paley decomposition of the product of two functions, we have the following Bony paraproduct decomposition
\begin{equation*}
    P_{\lambda}(fg) = \sum_{\lambda_1 \lesssim \lambda_2 \sim \lambda}P_{\lambda_1}fP_{\lambda_2}g + \sum_{\lambda_2 \lesssim \lambda_1 \sim \lambda} P_{\lambda_1}fP_{\lambda_2}g + P_{\lambda}\left(\sum_{\lambda_1 \sim \lambda_2 \gtrsim \lambda}P_{\lambda_1}fP_{\lambda_2}g\right),
\end{equation*}
which can be alternatively written as 
\begin{equation}\label{Bony}
    P_{\lambda}(fg) = \sum_{\lambda_1 \ll \lambda_2 \sim \lambda}P_{\lambda_1}fP_{\lambda_2}g + \sum_{\lambda_2 \ll \lambda_1 \sim \lambda} P_{\lambda_1}fP_{\lambda_2}g + P_{\lambda}\left(\sum_{\lambda_1 \sim \lambda_2 \gtrsim \lambda}P_{\lambda_1}fP_{\lambda_2}g\right)
\end{equation}
since the finite sum $\sum_{\lambda' \sim \lambda}P_{\lambda'}$ would make no difference in this paper. 

\subsection{Balance laws for NLS and SMF}\label{Balance}

We start with $1d$ NLS
\begin{equation}\label{1NLS}
    (\sqrt{-1}\partial_t + \partial_x^2)u = \mathcal{N}(u), \qquad u\colon \mathbb{R}_+ \times \mathbb{T} \rightarrow \mathbb{C}.
\end{equation}
By multiplying $\overline{u}$ on both sides of the equation and taking imaginary parts, we get 
\begin{equation*}
    \partial_t\left(\frac{1}{2}|u|^2\right) + \partial_x\Im(\overline{u}\partial_xu) = \Im(\overline{u}\mathcal{N}(u)).
\end{equation*}
Since the Littlewood-Paley operators commutes with usual derivatives, it follows that 
\begin{equation*}
    \partial_t\left(\frac{1}{2}|P_{\lambda}u|^2\right) + \partial_x\Im(\overline{P_{\lambda}u}\partial_xP_{\lambda}u) = \Im(\overline{P_{\lambda}u}P_{\lambda}\mathcal{N}(u)).
\end{equation*}
Moreover, if we replace $\partial_x$ by $\partial_x^C := \partial_x + \sqrt{-1}C$ in \eqref{1NLS} for some real constant $C$, it is easy to verify
\begin{equation*}
    \partial_t\left(\frac{1}{2}|P_{\lambda}u|^2\right) + \partial_x\Im(\overline{P_{\lambda}u}\partial_x^CP_{\lambda}u) = \Im(\overline{P_{\lambda}u}P_{\lambda}\mathcal{N}(u)),
\end{equation*}
and we also have the momentum balance law 
\begin{equation*}
    \partial_t\Im(\overline{P_{\lambda}u}\partial_x^CP_{\lambda}u) + \partial_x\left(2|\partial_x^CP_{\lambda}u|^2 - \partial_x^2\frac{|P_{\lambda}u|^2}{2}\right) = \Re(P_{\lambda}\mathcal{N}(u)\overline{\partial_x^CP_{\lambda}u} - P_{\lambda}u\overline{\partial_x^CP_{\lambda}\mathcal{N}(u)}).
\end{equation*}

For $d$-dimensional NLS 
\begin{equation*}
    (\sqrt{-1}\partial_t + \Delta_{\bm{C}})u = \mathcal{N}(u), \qquad u\colon [0, T] \times \mathbb{T}^d \rightarrow \mathbb{C},
\end{equation*}
where $\bm{C} = (C_1, \cdots, C_d) \in \mathbb{R}^d$, $D_i^{\bm{C}} := \partial_i + \sqrt{-1}C_i, i = 1, \cdots, d$, $\nabla_{\bm{C}} := (D_1^{\bm{C}}, \cdots, D_d^{\bm{C}})$ and $\Delta_{\bm{C}} := \nabla_{\bm{C}}^2 = (\nabla_x + \sqrt{-1}\bm{C})^2 = \Delta_x + 2\sqrt{-1}\bm{C} \cdot \nabla_x - |\bm{C}|^2$,
there hold
\begin{equation*}
    \partial_t\left(\frac{1}{2}|P_{\lambda}u|^2\right) + \nabla_x \cdot \Im(\overline{P_{\lambda}u}\nabla_{\bm{C}}P_{\lambda}u) = \Im(\overline{P_{\lambda}u}P_{\lambda}\mathcal{N}(u))
\end{equation*}
and ($i = 1, \cdots, d$)
\begin{equation*}
    \partial_t\Im(\overline{P_{\lambda}u}D_i^{\bm{C}}P_{\lambda}u) + \partial_j\left(2\Re(\overline{D_i^{\bm{C}}P_{\lambda}u}D_j^{\bm{C}}P_{\lambda}u) - \partial_i\partial_j\frac{|P_{\lambda}u|^2}{2}\right) = \{P_{\lambda}\mathcal{N}(u), P_{\lambda}u\}^{\bm{C}}_i,
\end{equation*}
where
\begin{equation*}
    \{f, g\}^{\bm{C}} := \Re(f\overline{\nabla_{\bm{C}}g} - g\overline{\nabla_{\bm{C}}f}) = (\{f, g\}^{\bm{C}}_i)_{1 \leq i \leq d}
\end{equation*}
is the momentum bracket.

Generally, we consider the SMF 
\begin{equation}\label{SSS}
    \partial_tu = J(u)D_k\partial_ku, \qquad u\colon [0, T] \times \mathbb{T}^d \rightarrow \mathcal{N},
\end{equation}
where $\mathcal{N}$ is isometrically embedded into $\mathbb{R}^N$ for sufficiently large $N$.
With the asummptions on $J$ and unit normals $n_i(u)$ in Section \ref{intro}, we now compute the energy and momentum balance laws for \eqref{SSS}.

Applying $\partial_{\ell}P_{\lambda}$ on both sides of \eqref{SSS}, we obtain
\begin{equation*}
    \partial_t\partial_{\ell}P_{\lambda}u = J(u)D_k\partial_k\partial_{\ell}P_{\lambda}u + \mathcal{R}_{\lambda},
\end{equation*}
where $\mathcal{R}_{\lambda} := [P_{\lambda}\partial_{\ell}, J(u)]D_k\partial_ku + J(u)[P_{\lambda}\partial_{\ell}, D_k]\partial_ku$ and the commutator is defined by 
\begin{equation*}
    [A, B] := AB - BA.
\end{equation*}
By taking inner product with $\partial_{\ell}P_{\lambda}u$ on both sides of the equation above we have  
\begin{equation}\label{enerb}
    \partial_t\frac{1}{2}|\nabla_xP_{\lambda}u|^2 = \langle \partial_{\ell}P_{\lambda}u, J(u)D_k\partial_k\partial_{\ell}P_{\lambda}u\rangle + \langle \partial_{\ell}P_{\lambda}u, \mathcal{R}_{\lambda}\rangle.
\end{equation}
Moreover, from the anti-symmetry of $J$, \eqref{DV} and Leibniz rule
\begin{align*}
    &\langle \partial_{\ell}P_{\lambda}u, J(u)D_k\partial_k\partial_{\ell}P_{\lambda}u\rangle \\
    = &-\langle J(u)\partial_{\ell}P_{\lambda}u, D_k\partial_k\partial_{\ell}P_{\lambda}u\rangle \\ 
    = &-\langle J(u)\partial_{\ell}P_{\lambda}u, \partial_k\partial_k\partial_{\ell}P_{\lambda}u\rangle + \langle \partial_k\partial_k\partial_{\ell}P_{\lambda}u, n_i(u)\rangle\langle J(u)\partial_{\ell}P_{\lambda}u, n_i(u)\rangle \\ 
    = &-\partial_k\langle \partial_k\partial_{\ell}P_{\lambda}u, J(u)\partial_{\ell}P_{\lambda}u\rangle + \langle \partial_k\partial_{\ell}P_{\lambda}u, \partial_kJ(u)\partial_{\ell}P_{\lambda}u\rangle \\ 
    &+ \langle \partial_k\partial_k\partial_{\ell}P_{\lambda}u, n_i(u)\rangle\langle J(u)\partial_{\ell}P_{\lambda}u, n_i(u)\rangle \\
    = &-\partial_k\langle \partial_k\partial_{\ell}P_{\lambda}u, J(u)\partial_{\ell}P_{\lambda}u\rangle + \langle \partial_k\partial_{\ell}P_{\lambda}u, (\partial_kJ(u))\partial_{\ell}P_{\lambda}u\rangle \\ 
    &+ \langle \partial_k\partial_k\partial_{\ell}P_{\lambda}u, n_i(u)\rangle\langle J(u)\partial_{\ell}P_{\lambda}u, n_i(u)\rangle,
\end{align*}
where $\partial_kJ \in \mathbb{R}^{N \times N}$. Note $\langle V, n_i(u)\rangle = 0, \forall V \in T_u\mathcal{N}$. It follows that 
\begin{align*}
    \langle \partial_k\partial_k\partial_{\ell}P_{\lambda}u, n_i(u)\rangle = \langle \partial_k\partial_k\partial_{\ell}P_{\lambda}u, n_i(u)\rangle - P_{\lambda}\partial_k\partial_k\langle \partial_{\ell}u, n_i(u)\rangle = -[P_{\lambda}\Delta_x, n_i(u)]\partial_{\ell}u,
\end{align*}
and 
\begin{align*}
    \langle J(u)\partial_{\ell}P_{\lambda}u, n_i(u)\rangle &= \langle P_{\lambda}J(u)\partial_{\ell}u, n_i(u)\rangle - \langle [P_{\lambda}, J(u)]\partial_{\ell}u, n_i(u) \rangle \\ 
    &= -[P_{\lambda}, n_i(u)]J(u)\partial_{\ell}u - \langle [P_{\lambda}, J(u)]\partial_{\ell}u, n_i(u) \rangle.
\end{align*}
Now we obtain the energy balance law
\begin{equation}\label{E1}
\begin{aligned}
    &\partial_t\frac{1}{2}|\nabla_xP_{\lambda}u|^2 + \partial_k\langle \partial_k\partial_{\ell}P_{\lambda}u, J(u)\partial_{\ell}P_{\lambda}u\rangle \\ 
    = \,&\langle \mathcal{R}_{\lambda}, \partial_{\ell}P_{\lambda}u\rangle + \langle \partial_k\partial_{\ell}P_{\lambda}u, (\partial_kJ(u))\partial_{\ell}P_{\lambda}u\rangle \\ 
    &+ \left([P_{\lambda}\Delta_x, n_i(u)]\partial_{\ell}u\right)\left([P_{\lambda}, n_i(u)]J(u)\partial_{\ell}u + \langle [P_{\lambda}, J(u)]\partial_{\ell}u, n_i(u) \rangle\right).
\end{aligned}
\end{equation}

As we have mentioned before, the crucial estimates ($L_{t, x}^2$ bilinear estimate and $L_t^6L_x^{\infty}$ estimate) in this paper are obtained by combining the balance laws of the equation and the div-curl lemma.
However, since the div-curl lemma can be only applied to the case of spatial dimension one (see Lemma \ref{divcur}), we need to further integrate $(d - 1)$ coordinates in \eqref{E1}:
\begin{align}\label{dc11}
    \partial_t\int_{\mathbb{T}^{d - 1}}\frac{1}{2}|\nabla_xP_{\lambda}u|^2 \,{\rm d}\widehat{x}_1 + \partial_1\int_{\mathbb{T}^{d - 1}}\langle \partial_1\partial_{\ell}P_{\lambda}u, J(u)\partial_{\ell}P_{\lambda}u\rangle \,{\rm d}\widehat{x}_1 = \mathcal{G}_{\lambda},
\end{align}
where $x = (x_1, \widehat{x}_1) \in \mathbb{T} \times \mathbb{T}^{d - 1}, \widehat{x}_1 = (x_2, \cdots, x_d)$ and
    \begin{align*}
            \mathcal{G}_{\lambda} &:= \int_{\mathbb{T}^{d - 1}}\langle \mathcal{R}_{\lambda}, \partial_{\ell}P_{\lambda}u\rangle \,{\rm d}\widehat{x}_1 + \int_{\mathbb{T}^{d - 1}}\langle \partial_k\partial_{\ell}P_{\lambda}u, (\partial_kJ(u))\partial_{\ell}P_{\lambda}u\rangle \,{\rm d}\widehat{x}_1 \\ 
    &\quad + \int_{\mathbb{T}^{d - 1}}\left([P_{\lambda}\Delta_x, n_i(u)]\partial_{\ell}u\right)\left([P_{\lambda}, n_i(u)]J(u)\partial_{\ell}u + \langle [P_{\lambda}, J(u)]\partial_{\ell}u, n_i(u) \rangle\right) \,{\rm d}\widehat{x}_1.
\end{align*}

For the momentum balance laws, we first have  
\begin{equation*}
    \partial_t\langle \partial_1\partial_{\ell}P_{\lambda}u, J(u)\partial_{\ell}P_{\lambda}u\rangle = \langle \partial_1\partial_{\ell}\partial_tP_{\lambda}u, J(u)\partial_{\ell}P_{\lambda}u\rangle + \langle \partial_1\partial_{\ell}P_{\lambda}u, \partial_tJ(u)\partial_{\ell}P_{\lambda}u\rangle.
\end{equation*}
On the one hand, 
    \begin{align*}
            \partial_1\partial_{\ell}\partial_tP_{\lambda}u &= \partial_1\partial_{\ell}P_{\lambda}J(u)D_k\partial_ku \\
    &= J(u)\partial_1\partial_{\ell}P_{\lambda}D_k\partial_ku + [P_{\lambda}\partial_1\partial_{\ell}, J(u)]D_k\partial_ku \\ 
    &= J(u)\partial_1\partial_{\ell}D_k\partial_kP_{\lambda}u + J(u)\partial_1\partial_{\ell}[P_{\lambda}, D_k]\partial_ku + [P_{\lambda}\partial_1\partial_{\ell}, J(u)]D_k\partial_ku \\ 
    &= J(u)\partial_1\partial_{\ell}\partial_k\partial_kP_{\lambda}u - J(u)\partial_1\partial_{\ell}(\langle \partial_k\partial_kP_{\lambda}u, n_i(u)\rangle n_i(u)) \\
    &\quad + J(u)\partial_1\partial_{\ell}[P_{\lambda}, D_k]\partial_ku + [P_{\lambda}\partial_1\partial_{\ell}, J(u)]D_k\partial_ku \\
    &= J(u)\partial_1\partial_{\ell}\partial_k\partial_kP_{\lambda}u + J(u)\partial_1\partial_{\ell}(([P_{\lambda}\partial_k, n_i(u)]\partial_ku) n_i(u)) \\
    &\quad + J(u)\partial_1\partial_{\ell}[P_{\lambda}, D_k]\partial_ku + [P_{\lambda}\partial_1\partial_{\ell}, J(u)]D_k\partial_ku.
    \end{align*}
Therefore, 
    \begin{align*}
            \langle \partial_1\partial_{\ell}\partial_tP_{\lambda}u, J(u)\partial_{\ell}P_{\lambda}u\rangle &= \langle \partial_1\partial_{\ell}\partial_k\partial_kP_{\lambda}u, \partial_{\ell}P_{\lambda}u\rangle + \mathcal{H}^1_{\lambda} \\ 
    &= \partial_k\langle \partial_1\partial_{\ell}\partial_kP_{\lambda}u, \partial_{\ell}P_{\lambda}u\rangle - \langle \partial_1\partial_{\ell}\partial_kP_{\lambda}u, \partial_{\ell}\partial_kP_{\lambda}u\rangle + \mathcal{H}_{\lambda}^1 \\
    &= \partial_k\langle \partial_1\partial_{\ell}\partial_kP_{\lambda}u, \partial_{\ell}P_{\lambda}u\rangle - \partial_1\sum_{k, \ell = 1}^d\frac{1}{2}|\partial_{\ell}\partial_kP_{\lambda}u|^2 + \mathcal{H}_{\lambda}^1,
    \end{align*}
where 
\begin{equation*}
\begin{aligned}
    \mathcal{H}_{\lambda}^1 &:= \langle \partial_1\partial_{\ell}(([P_{\lambda}\partial_k, n_i(u)]\partial_ku) n_i(u)), \partial_{\ell}P_{\lambda}u\rangle \\
    &\quad + \langle \partial_1\partial_{\ell}[P_{\lambda}, D_k]\partial_ku, \partial_{\ell}P_{\lambda}u\rangle +\langle [P_{\lambda}\partial_1\partial_{\ell}, J(u)]D_k\partial_ku, J(u)\partial_{\ell}P_{\lambda}u \rangle \\
    &= \langle [\partial_1\partial_{\ell}, n_i(u)]([P_{\lambda}\partial_k, n_i(u)]\partial_ku), \partial_{\ell}P_{\lambda}u\rangle + (\partial_1\partial_{\ell}([P_{\lambda}\partial_k, n_i(u)]\partial_ku))\langle \partial_{\ell}P_{\lambda}u, n_i(u)\rangle  \\
    &\quad + \langle \partial_1\partial_{\ell}[P_{\lambda}, D_k]\partial_ku, \partial_{\ell}P_{\lambda}u\rangle +\langle [P_{\lambda}\partial_1\partial_{\ell}, J(u)]D_k\partial_ku, J(u)\partial_{\ell}P_{\lambda}u \rangle \\
    &= \langle [\partial_1\partial_{\ell}, n_i(u)]([P_{\lambda}\partial_k, n_i(u)]\partial_ku), \partial_{\ell}P_{\lambda}u\rangle - (\partial_1\partial_{\ell}([P_{\lambda}\partial_k, n_i(u)]\partial_ku))([P_{\lambda}, n_i(u)]\partial_{\ell}u)  \\
    &\quad + \langle \partial_1\partial_{\ell}[P_{\lambda}, D_k]\partial_ku, \partial_{\ell}P_{\lambda}u\rangle +\langle [P_{\lambda}\partial_1\partial_{\ell}, J(u)]D_k\partial_ku, J(u)\partial_{\ell}P_{\lambda}u \rangle.
\end{aligned}
\end{equation*}
On the other hand, by Leibniz rule 
\begin{equation*}
    \partial_tJ(u)\partial_{\ell}P_{\lambda}u - \partial_{\ell}J(u)\partial_tP_{\lambda}u = (\partial_tJ(u))\partial_{\ell}P_{\lambda}u - (\partial_{\ell}J(u))\partial_tP_{\lambda}u,
\end{equation*}
and also note  
    \begin{align*}
            &\partial_{\ell}J(u)\partial_tP_{\lambda}u \\
    = &\partial_{\ell}J(u)P_{\lambda}J(u)D_k\partial_ku \\
    = &-\partial_{\ell}P_{\lambda}D_k\partial_ku + \partial_{\ell}J(u)[P_{\lambda}, J(u)]D_k\partial_ku \\ 
    = &-\partial_{\ell}D_k\partial_kP_{\lambda}u - \partial_{\ell}[P_{\lambda}, D_k]\partial_ku + \partial_{\ell}J(u)[P_{\lambda}, J(u)]D_k\partial_ku \\ 
    = &-\partial_{\ell}\partial_k\partial_kP_{\lambda}u + \partial_{\ell}(\langle \partial_k\partial_kP_{\lambda}u, n_i(u)\rangle n_i(u)) - \partial_{\ell}[P_{\lambda}, D_k]\partial_ku + \partial_{\ell}J(u)[P_{\lambda}, J(u)]D_k\partial_ku \\
    = &-\partial_{\ell}\partial_k\partial_kP_{\lambda}u - \partial_{\ell}(([P_{\lambda}\partial_k, n_i(u)]\partial_ku) n_i(u)) - \partial_{\ell}[P_{\lambda}, D_k]\partial_ku + \partial_{\ell}J(u)[P_{\lambda}, J(u)]D_k\partial_ku.
    \end{align*}
It follows that 
    \begin{align*}
            \langle \partial_1\partial_{\ell}P_{\lambda}u, \partial_tJ(u)\partial_{\ell}P_{\lambda}u\rangle &= -\langle\partial_1\partial_{\ell}P_{\lambda}u, \partial_{\ell}\partial_k\partial_kP_{\lambda}u\rangle + \mathcal{H}_{\lambda}^2 \\ 
    &= -\partial_k\langle \partial_1\partial_{\ell}P_{\lambda}u, \partial_{\ell}\partial_kP_{\lambda}u\rangle + \langle\partial_1\partial_{\ell}\partial_kP_{\lambda}u, \partial_{\ell}\partial_kP_{\lambda}u\rangle + \mathcal{H}_{\lambda}^2 \\
    &= -\partial_k\langle \partial_1\partial_{\ell}P_{\lambda}u, \partial_{\ell}\partial_kP_{\lambda}u\rangle + \partial_1\sum_{k, \ell = 1}^d\frac{1}{2}|\partial_{\ell}\partial_kP_{\lambda}u|^2+ \mathcal{H}_{\lambda}^2,
    \end{align*}
where 
\begin{equation*}
\begin{aligned}\label{Hl2}
    \mathcal{H}_{\lambda}^2 &:= - \langle\partial_{\ell}(([P_{\lambda}\partial_k, n_i(u)]\partial_ku) n_i(u)), \partial_1\partial_{\ell}P_{\lambda}u\rangle \\
    &\quad - \langle \partial_{\ell}[P_{\lambda}, D_k]\partial_ku, \partial_1\partial_{\ell}P_{\lambda}u\rangle + \langle \partial_{\ell}J(u)[P_{\lambda}, J(u)]D_k\partial_ku, \partial_1\partial_{\ell}P_{\lambda}u\rangle \\ 
    &\quad - \langle (\partial_tJ(u))\partial_{\ell}P_{\lambda}u - (\partial_{\ell}J(u))\partial_tP_{\lambda}u, \partial_1\partial_{\ell}P_{\lambda}u \rangle \\
    &= -\langle [\partial_{\ell}, n_i(u)]([P_{\lambda}\partial_k, n_i(u)]\partial_ku), \partial_1\partial_{\ell}P_{\lambda}u\rangle - (\partial_{\ell}([P_{\lambda}\partial_k, n_i(u)]\partial_ku))\langle \partial_1\partial_{\ell}P_{\lambda}u, n_i(u)\rangle \\
    &\quad - \langle \partial_{\ell}[P_{\lambda}, D_k]\partial_ku, \partial_1\partial_{\ell}P_{\lambda}u\rangle + \langle \partial_{\ell}J(u)[P_{\lambda}, J(u)]D_k\partial_ku, \partial_1\partial_{\ell}P_{\lambda}u\rangle \\ 
    &\quad - \langle (\partial_tJ(u))\partial_{\ell}P_{\lambda}u - (\partial_{\ell}J(u))\partial_tP_{\lambda}u, \partial_1\partial_{\ell}P_{\lambda}u \rangle \\
    &= -\langle [\partial_{\ell}, n_i(u)]([P_{\lambda}\partial_k, n_i(u)]\partial_ku), \partial_1\partial_{\ell}P_{\lambda}u\rangle + (\partial_{\ell}([P_{\lambda}\partial_k, n_i(u)]\partial_ku))([P_{\lambda}\partial_{\ell}, n_i(u)]\partial_1u) \\
    &\quad - \langle \partial_{\ell}[P_{\lambda}, D_k]\partial_ku, \partial_1\partial_{\ell}P_{\lambda}u\rangle + \langle \partial_{\ell}J(u)[P_{\lambda}, J(u)]D_k\partial_ku, \partial_1\partial_{\ell}P_{\lambda}u\rangle \\ 
    &\quad - \langle (\partial_tJ(u))\partial_{\ell}P_{\lambda}u - (\partial_{\ell}J(u))\partial_tP_{\lambda}u, \partial_1\partial_{\ell}P_{\lambda}u \rangle.
\end{aligned}
\end{equation*}

To conclude, 
\begin{equation*}
    \partial_t\langle \partial_1\partial_{\ell}P_{\lambda}u, J(u)\partial_{\ell}P_{\lambda}u\rangle = \partial_k\langle \partial_1\partial_{\ell}\partial_kP_{\lambda}u, \partial_{\ell}P_{\lambda}u\rangle - \partial_k\langle \partial_1\partial_{\ell}P_{\lambda}u, \partial_{\ell}\partial_kP_{\lambda}u\rangle + \mathcal{H}_{\lambda}^1 + \mathcal{H}_{\lambda}^2,
\end{equation*}
and hence
\begin{equation}\label{dc22}
\begin{aligned}
    \partial_t\int_{\mathbb{T}^{d - 1}}\langle \partial_1\partial_{\ell}P_{\lambda}u, J(u)\partial_{\ell}P_{\lambda}u\rangle \,{\rm d}\widehat{x}_1 - \partial_1\int_{\mathbb{T}^{d - 1}}(\langle \partial_{\ell}P_{\lambda}u, \partial_1^2\partial_{\ell}P_{\lambda}u\rangle - |\nabla_x\partial_1P_{\lambda}u|^2) \,{\rm d}\widehat{x}_1 = \mathcal{H}_{\lambda},
\end{aligned}
\end{equation}
where 
\begin{equation*}
    \mathcal{H}_{\lambda} = \int_{\mathbb{T}^{d - 1}}(\mathcal{H}_{\lambda}^1 + \mathcal{H}_{\lambda}^2) \,{\rm d}\widehat{x}_1.
\end{equation*}

\subsection{Some useful estimates}

In this section we collect some key estimates that will be frequently used in our analysis.
Although each result was originally established in Euclidean case from the literature, their generalizations to the periodic case are straightforward.

We begin by recalling the classical commutator estimates. We refer the readers to \cite[Lemma 2.97]{BCD} for the proof.

\begin{theorem}\label{Comu}
    For functions $f, g$ satisfying $\nabla f \in L^p(\mathbb{T}^d)$ and $g \in L^q(\mathbb{T}^d)$, there holds 
    \begin{equation}\label{gh}
        \Vert [P_{\lambda}, f]g\Vert_{L^r} \lesssim \lambda^{-1}\Vert \nabla f\Vert_{L^p}\Vert g \Vert_{L^q},
    \end{equation}
    where $[P_{\lambda}, f]g := P_{\lambda}(fg) - fP_{\lambda}g$ and $p, q, r \in [1, \infty]$ satisfy $1/r = 1/p + 1/q$.
\end{theorem}

As an immediate variant of Theorem \ref{Comu}, we have

\begin{lemma}\label{Comuu}
        For smooth functions $f, g$ and $h$ there holds 
    \begin{equation}\label{df}
        \sup_{a, b}\Vert h[P_{\lambda}, f^a]g^b\Vert_{L^p} \lesssim \lambda^{-1}\sup_{a, b}\Vert h|\nabla_xf^a|g^b\Vert_{L^p}, \qquad p \in [1, \infty],
    \end{equation}
    where $f^a(x) := f(x - a)$.
    \begin{proof}
        See \cite[Lemma 2.2]{LSZ}.
    \end{proof}
\end{lemma}

To formulate the estiamtes above in a unified framework, we introduce the following translation-invariant multilinear forms.

\begin{definition}[$L$-notation]
    Given scalar functions $f_i\ (i = 1, \cdots, n)$ we denote by $L(f_1, \cdots, f_n)$ any multilinear expression of the form 
    \begin{equation*}
        \begin{aligned}
            L(f_1, \cdots, f_n)(t, x) := \int K(y_1, \cdots, y_n)f_1^{y_1}(t, x) \cdots f_n^{y_n}(t, x) \,{\rm d}y_1 \cdots {\rm d}y_n,
        \end{aligned}
    \end{equation*}
    where $f_j^{y_j}(t, x) := \phi(t, x - y_j)$ and $K$ is an integral kernel (the absolute value of $K$ may change from line to line).
\end{definition}

This $L$-notation can be also extended to the case when $\phi_i$ take values as vectors or matrices by taking $K$ as an appropriate tensor.
For more details on $L$-notation we refer to \cite{T2}, and here, we only list some basic properties, which are sufficient for our use.

The $L$-notation is well-interacted with Littlewood-Paley operators. For instance, 
\begin{gather}
    \label{x}L(\nabla_xP_{\leq \lambda}f_1, f_2, \cdots, f_n) = \lambda L(f_1, f_2, \cdots, f_n), \\ 
    \label{y}\nabla_xP_{\leq \lambda}L(f_1, f_2, \cdots, f_n) = \lambda L(f_1, f_2, \cdots, f_n),
\end{gather}
and for homogenous operators, there holds
\begin{equation}
    \label{z1}L(P_{\lambda}f_1, f_2, \cdots, f_n) = \lambda^{-1}L(\nabla_xP_{\lambda}f_1, f_2, \cdots, f_n).
\end{equation}
Similarly for $P_{\lambda, e_i}$ (see Section \ref{1.3} for the definition). It also proves highly convenient to employ the $L$-notation in commutator representations.
Specifically, we have 

\begin{lemma}\label{BONYB}
    There holds 
    \begin{equation}\label{Leibiniz}
        P_{\lambda}(fg) = fP_{\lambda}g + \lambda^{-1}L(\nabla_xf, g).
    \end{equation}
    Moreover,
    \begin{equation}\label{BONY}
        [P_{\lambda}, f]g = \lambda^{-1}L(\nabla_xP_{\ll \lambda}f, P_{\lambda}g) + L(P_{\lambda}f, P_{\ll \lambda}g) + \sum_{\lambda_1 \sim \lambda_2 \gtrsim \lambda}P_{\lambda}L(P_{\lambda_1}f, P_{\lambda_2}g). 
    \end{equation}
    \begin{proof}
        The proof for \eqref{Leibiniz} (also known as Leibniz rule for $P_{\lambda}$) involves only fundamental theorem of calculus, as well as the definition of $P_{\lambda}$, and we refer to \cite[Lemma 2]{T2} for more details. 
        For \eqref{BONY}, by \eqref{Leibiniz}, \eqref{x}-\eqref{z1} and Bony decomposition \eqref{Bony} we have 
        \begin{equation*}
            \begin{aligned}
            \relax[P_{\lambda}, f]g &= \lambda^{-1}L(\nabla_xf, g) \\
            &= \lambda^{-1}L(\nabla_xP_{\ll \lambda}f, P_{\lambda}g) + \lambda^{-1}L(\nabla_xP_{\lambda}f, P_{\ll \lambda}g) + \lambda^{-1}\sum_{\lambda_1 \sim \lambda_2 \gtrsim \lambda}P_{\lambda}L(\nabla_xP_{\lambda_1}f, P_{\lambda_2}g) \\
            &= \lambda^{-1}L(\nabla_xP_{\ll \lambda}f, P_{\lambda}g) + L(P_{\lambda}f, P_{\ll \lambda}g) + \sum_{\lambda_1 \sim \lambda_2 \gtrsim \lambda}P_{\lambda}L(P_{\lambda_1}f, P_{\lambda_2}g). 
            \end{aligned}
        \end{equation*}
    \end{proof}
\end{lemma}

Finally, we list some estimates concerning the action of smooth functions on Besov spaces. The proofs can be found in \cite[Lemma 2.7]{BCD} and \cite[Lemma A.13]{KV}.

\begin{theorem}[Paralinearization]\label{para}
    Let $F$ be a smooth function vanishing at $0$, $\sigma > 0$ and $p, q \in [1, \infty]$.
    If $u \in L^{\infty}(\mathbb{T}^d) \cap B^\sigma_{p, q}(\mathbb{T}^d)$, then so does $F(u)$, and we have 
    \begin{equation*}
        \Vert F(u)\Vert_{B^\sigma_{p, q}} \lesssim_{s, F', \Vert u\Vert_{L^{\infty}}}\Vert u\Vert_{B^\sigma_{p, q}}. 
    \end{equation*}
\end{theorem}
\begin{theorem}[Nonlinear Bernstein]\label{Nbern}
    Let $G$ be a Lipschitz continuous function. Then for any $1 \leq p \leq +\infty$ and $\nabla u \in L^p(\mathbb{T}^d)$, 
    \begin{equation*}
        \Vert P_{\lambda}G(u)\Vert_{L^p} \lesssim \lambda^{-1}\Vert \nabla u\Vert_{L^p}.
    \end{equation*}
\end{theorem}

\subsection{Div-curl lemma}

The following div-curl lemma plays a crucial role in the proof for the $L_{t, x}^2$ bilinear estimate \eqref{16} and $L^6_tL_x^{\infty}$ estimate \eqref{bound2}; see Section \ref{L2L2} and \ref{Linfty}.

\begin{lemma}\label{divcur}
    Suppose that 
\begin{equation*}
\begin{cases}
        \partial_tf^{11} + \partial_xf^{12} = G^1, \\ 
    \partial_tf^{21} - \partial_xf^{22} = G^2,
\end{cases}
\end{equation*}
where $f^{ij}, G^k$ are periodic in variable $x$ with period $1$. Then 
\begin{equation}\label{div1}
    \begin{aligned}
        &\int_0^t\int_0^1(f^{11}f^{22} + f^{12}f^{21}) \,{\rm d}x{\rm d}t \\
        \leq &\,4\sup_t\Vert f^{11}(t)\Vert_{L_x^1}\Vert f^{21}(t)\Vert_{L_x^1} \\ 
        &+ \left|\int_0^t\int_0^1f^{21}\left(\int_0^1\int_y^xG^1 \,{\rm d}z{\rm d}y\right) \,{\rm d}x{\rm d}t\right| \\
        &+ \left|\int_0^t\int_0^1G^{2}\left(\int_0^1\int_y^xf^{11} \,{\rm d}z{\rm d}y\right) \,{\rm d}x{\rm d}t\right| \\ 
        &+ \left|\int_0^t\left(\int_0^1G^1 \,{\rm d}x\right)\left(\int_0^1\left(x - \frac{1}{2}\right)f^{21} \,{\rm d}x\right) \,{\rm d}t\right| \\
        &+ \left|\int_0^t\left(\int_0^1f^{11} \,{\rm d}x\right)\left(\int_0^1\left(x - \frac{1}{2}\right)G^2 \,{\rm d}x\right) \,{\rm d}t\right| \\
        &+ \int_0^t\left(\left(\int_0^1f^{11} \,{\rm d}x\right)\left(\int_0^1f^{22} \,{\rm d}x\right) + \left(\int_0^1f^{12} \,{\rm d}x\right)\left(\int_0^1f^{21} \,{\rm d}x\right)\right) \,{\rm d}t
    \end{aligned}
\end{equation}
provided the right-hand side is bounded.
In particular, 
\begin{equation}\label{div2}
    \begin{aligned}
                &\int_0^t\int_{\mathbb{T}}(f^{11}f^{22} + f^{12}f^{21}) \,{\rm d}x{\rm d}t \\
        \lesssim &\left(\Vert f^{11}(0)\Vert_{L_x^1} + \Vert f^{11}\Vert_{L_t^{\infty}L_x^1} + \Vert G^1\Vert_{L_{t, x}^1}\right)\left(\Vert f^{21}(0)\Vert_{L_x^1} + \Vert f^{21}\Vert_{L_t^{\infty}L_x^1} + \Vert G^2\Vert_{L_{t, x}^1}\right) \\ 
        &+ \int_0^t\left(\left(\int_{\mathbb{T}}f^{11} \,{\rm d}x\right)\left(\int_{\mathbb{T}}f^{22} \,{\rm d}x\right) + \left(\int_{\mathbb{T}}f^{12} \,{\rm d}x\right)\left(\int_{\mathbb{T}}f^{21} \,{\rm d}x\right)\right) \,{\rm d}t,
    \end{aligned}
\end{equation}
where 
\begin{equation*}
    \int_{\mathbb{T}^d} f \,{\rm d}x = \int_{[0, 1]^d}f \,{\rm d}x = \int_{[x_0, x_0 + 1]^d}f \,{\rm d}x, \qquad x_0 \in \mathbb{R}.
\end{equation*}
    \begin{proof}
        The calculations in \cite{WZ} imply the identity 
        \begin{equation*}
            \begin{aligned}
                &\int_0^t\int_0^1(f^{11}f^{22} + f^{12}f^{21}) \,{\rm d}x{\rm d}t \\ 
                = &\left.\left(-\int_0^1f^{21}\left(\int_0^1\int_y^xf^{11} \,{\rm d}z{\rm d}y\right) \,{\rm d}x + \left(\int_0^1f^{11} \,{\rm d}x\right)\left(\int_0^1\left(x - \frac{1}{2}\right)f^{21} \,{\rm d}x\right)\right)\right|_{0}^t \\ 
                &+ \int_0^t\int_0^1f^{21}\left(\int_0^1\int_y^xG^1 \,{\rm d}z{\rm d}y\right) \,{\rm d}x{\rm d}t \\
                &+ \int_0^t\int_0^1G^{2}\left(\int_0^1\int_y^xf^{11} \,{\rm d}z{\rm d}y\right) \,{\rm d}x{\rm d}t \\ 
                &- \int_0^t\left(\int_0^1G^1 \,{\rm d}x\right)\left(\int_0^1\left(x - \frac{1}{2}\right)f^{21} \,{\rm d}x\right) \,{\rm d}t \\
                &- \int_0^t\left(\int_0^1f^{11} \,{\rm d}x\right)\left(\int_0^1\left(x - \frac{1}{2}\right)G^2 \,{\rm d}x\right) \,{\rm d}t \\
                &+ \int_0^t\left(\left(\int_0^1f^{11} \,{\rm d}x\right)\left(\int_0^1f^{22} \,{\rm d}x\right) + \left(\int_0^1f^{12} \,{\rm d}x\right)\left(\int_0^1f^{21} \,{\rm d}x\right)\right) \,{\rm d}t,
            \end{aligned}
        \end{equation*}
        which yields \eqref{div1}, and it is clear that \eqref{div2} follows from \eqref{div1} and H\"older inequality.
    \end{proof}
\end{lemma}

For a Euclidean version of Lemma \ref{divcur} we refer to \cite[Lemma 2.1]{WZ1}. It is noteworthy that this div-curl lemma and the bilinear estimate method derived from it, have broad applicability in low-regularity problem for dispersive and wave equations.
We refer interested readers to \cite{HTZ,ZZ,WZ2,LSZ,HLPZ} for related works.

\section{Proof of Theorem \ref{LWP2}}\label{3A}

\subsection{Derivation of mSMF}\label{MSMF}

Recall the definition of $(\psi_{\alpha}, A_{\beta})$ in Section \ref{1.3}.
In this section we derive the equation satisfied by $\bm{\psi} = (\psi_1, \cdots, \psi_d)$, namely, the mSMF.

Combining \eqref{C} with \eqref{Coulomb}, we obtain 
\begin{equation}\label{Ak}
    \Delta_x\bm{A} = -\partial_{\ell}(\Im(\bm{\psi}\overline{\psi_{\ell}})).
\end{equation}
Moreover, by \eqref{LL} 
\begin{equation*}
    \psi_0 = \langle u \times \Delta_x u, e_1\rangle + \sqrt{-1}\langle u \times \Delta_x u, e_2\rangle,
\end{equation*}
and using \eqref{A}-\eqref{B}, we have $(\alpha = 0, 1, \cdots, d$)
\begin{equation*}
\begin{aligned}
    \partial_{\alpha}^2u = (\partial_{\alpha}\Re(\psi_{\alpha}) - A_{\alpha}\Im(\psi_{\alpha}))e_1 + (\partial_{\alpha}\Im(\psi_{\alpha}) + A_{\alpha}\Re(\psi_{\alpha}))e_2 - |\psi_{\alpha}|^2u.
\end{aligned}
\end{equation*}
Therefore,
\begin{align}\label{D}
    \psi_0 = -(\partial_{\ell}\Im(\psi_{\ell}) + A_{\ell}\Re(\psi_{\ell})) + \sqrt{-1}(\partial_{\ell}\Re(\psi_{\ell}) - A_{\ell}\Im(\psi_{\ell})) = \sqrt{-1}D_{\ell}^{\bm{A}}\psi_{\ell},
\end{align}
and it follows that 
\begin{equation}\label{abss}
\begin{aligned}
    D_0^{\bm{A}}\psi_k = D_k^{\bm{A}}\psi_0 = \sqrt{-1}D_k^{\bm{A}}D_{\ell}^{\bm{A}}\psi_{\ell} &= \sqrt{-1}D_{\ell}^{\bm{A}}D_k^{\bm{A}}\psi_{\ell} - \Im(\psi_k\overline{\psi_{\ell}})\psi_{\ell} \\ 
    &= \sqrt{-1}D_{\ell}^{\bm{A}}D_{\ell}^{\bm{A}}\psi_k - \Im(\psi_k\overline{\psi_{\ell}})\psi_{\ell}.
\end{aligned}
\end{equation}
Here we use the identity 
\begin{equation*}
    [D_{\alpha}^{\bm{A}}, D_{\beta}^{\bm{A}}] = \sqrt{-1}\Im(\psi_{\alpha}\overline{\psi_{\beta}}), \qquad \alpha, \beta = 0, 1, \cdots, d
\end{equation*}
and \eqref{Coulomb}. 

Note \eqref{abss} is equivalent to 
\begin{equation}\label{psii}
\begin{aligned}
    (\sqrt{-1}\partial_t + \Delta_x)\psi_k = -2\sqrt{-1}A_{\ell}\partial_{\ell}\psi_k + \left(A_0 + \sum_{\ell = 1}^dA_{\ell}^2\right)\psi_k - \sqrt{-1}\psi_{\ell}\Im(\psi_k\overline{\psi_{\ell}}),
\end{aligned}
\end{equation}
where $k = 1, \cdots, d$. Namely,
\begin{equation*}
    (\sqrt{-1}\partial_t + \Delta_x)\psi_k = -2\sqrt{-1}\bm{A} \cdot \nabla_x\psi_k + (A_0 + |\bm{A}|^2)\psi_k - \sqrt{-1}\bm{\psi} \cdot \Im(\psi_k\overline{\bm{\psi}}).
\end{equation*}

For the term $A_0$, by \eqref{C} and \eqref{Coulomb}
\begin{equation*}
    \Delta_x A_0 = \partial_{\ell}(\partial_0A_{\ell} + \Im(\psi_{\ell}\overline{\psi_0})) = \partial_{\ell}\Im(\psi_{\ell}\overline{\psi_0}).
\end{equation*}
It is easy to check that  
\begin{equation}\label{in}
    \overline{\psi_{\ell}}D_m^{\bm{A}}\psi_m = \partial_m(\overline{\psi_{\ell}}\psi_m) - \psi_m\overline{D_m^{\bm{A}}\psi_{\ell}}, \qquad m, \ell = 1, \cdots, d.
\end{equation}
Using \eqref{in}, from \eqref{D} and \eqref{C1} we have 
\begin{equation*}
\begin{aligned}
    \Im(\psi_{\ell}\overline{\psi_0}) = -\Re(\overline{\psi_{\ell}}D_m^{\bm{A}}\psi_m) &= -\partial_m\Re(\overline{\psi_{\ell}}\psi_m) + \Re(\psi_m\overline{D_m^{\bm{A}}\psi_{\ell}}) \\ 
    &= -\partial_m\Re(\overline{\psi_{\ell}}\psi_m) + \Re(\psi_m\overline{D_{\ell}^{\bm{A}}\psi_m}) \\ 
    &= -\partial_m\Re(\overline{\psi_{\ell}}\psi_m) + \frac{1}{2}\partial_{\ell}\sum_{m = 1}^d|\psi_m|^2 \\ 
    &= -\nabla_x \cdot \Re(\overline{\psi_{\ell}}\bm{\psi}) + \frac{1}{2}\partial_{\ell}|\bm{\psi}|^2.
\end{aligned}
\end{equation*}
Therefore, 
\begin{equation}\label{A0}
    \Delta_x A_0 = -\partial_m\partial_{\ell}\Re(\overline{\psi_{\ell}}\psi_m) + \frac{1}{2}\Delta_x|\bm{\psi}|^2.
\end{equation}

It remains to express $\bm{A}$ and $A_0$ in terms of $\psi_k$ to obtain a closed equation for $\psi_k$. Firstly, since 
\begin{equation*}
    \frac{1}{|\mathbb{T}^d|}\int_{\mathbb{T}^d}A_0(t, x) \,{\rm d}x = 0,
\end{equation*}
it follows from \eqref{A0} that 
\begin{equation*}
    A_0(t, x) = \mathcal{R}_m\mathcal{R}_{\ell}\Re(\overline{\psi_{\ell}}\psi_m) + \frac{1}{2}|\bm{\psi}|^2,
\end{equation*}
where $\mathcal{R}_m$ is the Riesz-type operator defined by Fourier multiplier $\xi \mapsto \sqrt{-1}\xi_m/|\xi|$.

For $\bm{A}$,   
\begin{equation*}
\begin{aligned}
    \bm{A}(t, x) &= \left(\bm{A}(t, x) - \frac{1}{|\mathbb{T}^d|}\int_{\mathbb{T}^d}\bm{A}(t, x) \,{\rm d}x\right) + \frac{1}{|\mathbb{T}^d|}\int_{\mathbb{T}^d}\bm{A}(t, x) \,{\rm d}x \\
    &=: \widetilde{\bm{A}}(t, x) + \bm{C}(t).
\end{aligned}
\end{equation*}
Note 
\begin{equation*}
    \frac{1}{|\mathbb{T}^d|}\int_{\mathbb{T}^d}\widetilde{\bm{A}}(t, x) \,{\rm d}x = 0
\end{equation*}
and $\Delta_x\widetilde{\bm{A}} = \Delta_x\bm{A}$. It follows from \eqref{Ak}
\begin{equation*}
    \widetilde{\bm{A}}(t, x) = \nabla^{-1}_x\sum_{\ell = 1}^d\mathcal{R}_{\ell}(\Im(\bm{\psi}\overline{\psi_{\ell}})),
\end{equation*}
where $\nabla_x^{-1}$ is a Fourier multiplier with symbol $|\xi|^{-1}$.
For $\bm{C} = (C_1, \cdots, C_d) \in \mathbb{R}^d$, by \eqref{C} we have 
\begin{equation*}
\begin{aligned}
    |\mathbb{T}^d|\frac{{\rm d}C_k(t)}{{\rm d}t} = \int_{\mathbb{T}^d}\partial_0A_k \,{\rm d}x &= \int_{\mathbb{T}^d}(\partial_kA_0 + \Im(\psi_0\overline{\psi_k})) \,{\rm d}x \\ 
    &= \int_{\mathbb{T}^d}\Im(\psi_0\overline{\psi_k}) \,{\rm d}x \\ 
    &= \int_{\mathbb{T}^d}\left(-\nabla_x \cdot \Re(\overline{\psi_k}\bm{\psi}) + \frac{1}{2}\partial_k|\bm{\psi}|^2\right) \,{\rm d}x \\ 
    &= 0,
\end{aligned}
\end{equation*}
and hence 
\begin{equation*}
    \bm{C}(t) = \bm{C}(0) = \frac{1}{|\mathbb{T}^d|}\int_{\mathbb{T}^d}\bm{A}(0, x) \,{\rm d}x,
\end{equation*}
where the right-hand side is a constant that depends only on initial data. 

Finally, we rewrite the equation \eqref{psii} as
\begin{equation}\label{psiii}
\begin{aligned}
    &(\sqrt{-1}\partial_t + \Delta_x + 2\sqrt{-1}\bm{C} \cdot \nabla_x - |\bm{C}|^2)\psi_k \\
    = &-2\sqrt{-1}\widetilde{\bm{A}} \cdot \nabla_x\psi_k + (A_0 + |\widetilde{\bm{A}}|^2 + 2\bm{C} \cdot \widetilde{\bm{A}})\psi_k - \sqrt{-1}\bm{\psi} \cdot \Im(\psi_k\overline{\bm{\psi}}).
\end{aligned}
\end{equation}
We conclude that 

\begin{proposition}\label{prop12}
    Suppose $u \in C^{\infty}([0, T] \times \mathbb{T}^d; \mathbb{S}^2)$ is a smooth solution to \eqref{LL} with initial data $\phi \in C^{\infty}(\mathbb{T}^d; \mathbb{S}^2)$, $\{e_1, e_2\}$ is a smooth frame of $u^{-1}T\mathbb{S}^2$ and the connection coefficients 
    \begin{equation*}
        A_k := \langle \partial_ke_1, e_2\rangle, \qquad k = 1, \cdots, d
    \end{equation*}
    satisfy the Coulomb gauge condition 
    \begin{equation*}
        \nabla_x\cdot \bm{A} \equiv 0, \qquad \bm{A} = (A_1, \cdots, A_d),
    \end{equation*}
    and also 
    \begin{equation*}
        \frac{1}{|\mathbb{T}^d|}\int_{\mathbb{T}^d}A_0(t, x) \,{\rm d}x = 0.
    \end{equation*}
    For $k = 1, \cdots, d$ we define 
    \begin{equation*}
        \psi_k := \langle \partial_ku, e_1\rangle + \sqrt{-1}\langle \partial_ku, e_2\rangle, \qquad \bm{\psi} = (\psi_1, \cdots, \psi_d).
    \end{equation*}
    Then $(\bm{A}, \bm{\psi})$ satisfies
    \begin{equation*}
        \begin{cases}
            D_k^{\bm{A}}\psi_{\ell} = D_{\ell}^{\bm{A}}\psi_k, \\ 
            \displaystyle \bm{A} = \nabla^{-1}_x\mathcal{R}_{\ell}(\Im(\overline{\psi_{\ell}}\bm{\psi})) + \frac{1}{|\mathbb{T}^d|}\int_{\mathbb{T}^d}\bm{A}(0) \,{\rm d}x =: \widetilde{\bm{A}} + \bm{C},
        \end{cases}
    \end{equation*}
    where $D_k^{\bm{A}} = \partial_k + \sqrt{-1}A_k, \bm{\widetilde{A}} = (\widetilde{A}_1, \cdots, \widetilde{A}_d), \bm{C} = (C_1, \cdots, C_d)$.
    Moreover, let 
    \begin{equation*}
        \begin{cases}
            \nabla_{\bm{C}} := \nabla_x + \sqrt{-1}\bm{C} = (D_1^{\bm{C}}, \cdots, D_d^{\bm{C}}), \\
            \Delta_{\bm{C}} := \nabla_{\bm{C}}^2 = \Delta_x + 2\sqrt{-1}\bm{C} \cdot \nabla_x - |\bm{C}|^2.
        \end{cases}
    \end{equation*}
    Then
\begin{equation}\label{psi-e}
\begin{aligned}
    (\sqrt{-1}\partial_t + \Delta_{\bm{C}})\psi_k = \mathcal{N}(\psi_k), \qquad k = 1, \cdots, d
\end{aligned}
\end{equation}
where 
\begin{align*}
    \mathcal{N}(\psi_k) = -2\sqrt{-1}\bm{\widetilde{A}} \cdot \nabla_x\psi_k + (A_0 + |\bm{\widetilde{A}}|^2 + 2\bm{C} \cdot \bm{\widetilde{A}})\psi_k - \sqrt{-1}\bm{\psi} \cdot \Im(\psi_k\overline{\bm{\psi}}),
\end{align*} 
and
\begin{equation}\label{A00}
    A_0 = \mathcal{R}_m\mathcal{R}_{\ell}\Re(\overline{\psi_{\ell}}\psi_m) + \frac{1}{2}|\bm{\psi}|^2.
\end{equation}
\end{proposition}

\subsection{Energy estimates}\label{EnergyEs}

From the analysis in Section \ref{Balance}, for a smooth solution to \eqref{psi-e} we have  
\begin{equation*}
    \partial_t\frac{|P_{\lambda}\psi_k|^2}{2} + \nabla_x \cdot \Im(\overline{P_{\lambda}\psi_k}\nabla_{\bm{C}} P_{\lambda}\psi_k) = \Im(\overline{P_{\lambda}\psi_k}P_{\lambda}\mathcal{N}(\psi_k)),
\end{equation*}
and hence 
\begin{equation*}
    \begin{aligned}
        \int_{\mathbb{T}^d}\frac{|P_{\lambda}\psi_k(t)|^2}{2} \,{\rm d}x = \int_{\mathbb{T}^d}\frac{|P_{\lambda}\psi_k(0)|^2}{2} \,{\rm d}x + \int_0^t\int_{\mathbb{T}^d}\Im(\overline{P_{\lambda}\psi_k}P_{\lambda}\mathcal{N}(\psi_k)) \,{\rm d}x{\rm d}t.
    \end{aligned}
\end{equation*} 
For the first term on the right-hand side, 
\begin{equation*}
    \int_{\mathbb{T}^d}\frac{|P_{\lambda}\psi_k(0)|^2}{2} \,{\rm d}x = \Vert P_{\lambda}\psi_k(0)\Vert_{L_x^2}^2 \lesssim \lambda^{-2\sigma + 2}a_{\lambda}^2.
\end{equation*}
For the second term, we have
\begin{equation*}
    \begin{aligned}
        \int_0^t\int_{\mathbb{T}^d} \Im(\overline{P_{\lambda}\psi_k}P_{\lambda}\mathcal{N}(\psi_k)) \,{\rm d}x{\rm d}t = \sum_{i = 1}^3\int_0^t\int_{\mathbb{T}^d}\Im(\overline{P_{\lambda}\psi_k}P_{\lambda}\mathcal{N}_i(\psi_k)) \,{\rm d}x{\rm d}t =: \mathcal{I}_1 + \mathcal{I}_2 + \mathcal{I}_3
    \end{aligned}
\end{equation*}
where 
\begin{equation*}
    \begin{aligned}
        &\mathcal{N}_1(\psi_k) = -2\sqrt{-1}\bm{\widetilde{A}} \cdot \nabla_x\psi_k \\
        &\mathcal{N}_2(\psi_k) = (A_0 + |\bm{\widetilde{A}}|^2 + 2\bm{C} \cdot \widetilde{\bm{A}})\psi_k, \\ 
        &\mathcal{N}_3(\psi_k) = -\sqrt{-1}\bm{\psi} \cdot \Im(\psi_k\overline{\bm{\psi}}).
    \end{aligned}
\end{equation*}

The following sections are devoted to proving 
\begin{equation*}
    \mathcal{I}_i \lesssim \lambda^{-2\sigma + 2}a_{\lambda}^2, \qquad i = 1, 2, 3,
\end{equation*}
where $t \in [0, T^*]$ and $T^* = T^*(\Vert \bm{\psi}(0)\Vert_{H_x^{\sigma - 1}}) > 0$.

\subsubsection{Estimates for \texorpdfstring{$\mathcal{I}_3$}{}}\label{3.1.1}

By symmetry we have
\begin{equation*}
    \begin{aligned}
        \mathcal{I}_3 \lesssim \left(\sum\nolimits_{3a} + \sum\nolimits_{3b} + \sum\nolimits_{3c}\right)\Vert P_{\lambda}\psi_kP_{\lambda_1}\psi_kP_{\lambda_2}\bm{\psi} \cdot P_{\lambda_3}\bm{\psi}\Vert_{L_{t, x}^1} =: \mathcal{I}_{3a} + \mathcal{I}_{3b} + \mathcal{I}_{3c},
    \end{aligned}
\end{equation*}
where 
\begin{equation*}
    \begin{aligned}
        \sum\nolimits_{3a} &= \sum_{\lambda_1 \sim \lambda \gg \lambda_2, \lambda_3} + \sum_{\lambda_3 \sim \lambda \gg \lambda_1, \lambda_2}, \\
        \sum\nolimits_{3b} &= \sum_{\lambda_2 \sim \lambda_3 \gtrsim \lambda \gg \lambda_1} + \sum_{\lambda_1 \sim \lambda_3 \gtrsim \lambda \gg \lambda_2},\\ 
        \sum\nolimits_{3c} &= \sum_{\lambda_2 \sim \lambda_3 \gtrsim \lambda_1 \gtrsim \lambda}.
    \end{aligned}
\end{equation*}

We first estimate $\mathcal{I}_{3a}$. From H\"older inequality and Bernstein inequality
\begin{equation*}
    \begin{aligned}
        \mathcal{I}_{3a} &\lesssim \sum_{i, j = 1}^d\sum_{\lambda \gg \lambda_1, \lambda_2}\Vert P_{\lambda, e_i}\bm{\psi} \cdot P_{\lambda_1}\bm{\psi}\Vert_{L_{t, x}^2}\Vert P_{\lambda, e_j}\bm{\psi} \cdot P_{\lambda_2}\bm{\psi}\Vert_{L_{t, x}^2} \\ 
        &\lesssim \sum_{i, j = 1}^d\sum_{\lambda \gg \lambda_1, \lambda_2}\lambda_1^{\frac{d - 1}{2}}\lambda_2^{\frac{d - 1}{2}}\left\Vert \Vert P_{\lambda, e_i}\bm{\psi}\Vert_{L_{\widehat{x}_i}^2} \Vert P_{\lambda_1}\bm{\psi}\Vert_{L_{\widehat{x}_i}^2}\right\Vert_{L_{t, x_i}^2}\left\Vert \Vert P_{\lambda, e_j}\bm{\psi}\Vert_{L_{\widehat{x}_j}^2} \Vert P_{\lambda_2}\bm{\psi}\Vert_{L_{\widehat{x}_j}^2}\right\Vert_{L_{t, x_j}^2}.
    \end{aligned}
\end{equation*}
For the control of the bilinear term
\begin{equation*}
    \mathcal{B} := \left\Vert \Vert P_{\lambda, e_i}\bm{\psi}\Vert_{L_{\widehat{x}_i}^2} \Vert P_{\lambda_1}\bm{\psi}\Vert_{L_{\widehat{x}_i}^2}\right\Vert_{L_{t, x_i}^2}
\end{equation*}
in the summation, we split it into two parts (recall that $\sigma = (d + 1)/2 + \varepsilon$ and $0 < \varepsilon \leq 1/2$):
\begin{equation*}
    \begin{aligned}
        \mathcal{B} &= \left\Vert \Vert P_{\lambda, e_i}\bm{\psi}\Vert_{L_{\widehat{x}_i}^2} \Vert P_{\lambda_1}\bm{\psi}\Vert_{L_{\widehat{x}_i}^2}\right\Vert_{L_{t, x_i}^2}^{\varepsilon}\left\Vert \Vert P_{\lambda, e_i}\bm{\psi}\Vert_{L_{\widehat{x}_i}^2} \Vert P_{\lambda_1}\bm{\psi}\Vert_{L_{\widehat{x}_i}^2}\right\Vert_{L_{t, x_i}^2}^{1 - \varepsilon} \\ 
    &=: \mathcal{B}_1^{\varepsilon}\mathcal{B}_2^{1 - \varepsilon}.
    \end{aligned}
\end{equation*}
On the one hand, from H\"older inequality, Bernstein inequality and \eqref{11}
\begin{equation*}
    \begin{aligned}
    \mathcal{B}_1^{\varepsilon} &\leq \Vert P_{\lambda, e_1}\bm{\psi}\Vert_{L_{t, x}^2}^{\varepsilon}\Vert P_{\lambda_1}\bm{\psi}\Vert_{L_t^{\infty}L_{x_i}^{\infty}L_{\widehat{x}_i}^2}^{\varepsilon}\\
     &\lesssim t^{\frac{\varepsilon}{2}}\lambda_1^{\frac{\varepsilon}{2}}\Vert P_{\lambda, e_1}\bm{\psi}\Vert_{L_t^{\infty}L_x^2}^{\varepsilon}\Vert P_{\lambda_1}\bm{\psi}\Vert_{L_t^{\infty}L_x^2}^{\varepsilon} \\
    &\lesssim t^{\frac{\varepsilon}{2} - \frac{\varepsilon^2}{4}}\lambda_1^{\frac{\varepsilon}{2}}(\lambda^{-\sigma + 1}\lambda_1^{-\sigma + 1}a_{\lambda}a_{\lambda_1})^{\varepsilon},
    \end{aligned}
\end{equation*}
On the other hand, by \eqref{13}
\begin{equation*}
    \mathcal{B}_2^{1 - \varepsilon} \lesssim t^{-(1 - \varepsilon)\frac{\varepsilon}{8}}(\lambda^{-\sigma + 1}\lambda_1^{-\sigma + 1}a_{\lambda}a_{\lambda_1})^{1 - \varepsilon},
\end{equation*}
Therefore, 
\begin{equation}\label{L2}
    \begin{aligned}
        \mathcal{B} \lesssim t^{\frac{3\varepsilon - \varepsilon^2}{8}}\lambda^{-\sigma + 1}\lambda_1^{-\sigma + 1 + \frac{\varepsilon}{2}}a_{\lambda}a_{\lambda_1}.
    \end{aligned}
\end{equation}
Similarly, 
\begin{equation*}
    \begin{aligned}
        \left\Vert \Vert P_{\lambda, e_i}\bm{\psi}\Vert_{L_{\widehat{x}_i}^2} \Vert P_{\lambda_2}\bm{\psi}\Vert_{L_{\widehat{x}_i}^2}\right\Vert_{L_{t, x_i}^2} \lesssim t^{\frac{3\varepsilon - \varepsilon^2}{8}}\lambda^{-\sigma + 1}\lambda_2^{-\sigma + 1 + \frac{\varepsilon}{2}}a_{\lambda}a_{\lambda_2}.
    \end{aligned}
\end{equation*}
Now we obtain 
\begin{equation*}
    \begin{aligned}
        \mathcal{I}_{3a} \lesssim t^{\frac{3\varepsilon - \varepsilon^2}{4}}\lambda^{-2\sigma + 2}a_{\lambda}^2\sum_{\lambda \gg \lambda_1, \lambda_2}\lambda_1^{-\sigma + \frac{d + 1}{2} + \frac{\varepsilon}{2}}\lambda_2^{-\sigma + \frac{d + 1}{2} + \frac{\varepsilon}{2}}a_{\lambda_1}a_{\lambda_2}.
    \end{aligned}
\end{equation*}
Note that
\begin{equation*}
    \frac{3\varepsilon - \varepsilon^2}{4} > 0, \qquad -\sigma + \frac{d + 1}{2} + \frac{\varepsilon}{2} < 0
\end{equation*}
since $\sigma = (d + 1)/2 + \varepsilon, 0 < \varepsilon \leq 1/2$. It follows from H\"older inequality that 
\begin{equation*}
    \mathcal{I}_{3a} \lesssim t^{\frac{3\varepsilon - \varepsilon^2}{4}}\left(\sum_{\lambda_1}a_{\lambda_1}^2\right)^{\frac{1}{2}}\left(\sum_{\lambda_2}a_{\lambda_2}^2\right)^{\frac{1}{2}}\lambda^{-2\sigma + 2}a_{\lambda}^2 \lesssim t^{\frac{3\varepsilon - \varepsilon^2}{4}}\Vert \bm{\psi}(0)\Vert_{H_x^{\sigma - 1}}^2\lambda^{-2\sigma + 2}a_{\lambda}^2.
\end{equation*}
Thus, we can choose $T^* = T^*(\Vert \bm{\psi}(0)\Vert_{H_x^{\sigma - 1}}) > 0$ such that 
\begin{equation*}
    \mathcal{I}_{3a} \lesssim \lambda^{-2\sigma + 2}a_{\lambda}^2, \qquad t \in [0, T^*].
\end{equation*} 

Moving to $\mathcal{I}_{3b}$, we further split the summation into two parts:
\begin{align*}
    \mathcal{I}_{3b} \lesssim \left(\sum_{\lambda_2 \sim \lambda_3 \gg \lambda \gg \lambda_1} + \sum_{\lambda_2 \sim \lambda_3 \sim \lambda \gg \lambda_1}\right)\Vert P_{\lambda}\psi_kP_{\lambda_1}\psi_kP_{\lambda_2}\bm{\psi} \cdot P_{\lambda_3}\bm{\psi}\Vert_{L_{t, x}^1} =: \mathcal{I}_{3b}^{\alpha} + \mathcal{I}_{3b}^{\beta}
\end{align*}
The term $\mathcal{I}_{3b}^{\alpha}$ can be treated exactly the same as $\mathcal{I}_{3a}$, using $L_{t, x}^2$ bilnear estimates \eqref{13} twicely:
\begin{equation*}
    \begin{aligned}
        \mathcal{I}_{3b}^{\alpha} &\lesssim \sum_{i, j = 1}^d\sum_{\lambda_2 \gg \lambda \gg \lambda_1}\lambda^{\frac{d - 1}{2}}\lambda_1^{\frac{d - 1}{2}}\left\Vert \Vert P_{\lambda_2, e_i}\bm{\psi}\Vert_{L_{\widehat{x}_i}^2} \Vert P_{\lambda}\bm{\psi}\Vert_{L_{\widehat{x}_i}^2}\right\Vert_{L_{t, x_i}^2}\left\Vert \Vert P_{\lambda_2, e_j}\bm{\psi}\Vert_{L_{\widehat{x}_j}^2} \Vert P_{\lambda_1}\bm{\psi}\Vert_{L_{\widehat{x}_j}^2}\right\Vert_{L_{t, x_j}^2} \\ 
        &\lesssim t^{\frac{3\varepsilon - \varepsilon^2}{4}}\sum_{\lambda_2 \gg \lambda \gg \lambda_1}\lambda_2^{-\sigma + 1}a_{\lambda_2}\lambda^{-\sigma + 1 + \frac{d - 1}{2} + \frac{\varepsilon}{2}}a_{\lambda}\lambda_2^{-\sigma + 1}a_{\lambda_2}\lambda_1^{-\sigma + 1 + \frac{d - 1}{2} + \frac{\varepsilon}{2}}a_{\lambda_1} \\ 
        &\lesssim  t^{\frac{3\varepsilon - \varepsilon^2}{4}}\lambda^{-\sigma + 1}a_{\lambda}\sum_{\lambda_2 \gg \lambda \gg \lambda_1} \lambda_2^{-\sigma + 1 + \frac{d - 1}{2} + \frac{\varepsilon}{2}}a_{\lambda_2}\lambda_2^{-\sigma + 1}a_{\lambda_2}\lambda_1^{-\sigma + \frac{d + 1}{2} + \frac{\varepsilon}{2}}a_{\lambda_1} \\ 
        &\lesssim t^{\frac{3\varepsilon - \varepsilon^2}{4}}\lambda^{-2\sigma + 2}a_{\lambda}^2\sum_{\lambda_2 \gg \lambda \gg \lambda_1} \lambda_2^{-\sigma + \frac{d + 1}{2} + \frac{\varepsilon}{2}}a_{\lambda_2}\left(\frac{\lambda_2}{\lambda}\right)^{-\sigma + 1 + \delta}\lambda_1^{-\sigma + \frac{d + 1}{2} + \frac{\varepsilon}{2}}a_{\lambda_1} \\ 
        &\lesssim \lambda^{-2\sigma + 2}a_{\lambda}^2.
    \end{aligned}
\end{equation*}
where we use the property \eqref{fre1} of frequency envelope $\{a_{\lambda}\}$.
Note also that we can choose $\delta = \delta(\sigma, d) > 0$ sufficiently small such that $-\sigma + 1 + \delta < 0$,
For $\mathcal{I}_{3b}^{\beta}$, by H\"older inequality and Bernstein inequality
\begin{equation*}
    \mathcal{I}_{3b}^{\beta} \lesssim \sum_{i = 1}^d\sum_{\lambda \gg \lambda_1}\lambda_1^{\frac{d - 1}{2}}\Vert P_{\lambda}\bm{\psi}\Vert_{L_{t, x}^4}^2\left\Vert \Vert P_{\lambda, e_i}\bm{\psi}\Vert_{L_{\widehat{x}_i}^2} \Vert P_{\lambda_1}\bm{\psi}\Vert_{L_{\widehat{x}_i}^2}\right\Vert_{L_{t, x_i}^2}
\end{equation*}
The bilinear term 
\begin{equation*}
    \widetilde{\mathcal{B}} := \left\Vert \Vert P_{\lambda, e_i}\bm{\psi}\Vert_{L_{\widehat{x}_i}^2} \Vert P_{\lambda_1}\bm{\psi}\Vert_{L_{\widehat{x}_i}^2}\right\Vert_{L_{t, x_i}^2}
\end{equation*}
can be controlled in the same way as $\mathcal{B}$, yielding the bound (see \eqref{L2})
\begin{equation*}
    \widetilde{\mathcal{B}} \lesssim t^{\frac{3\varepsilon - \varepsilon^2}{8}}\lambda^{-\sigma + 1}\lambda_1^{-\sigma + 1 + \frac{\varepsilon}{2}}a_{\lambda}a_{\lambda_1}.
\end{equation*}
For the estimates of $\Vert P_{\lambda}\bm{\psi}\Vert_{L_{t, x}^4}$, by \eqref{13}-\eqref{20} and Bernstein inequality 
\begin{equation}\label{L4L4}
    \begin{aligned}
        \Vert P_{\lambda}\bm{\psi}\Vert_{L_{t, x}^4} &= \Vert P_{\lambda}\bm{\psi}\Vert_{L_{t, x}^4}^{\varepsilon}\Vert P_{\lambda}\bm{\psi}\Vert_{L_{t, x}^4}^{1 - \varepsilon} \\
        &\lesssim (t^{\frac{1}{4}}\lambda^{\frac{d}{4}}\Vert P_{\lambda}\bm{\psi}\Vert_{L_t^{\infty}L_x^2})^{\varepsilon}(\lambda^{\frac{d - 3}{4}}\Vert P_{\lambda}\bm{\psi}\Vert_{L_{t, y}^4L_z^2})^{1 - \varepsilon} \\
        &= (t^{\frac{1}{4} - \frac{\varepsilon}{8}}\lambda^{-\sigma + \frac{d}{4} + 1}a_{\lambda})^{\varepsilon}(t^{-\frac{\varepsilon}{8}}\lambda^{-\sigma + \frac{d + 3}{4}}a_{\lambda})^{1 - \varepsilon} \\
        &= t^{\frac{\varepsilon}{8}}\lambda^{-\sigma + \frac{d}{4} + \frac{\varepsilon + 3}{4}}a_{\lambda},
    \end{aligned}
\end{equation}
where $x = (y, z) \in \mathbb{T}^3 \times \mathbb{T}^{d - 3}$.
Combining all the ingredients above we obtain 
\begin{equation*}
    \begin{aligned}
        \mathcal{I}_{3b}^{\beta} &\lesssim \sum_{i = 1}^d\sum_{\lambda \gg \lambda_1}\lambda_1^{\frac{d - 1}{2}}(t^{\frac{\varepsilon}{8}}\lambda^{-\sigma + \frac{d}{4} + \frac{\varepsilon + 3}{4}}a_{\lambda})^2(t^{\frac{3\varepsilon - \varepsilon^2}{8}}\lambda^{-\sigma + 1}\lambda_1^{-\sigma + 1 + \frac{\varepsilon}{2}}a_{\lambda}a_{\lambda_1}) \\ 
        &\lesssim t^{\frac{5\varepsilon - \varepsilon^2}{8}}\lambda^{-2\sigma + 2}a_{\lambda}^2\sum_{\lambda \gg \lambda_1}\lambda^{-\sigma + \frac{d + 1}{2} + \frac{\varepsilon}{2}}a_{\lambda}\lambda_1^{-\sigma + \frac{d + 1}{2} + \frac{\varepsilon}{2}}a_{\lambda_1} \\ 
        &\lesssim \lambda^{-2\sigma + 2}a_{\lambda}^2,
    \end{aligned}
\end{equation*}
We note that since $(5\varepsilon - \varepsilon^2)/8 > 0$ and $-\sigma + (d + 1)/2 + \varepsilon/2 > 0$, the last inequality holds by choosing $T^* = T^*(\Vert \bm{\psi}(0)\Vert_{H_x^{\sigma - 1}}) > 0$  sufficiently small.
Now we obtain 
\begin{equation*}
    \mathcal{I}_{3b} \lesssim \lambda^{-2\sigma + 2}a_{\lambda}^2, \qquad t \in [0, T^*].
\end{equation*}

Finally, for $\mathcal{I}_{3c}$ we split it into three parts:
\begin{align*}
    \mathcal{I}_{3c} &= \left(\sum_{\lambda_2 \sim \lambda_3 \gg \lambda_1 \gtrsim \lambda} + \sum_{\lambda_2 \sim \lambda_3 \sim \lambda_1 \gg \lambda} + \sum_{\lambda_2 \sim \lambda_3 \sim \lambda_1 \sim \lambda}\right)\Vert P_{\lambda}\psi_kP_{\lambda_1}\psi_kP_{\lambda_2}\bm{\psi} \cdot P_{\lambda_3}\bm{\psi}\Vert_{L_{t, x}^1} \\
    &=: \mathcal{I}_{3c}^{\alpha} + \mathcal{I}_{3c}^{\beta} + \mathcal{I}_{3c}^{\gamma}.
\end{align*}
The terms $\mathcal{I}_{3c}^{\alpha}$ and $\mathcal{I}_{3c}^{\beta}$ can be treated exactly the same as $\mathcal{I}_{3a}$ and $\mathcal{I}_{3b}^{\beta}$, respectively. 
For the term $\mathcal{I}_{3c}^{\gamma}$, by \eqref{L4L4} we have 
\begin{equation*}
    \begin{aligned}
        \mathcal{I}_{3c} &\lesssim \sum_{\lambda_1 \sim \lambda}\Vert P_{\lambda_1}\bm{\psi}\Vert_{L_{t, x}^4}^2\Vert P_{\lambda}\bm{\psi}\Vert_{L_{t, x}^4}^2 \\
        &\lesssim t^{\frac{\varepsilon}{2}}\sum_{\lambda_1 \sim \lambda}\lambda_1^{-2\sigma + \frac{d}{2} + \frac{\varepsilon + 3}{2}}a_{\lambda_1}^2\lambda^{-2\sigma + \frac{d}{2} + \frac{\varepsilon + 3}{2}}a_{\lambda}^2 \\ 
        &\lesssim t^{\frac{\varepsilon}{2}}\lambda^{-2\sigma + 2}a_{\lambda}^2\sum_{\lambda_1 \sim \lambda}(\lambda_1^{-\sigma + \frac{d + 1}{2} + \frac{\varepsilon}{2}}a_{\lambda_1})^2 \\
        &\lesssim \lambda^{-2\sigma + 2}a_{\lambda}^2,
    \end{aligned}
\end{equation*}
where $t \in [0, T^*]$ and $(T^*)^{\varepsilon/2}\Vert \bm{\psi}(0)\Vert_{H_x^{\sigma - 1}}^2 \leq 1$.
To conclude, 
\begin{equation*}
    \mathcal{I}_3 \lesssim \mathcal{I}_3^a + \mathcal{I}_3^b + \mathcal{I}_3^c \lesssim  \lambda^{-2\sigma + 2}a_{\lambda}^2.
\end{equation*}

\subsubsection{Estimates for \texorpdfstring{$\mathcal{I}_1$}{}}\label{I1}

We recall that 
\begin{equation*}
    \widetilde{\bm{A}} = \nabla^{-1}_x\sum_{k = 1}^d\mathcal{R}_k(\Im(\overline{\psi_k}\bm{\psi}))
\end{equation*}
satisfies 
\begin{equation}\label{Cou}
    \nabla_x \cdot \widetilde{\bm{A}} = \nabla_x \cdot (\bm{A} - \bm{C}) = \nabla_x \cdot \bm{A} = 0.
\end{equation}
Using \eqref{Cou} and integrating by parts, we have 
\begin{equation*}
    \begin{aligned}
        \mathcal{I}_1 &= -2\int_0^t\int_{\mathbb{T}^d}\Re(\overline{P_{\lambda}\psi_k}P_{\lambda}(\widetilde{\bm{A}} \cdot \nabla_x\psi_k)) \,{\rm d}x{\rm d}t \\
        &= -2\int_0^t\int_{\mathbb{T}^d}\Re(\overline{P_{\lambda}\psi_k}P_{\lambda}\nabla_x \cdot (\widetilde{\bm{A}}\psi_k)) \,{\rm d}x{\rm d}t \\
        &= 2\int_0^t\int_{\mathbb{T}^d}\Re(\overline{\nabla_xP_{\lambda}\psi_k} \cdot P_{\lambda}(\widetilde{\bm{A}}\psi_k)) \,{\rm d}x{\rm d}t.
    \end{aligned}
\end{equation*}
It follows that 
\begin{equation*}
    \begin{aligned}
        \mathcal{I}_1 &= 2\int_0^t\int_{\mathbb{T}^d}\Re(\overline{\nabla_xP_{\lambda}\psi_k} \cdot [P_{\lambda}, \widetilde{\bm{A}}]\psi_k) \,{\rm d}x{\rm d}t + 2\int_0^t\int_{\mathbb{T}^d}\Re(\overline{\nabla_xP_{\lambda}\psi_k} \cdot (\widetilde{\bm{A}}P_{\lambda}\psi_k)) \,{\rm d}x{\rm d}t \\
        &=: \mathcal{I}_1^a + \mathcal{I}_1^b.
    \end{aligned}
\end{equation*}

We first note that, from \eqref{Cou} and integration by parts
\begin{equation*}
    \mathcal{I}_1^b = 2\int_0^t\int_{\mathbb{T}^d}\widetilde{\bm{A}} \cdot \nabla_x\frac{|P_{\lambda}\psi_k|^2}{2} \,{\rm d}x{\rm d}t = -\int_0^t\int_{\mathbb{T}^d}(\nabla_x \cdot \bm{\widetilde{A}})|P_{\lambda}\psi_k|^2 \,{\rm d}x{\rm d}t = 0.
\end{equation*}
Hence $\mathcal{I}_1^b$ has no contribution to $\mathcal{I}_1$.
Estimating $\mathcal{I}_1^a$, using Lemma \ref{BONYB} we have 
\begin{equation*}
    \begin{aligned}
        \mathcal{I}_1^a &\lesssim \sum_{\lambda_1 \ll \lambda_2 \sim \lambda}\lambda^{-1}\Vert \nabla_xP_{\lambda}\psi_k \cdot L(\nabla_xP_{\lambda_1}\widetilde{\bm{A}}, P_{\lambda_2}\psi_k)\Vert_{L_{t, x}^1} \\
        &\quad+ \sum_{\lambda_2 \ll \lambda_1 \sim \lambda}\Vert \nabla_xP_{\lambda}\psi_k \cdot L(P_{\lambda_1}\widetilde{\bm{A}}, P_{\lambda_2}\psi_k)\Vert_{L_{t, x}^1}\\
        &\quad+ \sum_{\lambda_1 \sim \lambda_2 \gtrsim \lambda}\Vert \nabla_xP_{\lambda}\psi_k \cdot L(P_{\lambda_1}\widetilde{\bm{A}}, P_{\lambda_2}\psi_k)\Vert_{L_{t, x}^1} \\
        &=: \mathcal{I}_{1\alpha}^a + \mathcal{I}_{1\beta}^a + \mathcal{I}_{1\gamma}^a.
    \end{aligned}
\end{equation*}
For the term $\mathcal{I}_{1\alpha}^a$, by Lemma \ref{Comuu} and the definition of $L$-notation
\begin{equation*}
    \mathcal{I}_{1\alpha}^a \lesssim \sum_{\lambda_1 \ll \lambda}\sup_{a, b, c}\Vert P_{\lambda}\psi_k^a|\nabla_xP_{\lambda_1}\widetilde{\bm{A}}^b|P_{\lambda}\psi_k^c\Vert_{L_{t, x}^1}.
\end{equation*}
From the definition of $\widetilde{\bm{A}}$, using the same techniques when controlling $\mathcal{I}_3$ in Section \ref{EnergyEs} we obtain 
\begin{equation*}
    \mathcal{I}_{1\alpha}^a \lesssim \lambda^{-2\sigma + 2}a_{\lambda}^2.
\end{equation*}
Moreover, it is clear that
\begin{equation*}
    \mathcal{I}_{1\beta}^a, \mathcal{I}_{1\gamma}^a \lesssim \lambda^{-2\sigma + 2}a_{\lambda}^2,
\end{equation*}
since the derivative from the term $\nabla_xP_{\lambda}\psi_k$ is canceled by the $-1$ derivative contained in $P_{\lambda_1}\widetilde{\bm{A}}$, and we can control $\mathcal{I}_{1\beta}^a$ and $\mathcal{I}_{1\gamma}^a$ in the same way as $\mathcal{I}_3$.
To conlude, 
\begin{equation*}
    \mathcal{I}_1 \lesssim \lambda^{-2\sigma + 2}a_{\lambda}^2.
\end{equation*}

\subsubsection{Estimates for \texorpdfstring{$\mathcal{I}_2$}{}}\label{I2}

We have 
\begin{equation*}
    \mathcal{I}_2 \lesssim \Vert P_{\lambda}\psi_kP_{\lambda}(A_0\psi_k)\Vert_{L_{t, x}^1} + \Vert P_{\lambda}\psi_kP_{\lambda}((|\widetilde{\bm{A}}|^2 + 2\bm{C} \cdot \bm{\widetilde{A}})\psi_k)\Vert_{L_{t, x}^1}
\end{equation*}
For the first term, from \eqref{A00} we know that $A_0$ can be regarded as a quadric term of $\bm{\psi}$, containing no derivatives at scale $\lambda$. Therefore, by the analysis in Section \ref{3.1.1} we have 
\begin{equation*}
    \Vert P_{\lambda}\psi_kP_{\lambda}(A_0\psi_k)\Vert_{L_{t, x}^1} \lesssim \lambda^{-2\sigma + 2}a_{\lambda}^2.
\end{equation*}
For the second term, from Sobolev embeddings $H_x^{(d/2)+} \hookrightarrow L_x^{\infty}$ and Proposition \ref{B1}, it is obvious that 
\begin{equation*}
    \Vert P_{\lambda}\psi_kP_{\lambda}((|\widetilde{\bm{A}}|^2 + 2\bm{C} \cdot \bm{\widetilde{A}})\psi_k)\Vert_{L_{t, x}^1} \lesssim \lambda^{-2\sigma + 2}a_{\lambda}^2.
\end{equation*}
Thus, 
\begin{equation*}
    \mathcal{I}_2 \lesssim \lambda^{-2\sigma + 2}a_{\lambda}^2.
\end{equation*}

\subsection{\texorpdfstring{$L_{t, x}^2$}{} bilinear estimate}\label{L2L2}

Without loss of generality, we only restrict our attention to the case $i = 1$ in proving \eqref{16}.
We recall the mass balance law (see Section \ref{Balance})
\begin{equation}\label{mass}
    \partial_t\int_{\mathbb{T}^{d - 1}}\frac{|P_{\mu}\psi_k^a|^2}{2} \,{\rm d}\widehat{x}_1 + \partial_1\int_{\mathbb{T}^{d - 1}}\Im(\overline{P_{\mu}\psi_k^a}D_1^{\bm{C}}P_{\mu}\psi_k^a) \,{\rm d}\widehat{x}_1 = \int_{\mathbb{T}^{d - 1}}\Im(\overline{P_{\mu}\psi_k^a}P_{\mu}\mathcal{N}(\psi_k^a)) \,{\rm d}\widehat{x}_1,
\end{equation}
and momentum balance law
\begin{equation}\label{mom}
    \begin{aligned}
        &-\partial_t\int_{\mathbb{T}^{d - 1}}\Im(\overline{P_{\lambda, e_1}\psi_{\ell}^b}D_1^{\bm{C}}P_{\lambda, e_1}\psi_{\ell}^b) \,{\rm d}\widehat{x}_1 + \partial_1\int_{\mathbb{T}^{d - 1}}\left(\partial_1^2\frac{|P_{\lambda, e_1}\psi_{\ell}^b|^2}{2} - 2|D_1^{\bm{C}}P_{\lambda, e_1}\psi_{\ell}^b|^2\right) \,{\rm d}\widehat{x}_1 \\
        = &\int_{\mathbb{T}^{d - 1}}\{P_{\lambda, e_1}\psi_{\ell}^b, P_{\lambda, e_1}\mathcal{N}(\psi_\ell^b)\}_1^{\bm{C}} \,{\rm d}\widehat{x}_1.
    \end{aligned}
\end{equation}
where $D_1^{\bm{C}} = \partial_1+ \sqrt{-1}C_1$.
Applying div-curl lemma \eqref{div1} to the system \eqref{mass}-\eqref{mom}, we have
\begin{equation}\label{dc1}
    {\rm LHS} \leq {\rm RHS},
\end{equation}
where 
\begin{equation*}
\begin{aligned}
    {\rm LHS} &= L_1 + L_2, \\ 
    {\rm RHS} &= R_1 + R_2 + R_3 + R_4
\end{aligned}
\end{equation*}
and
\begin{equation*}
\begin{aligned}
    L_1 &= -\int_0^t\int_0^1\left(\int_{[0, 1]^{d - 1}}\frac{|P_{\mu}\psi_k^a|^2}{2} \,{\rm d}\widehat{x}_1\right)\left(\int_{[0, 1]^{d - 1}}\left(\partial_1^2\frac{|P_{\lambda, e_1}\psi_{\ell}^b|^2}{2} - 2|D_1^{\bm{C}}P_{\lambda, e_1}\psi_{\ell}^b|^2\right) \,{\rm d}\widehat{x}_1\right) \,{\rm d}x_1{\rm d}t,\\
    L_2 &= -\int_0^t\int_0^1\left(\int_{[0, 1]^{d - 1}}\Im(\overline{P_{\lambda, e_1}\psi_{\ell}^b}D_1^{\bm{C}}P_{\lambda, e_1}\psi_{\ell}^b) \,{\rm d}\widehat{x}_1\right)\left(\int_{[0, 1]^{d - 1}}\Im(\overline{P_{\mu}\psi_k^a}D_1^{\bm{C}}P_{\mu}\psi_k^a) \,{\rm d}\widehat{x}_1\right) \,{\rm d}x_1{\rm d}t, \\
    R_1 &= 4\sup_t\Vert P_{\mu}\psi_k^a\Vert_{L_x^2}^2\Vert \Im(\overline{P_{\lambda, e_1}\psi_{\ell}^b}D_1^{\bm{C}}P_{\lambda, e_1}\psi_{\ell}^b)\Vert_{L_x^1}, \\
    R_2 &= \Bigg|\int_0^t\int_0^1\left(\int_{[0, 1]^{d - 1}}\Im(\overline{P_{\lambda, e_1}\psi_{\ell}^b}D_1^{\bm{C}}P_{\lambda, e_1}\psi_{\ell}^b) \,{\rm d}\widehat{x}_1\right)\\
    &\quad\quad\quad\,\,\,\quad \left(\int_0^1\int_{y_1}^{x_1}\int_{[0, 1]^{d - 1}}\Im(\overline{P_{\mu}\psi_k^a}P_{\mu}\mathcal{N}(\psi_k^a)) \,{\rm d}\widehat{x}_1{\rm d}z_1{\rm d}y_1 \right)\,{\rm d}x_1{\rm d}t\Bigg| \\
    &\quad+ \Bigg|\int_0^t\left(\int_{[0, 1]^d}\Im(\overline{P_{\mu}\psi_k^a}P_{\mu}\mathcal{N}(\psi_k^a)) \,{\rm d}x\right)\\
    &\quad\quad\,\,\,\,\quad\left(\int_0^1\left(x - \frac{1}{2}\right)\int_{[0, 1]^{d - 1}}\Im(\overline{P_{\lambda, e_1}\psi_{\ell}^b}D_1^{\bm{C}}P_{\lambda, e_1}\psi_{\ell}^b) \,{\rm d}\widehat{x}_1\right) \,{\rm d}x_1\Bigg|\\
    R_3 &= \Bigg|\int_0^t\int_0^1\left(\int_{[0, 1]^{d - 1}}\{P_{\lambda, e_1}\psi_{\ell}^b, P_{\lambda, e_1}\mathcal{N}(\psi_\ell^b)\}_1^{\bm{C}} \,{\rm d}\widehat{x}_1\right)\\
    &\quad\quad\quad\quad\,\,\,\left(\int_0^1\int_{y_1}^{x_1}\int_{[0, 1]^{d - 1}}\frac{|P_{\mu}\psi_k^a|^2}{2} \,{\rm d}\widehat{x}_1{\rm d}z_1{\rm d}y_1 \right)\,{\rm d}x_1{\rm d}t\Bigg| \\
    &\quad+ \Bigg|\int_0^t\left(\int_{[0, 1]^d}\frac{|P_{\mu}\psi_k^a|^2}{2} \,{\rm d}x\right)\left(\int_0^1\left(x - \frac{1}{2}\right)\int_{[0, 1]^{d - 1}}\{P_{\lambda, e_1}\psi_{\ell}^b, P_{\lambda, e_1}\mathcal{N}(\psi_\ell^b)\}_1^{\bm{C}} \,{\rm d}\widehat{x}_1\right) \,{\rm d}x_1\Bigg|\\
    R_4 &= -\int_0^t\left(\int_{[0, 1]^d}\frac{|P_{\mu}\psi_k^a|^2}{2} \,{\rm d}x\right)\left(\int_{[0, 1]^d}\left(\partial_1^2\frac{|P_{\lambda, e_1}\psi_{\ell}^b|^2}{2} - 2|D_1^{\bm{C}}P_{\lambda, e_1}\psi_{\ell}^b|^2\right) \,{\rm d}x\right) \,{\rm d}t \\
    &\quad -\int_0^t\left(\int_{[0, 1]^d}\Im(P_{\lambda, e_1}\psi_{\ell}^b\overline{D_1^{\bm{C}}P_{\lambda}\psi_k^b}) \,{\rm d}x\right)\left(\int_{[0, 1]^d}\Im(P_{\mu}\psi_k^a\overline{D_1^{\bm{C}}P_{\mu}\psi_k^a}) \,{\rm d}x\right) \,{\rm d}t.
\end{aligned}
\end{equation*}
Moreover, by interchanging $\lambda$ and $\mu$, $k$ and $\ell$ in \eqref{mass}-\eqref{mom} we also have 
\begin{equation}\label{mass1}
    \begin{aligned}
    \partial_t\int_{\mathbb{T}^{d - 1}}\frac{|P_{\lambda, e_1}\psi_{\ell}^b|^2}{2} \,{\rm d}\widehat{x}_1 + \partial_1\int_{\mathbb{T}^{d - 1}}\Im(\overline{P_{\lambda, e_1}\psi_{\ell}^b}D_1^{\bm{C}}P_{\lambda, e_1}\psi_{\ell}^b) \,{\rm d}\widehat{x}_1 = \int_{\mathbb{T}^{d - 1}}\Im(\overline{P_{\lambda, e_1}\psi_{\ell}^b}P_{\lambda, e_1}\mathcal{N}(\psi_k^b)) \,{\rm d}\widehat{x}_1
    \end{aligned}
\end{equation}
and 
\begin{equation}\label{mom1}
\begin{aligned}
        &\partial_t\int_{\mathbb{T}^{d - 1}}\Im(\overline{P_{\mu}\psi_k^a}D_1^{\bm{C}}P_{\mu}\psi_k^a) \,{\rm d}\widehat{x}_1 - \partial_1\int_{\mathbb{T}^{d - 1}}\left(\partial_1^2\frac{|P_{\mu}\psi_k^a|^2}{2} - 2|D_1^{\bm{C}}P_{\mu}\psi_k^a|^2\right) \,{\rm d}\widehat{x}_1 \\
        = &\int_{\mathbb{T}^{d - 1}}\{P_{\mu}\mathcal{N}(\psi_k^a), P_{\mu}\psi_k^a\} \,{\rm d}\widehat{x}_1.
\end{aligned}
\end{equation}
Again, applying div-curl lemma to \eqref{mass1}-\eqref{mom1} we obtain    
\begin{equation}\label{dc2}
    {\rm LHS}' \leq {\rm RHS}',
\end{equation}
where
\begin{equation*}
\begin{aligned}
    {\rm LHS}' &= L_1' + L_2', \\ 
    {\rm RHS}' &= R_1' + R_2' + R_3' + R_4'
\end{aligned}
\end{equation*}
and
\begin{equation*}
\begin{aligned}
    L_1' &= \int_0^t\int_0^1\left(\int_{[0, 1]^{d - 1}}\frac{|P_{\lambda, e_1}\psi_{\ell}^b|^2}{2} \,{\rm d}\widehat{x}_1\right)\left(\int_{[0, 1]^{d - 1}}\left(\partial_1^2\frac{|P_{\mu}\psi_k^a|^2}{2} - 2|D_1^{\bm{C}}P_{\mu}\psi_k^a|^2\right) \,{\rm d}\widehat{x}_1\right) \,{\rm d}x_1{\rm d}t,\\
    L_2' &= \int_0^t\int_0^1\left(\int_{[0, 1]^{d - 1}}\Im(\overline{P_{\mu}\psi_k^a}D_1^{\bm{C}}P_{\mu}\psi_k^a) \,{\rm d}\widehat{x}_1\right)\left(\int_{[0, 1]^{d - 1}}\Im(\overline{P_{\lambda, e_1}\psi_{\ell}^b}D_1^{\bm{C}}P_{\lambda, e_1}\psi_{\ell}^b) \,{\rm d}\widehat{x}_1\right) \,{\rm d}x_1{\rm d}t, \\
    R_1' &= 4\sup_t\Vert P_{\lambda, e_1}\psi_\ell^b\Vert_{L_x^2}^2\Vert \Im(\overline{P_{\mu}\psi_k^a}D_1^{\bm{C}}P_{\mu}\psi_k^a)\Vert_{L_x^1}, \\
    R_2' &= \Bigg|\int_0^t\int_0^1\left(\int_{[0, 1]^{d - 1}}\Im(\overline{P_{\mu}\psi_k^a}D_1^{\bm{C}}P_{\mu}\psi_k^a) \,{\rm d}\widehat{x}_1\right)\\
    &\quad\quad\quad\,\,\,\quad \left(\int_0^1\int_{y_1}^{x_1}\int_{[0, 1]^{d - 1}}\Im(\overline{P_{\lambda, e_1}\psi_{\ell}^b}P_{\lambda, e_1}\mathcal{N}(\psi_{\ell}^b)) \,{\rm d}\widehat{x}_1{\rm d}z_1{\rm d}y_1 \right)\,{\rm d}x_1{\rm d}t\Bigg| \\
    &\quad+ \Bigg|\int_0^t\left(\int_{[0, 1]^d}\Im(\overline{P_{\lambda, e_1}\psi_{\ell}^b}P_{\lambda, e_1}\mathcal{N}(\psi_{\ell}^b)) \,{\rm d}x\right)\\
    &\quad\quad\,\,\,\,\quad\left(\int_0^1\left(x - \frac{1}{2}\right)\int_{[0, 1]^{d - 1}}\Im(\overline{P_{\mu}\psi_k^a}D_1^{\bm{C}}P_{\lambda, e_1}\psi_{\ell}^b) \,{\rm d}\widehat{x}_1\right) \,{\rm d}x_1\Bigg|\\
    R_3' &= \Bigg|\int_0^t\int_0^1\left(\int_{[0, 1]^{d - 1}}\{P_{\mu}\psi_k^a, P_{\mu}\mathcal{N}(\psi_k^a)\}_1^{\bm{C}} \,{\rm d}\widehat{x}_1\right)\\
    &\quad\quad\,\,\quad\,\,\,\,\,\,\left(\int_0^1\int_{y_1}^{x_1}\int_{[0, 1]^{d - 1}}\frac{|P_{\lambda, e_1}\psi_{\ell}^b|^2}{2} \,{\rm d}\widehat{x}_1{\rm d}z_1{\rm d}y_1 \right)\,{\rm d}x_1{\rm d}t\Bigg| \\
    &\quad+ \Bigg|\int_0^t\left(\int_{[0, 1]^d}\frac{|P_{\lambda, e_1}\psi_{\ell}^b|^2}{2} \,{\rm d}x\right)\left(\int_0^1\left(x - \frac{1}{2}\right)\int_{[0, 1]^{d - 1}}\{P_{\mu}\psi_k^a, P_{\mu}\mathcal{N}(\psi_k^a)\}_1^{\bm{C}} \,{\rm d}\widehat{x}_1\right) \,{\rm d}x_1\Bigg|\\
    R_4' &= \int_0^t\left(\int_{[0, 1]^{d }}\frac{|P_{\lambda, e_1}\psi_{\ell}^b|^2}{2} \,{\rm d}x\right)\left(\int_{[0, 1]^{d}}\left(\partial_1^2\frac{|P_{\mu}\psi_k^a|^2}{2} - 2|D_1^{\bm{C}}P_{\mu}\psi_k^a|^2\right) \,{\rm d}x\right) \,{\rm d}t \\
    &\quad + \int_0^t\left(\int_{[0, 1]^{d}}\Im(\overline{P_{\mu}\psi_k^a}D_1^{\bm{C}}P_{\mu}\psi_k^a) \,{\rm d}x\right)\left(\int_{[0, 1]^{d}}\Im(\overline{P_{\lambda, e_1}\psi_{\ell}^b}D_1^{\bm{C}}P_{\lambda, e_1}\psi_{\ell}^b) \,{\rm d}x\right) \,{\rm d}t.
\end{aligned}
\end{equation*}
Note that 
\begin{equation*}
    \begin{aligned}
    &\int_0^t\int_0^1\left(\int_{[0, 1]^{d - 1}}\frac{|P_{\mu}\psi_k^a|^2}{2} \,{\rm d}\widehat{x}_1\right)\left(\int_{[0, 1]^{d - 1}}\partial_1^2\frac{|P_{\lambda, e_1}\psi_{\ell}^b|^2}{2} \,{\rm d}\widehat{x}_1\right) \,{\rm d}x_1{\rm d}t \\ 
    = &-\int_0^t\int_0^1\left(\partial_1\int_{[0, 1]^{d - 1}}\frac{|P_{\mu}\psi_k^a|^2}{2} \,{\rm d}\widehat{x}_1\right)\left(\partial_1\int_{[0, 1]^{d - 1}}\frac{|P_{\lambda, e_1}\psi_{\ell}^b|^2}{2} \,{\rm d}\widehat{x}_1\right) \,{\rm d}x_1{\rm d}t \\ 
    = &\int_0^t\int_0^1\left(\int_{[0, 1]^{d - 1}}\partial_1^2\frac{|P_{\mu}\psi_k^a|^2}{2} \,{\rm d}\widehat{x}_1\right)\left(\int_{[0, 1]^{d - 1}}\frac{|P_{\lambda, e_1}\psi_{\ell}^b|^2}{2} \,{\rm d}\widehat{x}_1\right) \,{\rm d}x_1{\rm d}t.
\end{aligned}
\end{equation*}
Adding \eqref{dc1} to \eqref{dc2} yields
\begin{equation}\label{b1}
    \begin{aligned}
        \left\Vert \Vert P_{\mu}\psi_k^a\Vert_{L_{\widehat{x}_1}^2}\Vert D_1^{\bm{C}}P_{\lambda, e_1}\psi_{\ell}^b\Vert_{L_{\widehat{x}_1}^2}\right\Vert^2_{L_{t, x_1}^2} - \left\Vert \Vert P_{\lambda, e_1}\psi_{\ell}^b\Vert_{L_{\widehat{x}_1}^2}\Vert D_1^{\bm{C}}P_{\mu}\psi_k^a\Vert_{L_{\widehat{x}_1}^2}\right\Vert^2_{L_{t, x_1}^2} \lesssim \mathcal{R} + \mathcal{S},
    \end{aligned}
\end{equation}
where 
\begin{equation*}
    \begin{aligned}
    \mathcal{R} &= R_1 + R_2 + R_3 + R_1' + R_2' + R_3', \\ 
    \mathcal{S} &= R_4 + R_4' = \left\Vert \Vert P_{\mu}\psi_k^a\Vert_{L_x^2}\Vert D_1^{\bm{C}}P_{\lambda, e_1}\psi_{\ell}^b\Vert_{L_x^2}\right\Vert_{L_t^2}^2 - \left\Vert \Vert P_{\lambda, e_1}\psi_{\ell}^b\Vert_{L_x^2}\Vert D_1^{\bm{C}}P_{\mu}\psi_k^a\Vert_{L_x^2}\right\Vert_{L_t^2}^2.
\end{aligned}
\end{equation*}

Now take supreme over $a$ and $b$ on both sides of \eqref{b1}.
On the one hand, from the definition of $D_1^{\bm{C}}$ we have 
\begin{equation*}
    \begin{aligned}
                &\left\Vert \Vert P_{\mu}\psi_k^a\Vert_{L_{\widehat{x}_1}^2}\Vert D_1^{\bm{C}}P_{\lambda, e_1}\psi_{\ell}^b\Vert_{L_{\widehat{x}_1}^2}\right\Vert^2_{L_{t, x_1}^2} - \left\Vert \Vert P_{\lambda, e_1}\psi_{\ell}^b\Vert_{L_{\widehat{x}_1}^2}\Vert D_1^{\bm{C}}P_{\mu}\psi_k^a\Vert_{L_{\widehat{x}_1}^2}\right\Vert^2_{L_{t, x_1}^2} \\ 
    = &\left\Vert \Vert P_{\mu}\psi_k^a\Vert_{L_{\widehat{x}_1}^2}\Vert P_{\lambda, e_1}\partial_1\psi_{\ell}^b\Vert_{L_{\widehat{x}_1}^2}\right\Vert^2_{L_{t, x_1}^2} - \left\Vert \Vert P_{\lambda, e_1}\psi_{\ell}^b\Vert_{L_{\widehat{x}_1}^2}\Vert P_{\mu}\partial_1\psi_k^a\Vert_{L_{\widehat{x}_1}^2}\right\Vert^2_{L_{t, x_1}^2} \\
    &+ 2C_1\int_0^t\int_0^1\Vert P_{\mu}\psi_k^a\Vert_{L_{\widehat{x}_1}^2}^2\left(\int_{[0, 1]^{d - 1}}\Im(\overline{P_{\lambda, e_1}\psi_{\ell}^b}\partial_1P_{\lambda, e_1}\psi_{\ell}^b) \,{\rm d}\widehat{x}_1\right) \,{\rm d}x_1{\rm d}t \\
    &- 2C_1\int_0^t\int_0^1\Vert P_{\lambda, e_1}\psi_{\ell}^b\Vert_{L_{\widehat{x}_1}^2}^2\left(\int_{[0, 1]^{d - 1}}\Im(\overline{P_{\mu}\psi_k^a}\partial_1P_{\mu}\psi_k^a) \,{\rm d}\widehat{x}_1\right) \,{\rm d}x_1{\rm d}t.
\end{aligned}
\end{equation*}
Thus, 
\begin{equation*}
    \begin{aligned}
            &\sup_{a, b}\left(\left\Vert \Vert P_{\mu}\psi_k^a\Vert_{L_{\widehat{x}_1}^2}\Vert D_1^{\bm{C}}P_{\lambda, e_1}\psi_{\ell}^b\Vert_{L_{\widehat{x}_1}^2}\right\Vert^2_{L_{t, x_1}^2} - \left\Vert \Vert P_{\lambda, e_1}\psi_{\ell}^b\Vert_{L_{\widehat{x}_1}^2}\Vert D_1^{\bm{C}}P_{\mu}\psi_k^a\Vert_{L_{\widehat{x}_1}^2}\right\Vert^2_{L_{t, x_1}^2}\right) \\ 
    \geq &\sup_{a, b, c} \lambda^2\left\Vert \Vert P_{\mu}\psi_k^a\Vert_{L_{\widehat{x}_1}^2}\Vert P_{\lambda, e_1}\psi_{\ell}^{b + c}\Vert_{L_{\widehat{x}_1}^2}\right\Vert^2_{L_{t, x_1}^2} - \sup_{a, b, c}\mu^2\left\Vert \Vert P_{\lambda, e_1}\psi_{\ell}^b\Vert_{L_{\widehat{x}_1}^2}\Vert P_{\mu}\psi_k^{a + c}\Vert_{L_{\widehat{x}_1}^2}\right\Vert^2_{L_{t, x_1}^2} \\ 
    &- 2|C_1|\sup_{a, b, c}\lambda\left\Vert \Vert P_{\mu}\psi_k^a\Vert_{L_{\widehat{x}_1}^2}\Vert P_{\lambda, e_1}\psi_{\ell}^{b + c}\Vert_{L_{\widehat{x}_1}^2}\right\Vert^2_{L_{t, x_1}^2} \\
    &- 2|C_1|\sup_{a, b, c}\mu\left\Vert \Vert P_{\lambda, e_1}\psi_{\ell}^b\Vert_{L_{\widehat{x}_1}^2}\Vert P_{\mu}\psi_k^{a + c}\Vert_{L_{\widehat{x}_1}^2}\right\Vert^2_{L_{t, x_1}^2}.
\end{aligned}
\end{equation*}
Namely, 
\begin{equation*}
    \begin{aligned}
            &\sup_{a, b}\left(\left\Vert \Vert P_{\mu}\psi_k^a\Vert_{L_{\widehat{x}_1}^2}\Vert D_1^{\bm{C}}P_{\lambda, e_1}\psi_{\ell}^b\Vert_{L_{\widehat{x}_1}^2}\right\Vert^2_{L_{t, x_1}^2} - \left\Vert \Vert P_{\lambda, e_1}\psi_{\ell}^b\Vert_{L_{\widehat{x}_1}^2}\Vert D_1^{\bm{C}}P_{\mu}\psi_k^a\Vert_{L_{\widehat{x}_1}^2}\right\Vert^2_{L_{t, x_1}^2}\right) \\ 
    \geq &\sup_a \lambda^2\left\Vert \Vert P_{\mu}\psi_k^a\Vert_{L_{\widehat{x}_1}^2}\Vert P_{\lambda, e_1}\psi_{\ell}\Vert_{L_{\widehat{x}_1}^2}\right\Vert^2_{L_{t, x_1}^2} - \sup_a\mu^2\left\Vert \Vert P_{\lambda, e_1}\psi_{\ell}\Vert_{L_{\widehat{x}_1}^2}\Vert P_{\mu}\psi_k^a\Vert_{L_{\widehat{x}_1}^2}\right\Vert^2_{L_{t, x_1}^2} \\ 
    &- 2|C_1|\sup_a\lambda\left\Vert \Vert P_{\mu}\psi_k^a\Vert_{L_{\widehat{x}_1}^2}\Vert P_{\lambda, e_1}\psi_{\ell}\Vert_{L_{\widehat{x}_1}^2}\right\Vert^2_{L_{t, x_1}^2} \\
    &- 2|C_1|\sup_a\mu\left\Vert \Vert P_{\lambda, e_1}\psi_{\ell}\Vert_{L_{\widehat{x}_1}^2}\Vert P_{\mu}\psi_k^a\Vert_{L_{\widehat{x}_1}^2}\right\Vert^2_{L_{t, x_1}^2} \\ 
    = &\,(\lambda^2 - \mu^2 - 2|C_1|(\lambda + \mu))\sup_a\left\Vert \Vert P_{\mu}\psi_k^a\Vert_{L_{\widehat{x}_1}^2}\Vert P_{\lambda, e_1}\psi_{\ell}\Vert_{L_{\widehat{x}_1}^2}\right\Vert^2_{L_{t, x_1}^2}.
\end{aligned}
\end{equation*}
Since $\lambda \gg \mu$, $\lambda^2 - \mu^2 - 2|C_1|(\lambda + \mu) \gtrsim \lambda^2$.
It follows that 
\begin{equation}\label{LHS}
\begin{aligned}
    &\sup_{a, b}\left(\left\Vert \Vert P_{\mu}\psi_k^a\Vert_{L_{\widehat{x}_1}^2}\Vert D_1^{\bm{C}}P_{\lambda, e_1}\psi_{\ell}^b\Vert_{L_{\widehat{x}_1}^2}\right\Vert^2_{L_{t, x_1}^2} - \left\Vert \Vert P_{\lambda, e_1}\psi_{\ell}^b\Vert_{L_{\widehat{x}_1}^2}\Vert D_1^{\bm{C}}P_{\mu}\psi_k^a\Vert_{L_{\widehat{x}_1}^2}\right\Vert^2_{L_{t, x_1}^2}\right) \\ 
    \gtrsim &\,\lambda^2\sup_a\left\Vert \Vert P_{\mu}\psi_k^a\Vert_{L_{\widehat{x}_1}^2}\Vert P_{\lambda, e_1}\psi_{\ell}\Vert_{L_{\widehat{x}_1}^2}\right\Vert^2_{L_{t, x_1}^2}.
\end{aligned}
\end{equation}

In the following we estimate $\mathcal{R}$ term by term. Firstly by energy estimates (see Section \ref{EnergyEs})
\begin{equation*}
    R_1 \lesssim \lambda\mu^{-\sigma + 2}a_{\mu}\lambda^{-2\sigma + 2}a_{\lambda}.
\end{equation*}
For $R_2$ we first consider the first term, and focus on dealing with $\Im(\overline{P_{\mu}\psi_k^a}P_{\mu}\mathcal{N}(\psi_k^a))$. 
We recall that 
\begin{equation*}
    \mathcal{N}(\psi_k^a) = -2\sqrt{-1}\bm{\widetilde{A}} \cdot \nabla_x\psi_k^a + (A_0 + |\widetilde{\bm{A}}|^2 + 2\bm{C} \cdot \bm{\widetilde{A}})\psi_k^a - \sqrt{-1}\bm{\psi}^a \cdot \Im(\psi_k^a\overline{\bm{\psi}^a})
\end{equation*}
For the last two term on the right-hand side we may proceed as in Sections \ref{I2} and \ref{3.1.1}.
For the magnetic term $\sqrt{-1}\widetilde{\bm{A}} \cdot \nabla_x\psi_k^a$ (the constant $-2$ is ignored) we use integration by parts, just we have done in Section \ref{I1}.
Precisely, 
\begin{equation}\label{456}
    \begin{aligned}
        &\int_0^1\int_{y_1}^{x_1}\int_{[0, 1]^{d - 1}}\Im(\overline{P_{\mu}\psi_k^a}P_{\mu}(\sqrt{-1}\bm{\widetilde{A}} \cdot \nabla_x\psi_k^a)) \,{\rm d}\widehat{x}_1{\rm d}z_1{\rm d}y_1 \\ 
        = &\int_0^1\int_{y_1}^{x_1}\int_{[0, 1]^{d - 1}}\Re(\overline{P_{\mu}\psi_k^a}P_{\mu}(\bm{\widetilde{A}} \cdot \nabla_x\psi_k^a)) \,{\rm d}\widehat{x}_1{\rm d}z_1{\rm d}y_1 \\ 
        = &\int_0^1\int_{y_1}^{x_1}\int_{[0, 1]^{d - 1}}\Re(\overline{P_{\mu}\psi_k^a}[P_{\mu}, \bm{\widetilde{A}}] \cdot \nabla_x\psi_k^a) \,{\rm d}\widehat{x}_1{\rm d}z_1{\rm d}y_1 + \int_0^1\int_{y_1}^{x_1}\int_{[0, 1]^{d - 1}}\Re(\overline{P_{\mu}\psi_k^a}\bm{\widetilde{A}} \cdot \nabla_xP_{\mu}\psi_k^a) \,{\rm d}\widehat{x}_1{\rm d}z_1{\rm d}y_1,
    \end{aligned}
\end{equation}
and 
\begin{equation}\label{123}
    \begin{aligned}
        \int_0^1\int_{y_1}^{x_1}\int_{[0, 1]^{d - 1}}\Re(\overline{P_{\mu}\psi_k^a}\bm{\widetilde{A}} \cdot \nabla_xP_{\mu}\psi_k^a) \,{\rm d}\widehat{x}_1{\rm d}z_1{\rm d}y_1 &= \int_0^1\int_{y_1}^{x_1}\int_{[0, 1]^{d - 1}}\bm{\widetilde{A}} \cdot \nabla_x\frac{|P_{\mu}\psi_k^a|^2}{2} \,{\rm d}\widehat{x}_1{\rm d}z_1{\rm d}y_1\\ 
        &= \int_0^1\int_{y_1}^{x_1}\partial_{z_1}\int_{[0, 1]^{d - 1}}\widetilde{A}_1\frac{|P_{\mu}\psi_k^a|^2}{2} \,{\rm d}\widehat{x}_1{\rm d}z_1{\rm d}y_1 \\ 
        &= \frac{1}{2}\int_{[0, 1]^{d - 1}}\widetilde{A}_1|P_{\mu}\psi_k^a|^2(x_1) \,{\rm d}\widehat{x}_1 - \frac{1}{2}\int_{[0, 1]^d}\widetilde{A}_1|P_{\mu}\psi_k^a|^2 \,{\rm d}x,
    \end{aligned}
\end{equation}
where we use the Coulomb gauge condition $\nabla_x \cdot \bm{\widetilde{A}} = 0$.
The commutator term in \eqref{456} should be expanded as 
\begin{equation*}
    \begin{aligned}\relax
        [P_{\mu}, \widetilde{\bm{A}}] \cdot \nabla_x\psi_k^a = \mu^{-1}L(\nabla_xP_{\ll \mu}\widetilde{\bm{A}}, \nabla_xP_{\mu}\psi_k^a) + L(P_{\mu}\widetilde{\bm{A}}, \nabla_xP_{\ll\mu}\psi_k^a) + \sum_{\mu_1 \sim \mu_2 \gtrsim \mu}P_{\mu}(P_{\mu_1}\widetilde{\bm{A}}, \nabla_xP_{\mu_2}\psi_k^a).
    \end{aligned}
\end{equation*} 
using \eqref{BONY}, which can then be bounded in a analogous approach when controlling $\mathcal{I}_1$ in Section \ref{I1}.
Inserting \eqref{123} into the definition of $R_2$, it suffices to control
\begin{equation*}
    \begin{aligned}
        &\left|\int_0^t\int_0^1\left(\int_{[0, 1]^{d - 1}}\Im(\overline{P_{\lambda, e_1}\psi_\ell^b}D_1^{\bm{C}}P_{\lambda, e_1}\psi_\ell^b) \,{\rm d}\widehat{x}_1\right)\left(\int_{[0, 1]^{d - 1}}\widetilde{A}_1|P_{\mu}\psi_k^a|^2 \,{\rm d}\widehat{x}_1\right) \,{\rm d}x_1{\rm d}t\right| \\ 
        + &\left|\int_0^t\left(\int_{[0, 1]^d}\Im(\overline{P_{\lambda, e_1}\psi_\ell^b}D_1^{\bm{C}}P_{\lambda, e_1}\psi_\ell^b) \,{\rm d}x\right)\left(\int_{[0, 1]^d}\widetilde{A}_1|P_{\mu}\psi_k^a|^2 \,{\rm d}x\right) \,{\rm d}t\right| =: R_2^{\alpha} + R_2^{\beta}
    \end{aligned}
\end{equation*}
For $R_2^a$, by H\"older inequality 
\begin{equation*}
    \begin{aligned}
            R_2^{\alpha} &\lesssim \lambda\int_0^t\int_0^1\Vert P_{\lambda, e_1}\bm{\psi}\Vert_{L_{\widehat{x}_1}^2}^2\Vert \widetilde{A}_1\Vert_{L_{\widetilde{x}_1}^{\infty}}\Vert P_{\mu}\bm{\psi}\Vert_{L_{\widehat{x}_1}^2}^2 \,{\rm d}x_1{\rm d}t \\
    &\lesssim \lambda\Vert P_{\lambda}\bm{\psi}\Vert_{L_t^{\infty}L_x^2}^2\Vert \widetilde{A}_1\Vert_{L_t^2L_x^{\infty}}\Vert P_{\mu}\bm{\psi}\Vert_{L_t^4L_{x_1}^{\infty}L_{\widehat{x}_1}^2}^2.
    \end{aligned}
\end{equation*}
By Sobolev embedding and Proposition \ref{B1} 
\begin{equation*}
    \Vert \widetilde{A}_1\Vert_{L_t^2L_x^{\infty}} \lesssim \Vert \widetilde{A}_1\Vert_{L_t^2H_x^{\frac{d + 1}{2}}} \lesssim t^{-\sqrt{\frac{\varepsilon}{2}}}\Vert \bm{\psi}(0)\Vert_{H_x^{\sigma - 1}}^2.
\end{equation*}
Together with energy estimates, we obtain 
\begin{gather*}
    \Vert P_{\lambda, e_1}\bm{\psi}\Vert_{L_t^{\infty}L_x^2}^2 \lesssim \lambda^{-2\sigma + 2}a_{\lambda}^2, \\ 
    \Vert P_{\mu}\bm{\psi}\Vert_{L_t^4L_{x_1}^{\infty}L_{\widehat{x}_1}^2}^2 \lesssim t^{\frac{1}{2}}\mu\Vert P_{\mu}\bm{\psi}\Vert_{L_t^{\infty}L_x^2}^2 \lesssim t^{\frac{1}{2}}\mu^{-2\sigma + 2 + 1}a_{\mu}^2.
\end{gather*}
We conclude 
\begin{equation*}
    R_2^{\alpha} \lesssim t^{\frac{1}{2} - \sqrt{\frac{\varepsilon}{2}}}\Vert \bm{\psi}(0)\Vert_{H_x^{\sigma - 1}}^2\lambda\mu\lambda^{-2\sigma + 2}a_{\lambda}^2\mu^{-2\sigma + 2}a_{\mu}^2 \lesssim \lambda^2\lambda^{-2\sigma + 2}a_{\lambda}^2\mu^{-2\sigma + 2}a_{\mu}^2,
\end{equation*}
where $t \in [0, T^*]$ be such that $(T^*)^{1/2 - \sqrt{\varepsilon/2}}\Vert \bm{\psi}(0)\Vert_{H_x^{\sigma - 1}}^2 \lesssim 1$.
Similarly 
\begin{equation*}
    \begin{aligned}
        R_2^{\beta} \lesssim \lambda\Vert P_{\lambda, e_1}\bm{\psi}\Vert_{L_t^8L_x^2}^2\Vert \widetilde{A}_1\Vert_{L_t^2L_x^{\infty}}\Vert P_{\mu}\bm{\psi}\Vert_{L_t^8L_x^2}^2 &\lesssim t^{\frac{1}{2} - \sqrt{\frac{\varepsilon}{2}}}\Vert \bm{\psi}(0)\Vert_{H_x^{\sigma - 1}}^2\lambda\lambda^{-2\sigma + 2}a_{\lambda}\mu^{-2\sigma + 2}a_{\mu}^2 \\ 
        &\lesssim \lambda\lambda^{-2\sigma + 2}a_{\lambda}\mu^{-2\sigma + 2}a_{\mu}^2,
    \end{aligned}
\end{equation*}
where we use H\"older inequality, energy estimates and Proposition \ref{B1}.
For the second term of $R_2$ we also use integration by parts and commutator estimates to handle $\Im(\overline{P_{\lambda, e_1}\psi_{\ell}^b}P_{\lambda, e_1}\mathcal{N}(\psi_{\ell}^b))$ when encountering the magnetic term $\sqrt{-1}A \cdot \nabla_x\psi_{\ell}^b$:
\begin{equation*}
    \begin{aligned}
            \int_{[0, 1]^d}\Im(\overline{P_{\lambda, e_1}\psi_{\ell}^b}P_{\lambda, e_1}(\sqrt{-1}\bm{\widetilde{A}} \cdot \nabla_x\psi_{\ell}^b)) \,{\rm d}x = &\int_{[0, 1]^d}\Re(\overline{P_{\lambda, e_1}\psi_{\ell}^b}\bm{\widetilde{A}} \cdot \nabla_xP_{\lambda, e_1}\psi_{\ell}^b) \,{\rm d}x \\ 
            &+ \int_{[0, 1]^d}\Re(\overline{P_{\lambda, e_1}\psi_{\ell}^b}[P_{\lambda, e_1}, \widetilde{\bm{A}}] \cdot \nabla_x\psi_{\ell}^b) \,{\rm d}x,
    \end{aligned}
\end{equation*}
and the first term on the right-hand side vanishes after integration by parts since $\nabla_x \cdot \bm{\widetilde{A}} = 0$.
This implies the second term of $R_2$ share similar bound with the first term, which in turn yields
\begin{equation*}
    R_2 \lesssim \lambda^2\lambda^{-2\sigma + 2}a_{\lambda}^2\mu^{-2\sigma + 2}a_{\mu}^2.
\end{equation*}

Moving to $R_3$, from the definition of the momentum bracket and the existence of magnetic term, we see that 
\begin{equation*}
    \{P_{\lambda, e_1}\psi_{\ell}^b, P_{\lambda, e_1}\mathcal{N}(\psi_\ell^b)\}_1^{\bm{C}} = \Re(P_{\lambda, e_1}\psi_\ell^b\overline{D_1^{\bm{C}}P_{\lambda, e_1}\mathcal{N}(\psi_\ell^b)} - P_{\lambda, e_1}\mathcal{N}(\psi_\ell^b)\overline{D_1^{\bm{C}}P_{\lambda, e_1}\psi_\ell^b})
\end{equation*}
invloves at most two derivatives at scale $\lambda$. Utilizing H\"older inequality, Bernstein inequality and energy estimates, we conclude
\begin{equation*}
    R_3 \lesssim \lambda^2\lambda^{-2\sigma + 2}a_{\lambda}^2\mu^{-2\sigma + 2}a_{\mu}^2.
\end{equation*}
Similarly, repeating the analysis above yields
\begin{equation*}
    R_1', R_2', R_3' \lesssim \lambda^2\lambda^{-2\sigma + 2}a_{\lambda}^2\mu^{-2\sigma + 2}a_{\mu}^2,
\end{equation*}
and we conclude that 
\begin{equation*}
    \sup_{a, b}\mathcal{R} \lesssim \lambda^{-2\sigma + 2}a_{\lambda}^2.
\end{equation*}
Moreover, since 
\begin{equation*}
    \begin{aligned}
    \mathcal{S} = R_4 + R_4' &= \left\Vert \Vert P_{\mu}\psi_k^a\Vert_{L_x^2}\Vert \partial_1P_{\lambda, e_1}\psi_{\ell}^b\Vert_{L_x^2}\right\Vert_{L_t^2}^2 - \left\Vert \Vert P_{\lambda, e_1}\psi_{\ell}^b\Vert_{L_x^2}\Vert \partial_1P_{\mu}\psi_k^a\Vert_{L_x^2}\right\Vert_{L_t^2}^2 \\
    &\quad + 2C_1\int_0^t\Vert P_{\mu}\psi_k^a\Vert_{L_x^2}^2\left(\int_{[0, 1]^{d}}\Im(\overline{P_{\lambda, e_1}\psi_{\ell}^b}\partial_1P_{\lambda, e_1}\psi_{\ell}^b) \,{\rm d}x\right) \,{\rm d}t \\
    &\quad - 2C_1\int_0^t\Vert P_{\lambda, e_1}\psi_{\ell}^b\Vert_{L_x^2}^2\left(\int_{[0, 1]^{d}}\Im(\overline{P_{\mu}\psi_k^a}\partial_1P_{\mu}\psi_k^a) \,{\rm d}x\right)\,{\rm d}t,
\end{aligned}
\end{equation*}
it follows from H\"older inequality, Bernstein inequality and energy estimates that 
\begin{equation*}
    \begin{aligned}
    \sup_{a, b}\mathcal{S} &\lesssim t^{1 - \frac{\varepsilon}{2}}((\lambda^2 + \mu^2)\lambda^{-2\sigma + 2}a_{\lambda}^2\mu^{-2\sigma + 2}a_{\mu}^2 + 2|C_1|(\lambda + \mu)\lambda^{-2\sigma + 2}a_{\lambda}^2\mu^{-2\sigma + 2}a_{\mu}^2) \\
    &\lesssim \lambda^2\lambda^{-2\sigma + 2}a_{\lambda}^2\mu^{-2\sigma + 2}a_{\mu}^2.
    \end{aligned}
\end{equation*}
Therefore, 
\begin{equation}\label{RHS}
    \sup_{a, b}(\mathcal{R} + \mathcal{S}) \leq \sup_{a, b}\mathcal{R} + \sup_{a, b}\mathcal{S} \lesssim \lambda^2\mu^{-2\sigma + 2}\lambda^{-2\sigma + 2}a_{\mu}^2a_{\lambda}^2,
\end{equation}
where $t \in [0, T^*]$ and $T^* = T^*(\Vert \bm{\psi}(0)\Vert_{H^{\sigma - 1}_x}) > 0$.
Combining \eqref{b1} with \eqref{LHS} and \eqref{RHS} yields
\begin{equation*}
    \sup_a\left\Vert \Vert P_{\mu}\psi_k^a\Vert_{L_{\widehat{x}_1}^2}\Vert P_{\lambda, e_1}\psi_{\ell}\Vert_{L_{\widehat{x}_1}^2}\right\Vert_{L_{t, x_1}^2} \lesssim \mu^{-\sigma + 1}\lambda^{-\sigma + 1}a_{\mu}a_{\lambda}.
\end{equation*}

\subsection{Interaction Morawetz estimates}\label{Morawetzx}

Suppose $\psi_k\ (k = 1, \cdots, d)$ is a smooth solution to \eqref{psi-e}. Without loss of generality, we may regard $\psi_k$ as periodic functions on $\mathbb{R}^d$ satisfying
\begin{equation*}
    \psi_k(t, x + 1) = \psi_k(t, x), \quad x \in \mathbb{R}^d.
\end{equation*}
Let $\chi$ be a nonnegative smooth function defined on $\mathbb{R}^3$, supported in $[-1, 1]^3$ and satisfying $0 \leq \chi \leq 1$, $\chi(x) = \chi(-x)$ and $\chi(x) = 1$ for $x \in [-1/2, 1/2]^3$.
Moreover, there exists $C > 0$ such that 
\begin{equation*}
    |\nabla\chi| \leq C.
\end{equation*}

We set $a(y) = |y|, y = (y_1, y_2, y_3) \in \mathbb{R}^3$. Simple comptations show
\begin{gather}
    \partial_ia = \frac{y_i}{|y|}, \\
    \label{51}\partial_i\partial_ja = \frac{1}{|y|}\left(\delta_{jk} - \frac{y_iy_j}{|y|^2}\right), \\
    \label{52}\Delta_ya = \frac{2}{|y|}, \\ 
    \label{53}\Delta_{y}^2a = -8\pi\delta(y),
\end{gather}
where $\delta_{jk}$ is the Kronecker delta and $\delta$ is the Dirac function. The interaction Morawetz potential is defined by
\begin{equation}\label{Mo}
\begin{aligned}
    &M(t) := -\int_{\mathbb{R}^3 \times \mathbb{R}^3}\left(\chi(y_1)\int_{\mathbb{T}^{d - 3}}\frac{1}{2}|P_{\lambda}\psi_k|^2(y_1) \,{\rm d}z\right)\\
    &\quad\quad\quad\quad\quad\quad\,\,\,\,\left(\chi(y_2)\int_{\mathbb{T}^{d - 3}}\Im(\overline{P_{\lambda}\psi_k}\nabla_{y_2}^{\bm{C}}P_{\lambda}\psi_k)(y_2) \,{\rm d}z\right)\cdot \nabla a(y_1 - y_2) \,{\rm d}y_1{\rm d}y_2.
\end{aligned}
\end{equation}
where $x = (y, z) \in \mathbb{T}^3 \times \mathbb{T}^{d - 3} = \mathbb{T}^d, \nabla_y^{\bm{C}} = (D_1^{\bm{C}}, D_2^{\bm{C}}, D_3^{\bm{C}})$.

\subsubsection{Viral-type identity}

We recall the following balance laws for $\psi_k$ (see Section \ref{Balance}):
\begin{equation}\label{1}
    \begin{aligned}
        \partial_t\int_{\mathbb{T}^{d - 3}}\frac{|P_{\lambda}\psi_k|^2}{2} \,{\rm d}z + \nabla_y\cdot\int_{\mathbb{T}^{d - 3}}\Im(\overline{P_{\lambda}\psi_k}\nabla_y^{\bm{C}}P_{\lambda}\psi_k) \,{\rm d}z = \int_{\mathbb{T}^{d - 3}}\Im(\overline{P_{\lambda}\psi_k}P_{\lambda}\mathcal{N}(\psi_k)) \,{\rm d}z
    \end{aligned}
\end{equation}
and ($j = 1, 2, 3$)
\begin{equation}\label{22}
    \begin{aligned}
        &\partial_t\int_{\mathbb{T}^{d - 3}}\Im(\overline{P_{\lambda}\psi_k}D_j^{\bm{C}}P_{\lambda}\psi_k) \,{\rm d}z + \partial_\ell\int_{\mathbb{T}^{d - 3}}\left(2\Re(\overline{D_j^{\bm{C}}P_{\lambda}\psi_k}D_\ell^{\bm{C}}P_{\lambda}\psi_k) - \frac{1}{2}\partial_j\partial_{\ell}|P_{\lambda}\psi_k|^2\right) \,{\rm d}z \\
        = &\int_{\mathbb{T}^{d - 3}}\{P_{\lambda}\mathcal{N}(\psi_k), P_{\lambda}\psi_k\}_j^{\bm{C}} \,{\rm d}z,
    \end{aligned}
\end{equation}
Inserting the spatial cutoff $\chi$ into \eqref{1}-\eqref{22}, we have 
\begin{equation}\label{1AA}
    \begin{aligned}
        &\partial_t\left(\chi\int_{\mathbb{T}^{d - 3}}\frac{|P_{\lambda}\psi_k|^2}{2} \,{\rm d}z\right) + \nabla_y\cdot\left(\chi\int_{\mathbb{T}^{d - 3}}\Im(\overline{P_{\lambda}\psi_k}\nabla_y^{\bm{C}}P_{\lambda}\psi_k) \,{\rm d}z\right) \\
        =\, &\nabla_y\chi \cdot \int_{\mathbb{T}^{d - 3}}\Im(\overline{P_{\lambda}\psi_k}\nabla_y^{\bm{C}}P_{\lambda}\psi_k) \,{\rm d}z + \chi\int_{\mathbb{T}^{d - 3}}\Im(\overline{P_{\lambda}\psi_k}P_{\lambda}\mathcal{N}(\psi_k)) \,{\rm d}z
    \end{aligned}
\end{equation}
and 
\begin{equation}\label{22AA}
    \begin{aligned}
        &\partial_t\int_{\mathbb{T}^{d - 3}}\Im(\overline{P_{\lambda}\psi_k}D_j^{\bm{C}}P_{\lambda}\psi_k) \,{\rm d}z + \partial_\ell\int_{\mathbb{T}^{d - 3}}\left(2\Re(\overline{D_j^{\bm{C}}P_{\lambda}\psi_k}D_\ell^{\bm{C}}P_{\lambda}\psi_k) - \frac{1}{2}\partial_j\partial_{\ell}|P_{\lambda}\psi_k|^2\right) \,{\rm d}z \\
        =&\,\partial_{\ell}\chi \int_{\mathbb{T}^{d - 3}}\left(2\Re(\overline{D_j^{\bm{C}}P_{\lambda}\psi_k}D_\ell^{\bm{C}}P_{\lambda}\psi_k) - \frac{1}{2}\partial_j\partial_{\ell}|u|^2\right) \,{\rm d}z + \chi\int_{\mathbb{T}^{d - 3}}\{P_{\lambda}\mathcal{N}(\psi_k), P_{\lambda}\psi_k\}_j^{\bm{C}} \,{\rm d}z.
    \end{aligned}
\end{equation}
Differentiating $M$ with respect to time $t$, by integration by parts and \eqref{1AA}-\eqref{22AA} we get
\begin{equation}\label{50}
        \begin{aligned}
            M'(t) = \mathcal{M} + \mathcal{M}_{{\rm error}},  
        \end{aligned}
\end{equation}
where 
\begin{equation*}
    \begin{aligned}
        &\mathcal{M} = -\int_{\mathbb{R}^3 \times \mathbb{R}^3}\left(\chi(y_1)\int_{\mathbb{T}^{d - 3}}\Im(\overline{P_{\lambda}\psi_k}D_{\ell}^{\bm{C}}P_{\lambda}\psi_k)(y_1) \,{\rm d}z\right)\\
        &\quad\quad\quad\quad\quad\,\,\,\,\left(\chi(y_2)\int_{\mathbb{T}^{d - 3}}\Im(\overline{P_{\lambda}\psi_k}D_{m}^{\bm{C}}P_{\lambda}\psi_k)(y_2) \,{\rm d}z\right)\partial_{\ell}\partial_ma(y_1 - y_2) \,{\rm d}y_1{\rm d}y_2 \\
        &\quad\quad\int_{\mathbb{R}^3 \times \mathbb{R}^3}\left(\chi(y_1)\int_{\mathbb{T}^{d - 3}}\frac{1}{2}|P_{\lambda}\psi_k|^2(y_1) \,{\rm d}z\right)\\
        &\quad\quad\quad\quad\,\,\,\,\left(\chi(y_2)\int_{\mathbb{T}^{d - 3}}\left(2\Re(\overline{D_{\ell}^{\bm{C}}P_{\lambda}\psi_k}D_{m}^{\bm{C}}P_{\lambda}\psi_k) - \frac{1}{2}\partial_{\ell}\partial_m|P_{\lambda}\psi_k|^2\right)(y_2) \,{\rm d}z\right)\partial_{\ell}\partial_ma(y_1 - y_2) \,{\rm d}y_1{\rm d}y_2
    \end{aligned}
\end{equation*}
and 
\begin{equation}\label{error}
\begin{aligned}
        &\mathcal{M}_{{\rm error}} \\
        = &-\int_{\mathbb{R}^3 \times \mathbb{R}^3}\left(\nabla_{y_1}\chi(y_1) \cdot \int_{\mathbb{T}^{d - 3}}\Im(\overline{P_{\lambda}\psi_k}\nabla_{y_1}^{\bm{C}}P_{\lambda}\psi_k)(y_1) \,{\rm d}z\right)\\
        &\quad\quad\quad\,\,\,\,\left(\chi(y_2)\int_{\mathbb{T}^{d - 3}}\Im(\overline{P_{\lambda}\psi_k}\nabla_{y_2}^{\bm{C}}P_{\lambda}\psi_k)(y_2) \,{\rm d}z\right) \cdot \nabla a(y_1 - y_2) \,{\rm d}y_1{\rm d}y_2 \\ 
        &-\int_{\mathbb{R}^3 \times \mathbb{R}^3}\left(\chi(y_1) \int_{\mathbb{T}^{d - 3}}\Im(\overline{P_{\lambda}\psi_k}P_{\lambda}\mathcal{N}(\psi_k))(y_1) \,{\rm d}z\right)\\
        &\quad\quad\quad\,\,\,\,\left(\chi(y_2)\int_{\mathbb{T}^{d - 3}}\Im(\overline{P_{\lambda}\psi_k}\nabla_{y_2}^{\bm{C}}P_{\lambda}\psi_k)(y_2) \,{\rm d}z\right) \cdot \nabla a(y_1 - y_2) \,{\rm d}y_1{\rm d}y_2 \\ 
        &-\int_{\mathbb{R}^3 \times \mathbb{R}^3}\left(\chi(y_1) \int_{\mathbb{T}^{d - 3}}|P_{\lambda}\psi_k|^2(y_1) \,{\rm d}z\right)\\
        &\quad\quad\quad\,\,\,\,\left(\partial_{\ell}\chi(y_2)\int_{\mathbb{T}^{d - 3}}\left(2\Re(\overline{D_m^{\bm{C}}P_{\lambda}\psi_k}D_{\ell}^{\bm{C}}P_{\lambda}\psi_k) - \frac{1}{2}\partial_{\ell}\partial_m|P_{\lambda}\psi_k|^2\right)(y_2) \,{\rm d}z\right)\partial_ma(y_1 - y_2) \,{\rm d}y_1{\rm d}y_2 \\
        &-\int_{\mathbb{R}^3 \times \mathbb{R}^3}\left(\chi(y_1) \int_{\mathbb{T}^{d - 3}}\frac{1}{2}|P_{\lambda}\psi_k|^2(y_1) \,{\rm d}z\right)\\
        &\quad\quad\quad\,\,\,\,\left(\chi(y_2)\int_{\mathbb{T}^{d - 3}}\{P_{\lambda}\mathcal{N}(\psi_k), P_{\lambda}\psi_k\}^{\bm{C}}(y_2) \,{\rm d}z\right) \cdot \nabla a(y_1 - y_2) \,{\rm d}y_1{\rm d}y_2.
\end{aligned}
\end{equation}
Furthermore, we define 
\begin{equation*}
    F_{\ell} = F_{\ell}(\psi_k) = u(y_1, z_1)D_{\ell}^{\bm{C}}u(y_2, z_2) - u(y_2, z_2)D_{\ell}^{\bm{C}}u(y_1, z_1), \qquad (y_1, z_1), (y_2, z_2) \in \mathbb{T}^3 \times \mathbb{T}^{d - 3}.
\end{equation*}
Observing the identity 
\begin{equation*}
    \begin{aligned}
    &\Re(F_\ell\overline{F_m}) \\
    = &\,|u(y_1, z_1)|^2\Re(\overline{D_{\ell}^{\bm{C}}P_{\lambda}\psi_k}D_{\ell}^{\bm{C}}P_{\lambda}\psi_k)(y_2, z_2) + |u(y_2, z_2)|^2\Re(\overline{D_{\ell}^{\bm{C}}P_{\lambda}\psi_k}D_{\ell}^{\bm{C}}P_{\lambda}\psi_k)(y_1, z_1) \\ 
    & -2\Im(\overline{P_{\lambda}\psi_k}D_{\ell}^{\bm{C}}P_{\lambda}\psi_k)(y_1, z_1)\Im(\overline{P_{\lambda}\psi_k}D_m^{\bm{C}}P_{\lambda}\psi_k)(y_2, z_2) - \frac{1}{2}\partial_{\ell}|P_{\lambda}\psi_k(y_1, z_1)|^2\partial_m|P_{\lambda}\psi_k(y_2, z_2)|^2
    \end{aligned}
\end{equation*}
and using the symmtery of $\partial_{\ell}\partial_ma$, we get 
\begin{equation*}
    \begin{aligned}
            &\mathcal{M} \\
            = &\frac{1}{2}\int_{\mathbb{R}^3 \times \mathbb{R}^3}\partial_{\ell}\partial_ma(y_1 - y_2)\chi(y_1)\chi(y_2)\left(\int_{\mathbb{T}^{d - 3}}F_{\ell}\right)\overline{\left(\int_{\mathbb{T}^{d - 3}}F_m\right)} \,{\rm d}y_1{\rm d}y_2 \\ 
            &+ \frac{1}{4}\int_{\mathbb{R}^3 \times \mathbb{R}^3}\partial_{\ell}\partial_ma(y_1 - y_2)\left(\chi(y_1)\int_{\mathbb{T}^{d - 3}}\partial_{\ell}|P_{\lambda}\psi_k(y_1)|^2\right)\left(\chi(y_2)\int_{\mathbb{T}^{d - 3}}\partial_m|P_{\lambda}\psi_k(y_2)|^2\right) \,{\rm d}y_1{\rm d}y_2 \\ 
            &- \frac{1}{4}\int_{\mathbb{R}^3 \times \mathbb{R}^3}\partial_{\ell}\partial_ma(y_1 - y_2)\left(\chi(y_1)\int_{\mathbb{T}^{d - 3}}|P_{\lambda}\psi_k(y_1)|^2\right)\left(\chi(y_2)\int_{\mathbb{T}^{d - 3}}\partial_{\ell}\partial_m|P_{\lambda}\psi_k(y_2)|^2\right) \,{\rm d}y_1{\rm d}y_2.
    \end{aligned}
\end{equation*}
But since $a$ is convex, we have 
\begin{equation*}
    \int_{\mathbb{R}^3 \times \mathbb{R}^3}\partial_{\ell}\partial_ma(y_1 - y_2)\chi(y_1)\chi(y_2)\left(\int_{\mathbb{T}^{d - 3}}F_{\ell}\right)\overline{\left(\int_{\mathbb{T}^{d - 3}}F_m\right)} \,{\rm d}y_1{\rm d}y_2 \geq 0.
\end{equation*}
Thus, from integration by parts and \eqref{53} we have  
\begin{equation*}
    \begin{aligned}
        &\mathcal{M} \\
        \geq &\frac{1}{4}\int_{\mathbb{R}^3 \times \mathbb{R}^3}\partial_{\ell}\partial_ma(y_1 - y_2)\left(\chi(y_1)\int_{\mathbb{T}^{d - 3}}\partial_{\ell}|P_{\lambda}\psi_k(y_1)|^2\right)\left(\chi(y_2)\int_{\mathbb{T}^{d - 3}}\partial_m|P_{\lambda}\psi_k(y_2)|^2\right) \,{\rm d}y_1{\rm d}y_2 \\ 
        &- \frac{1}{4}\int_{\mathbb{R}^3 \times \mathbb{R}^3}\partial_{\ell}\partial_ma(y_1 - y_2)\left(\chi(y_1)\int_{\mathbb{T}^{d - 3}}|P_{\lambda}\psi_k(y_1)|^2\right)\left(\chi(y_2)\int_{\mathbb{T}^{d - 3}}\partial_{\ell}\partial_m|P_{\lambda}\psi_k(y_2)|^2\right) \,{\rm d}y_1{\rm d}y_2 \\ 
        = &\frac{1}{2}\int_{\mathbb{R}^3 \times \mathbb{R}^3}(-\Delta\Delta a)(y_1 - y_2)\left(\chi(y_1)\int_{\mathbb{T}^{d - 3}}|P_{\lambda}\psi_k(y_1)|^2\right)\left(\chi(y_2)\int_{\mathbb{T}^{d - 3}}|P_{\lambda}\psi_k(y_2)|^2\right) \,{\rm d}y_1{\rm d}y_2 \\ 
        = &4\pi\int_{\mathbb{R}^3}\left(\chi\int_{\mathbb{T}^{d - 3}}|P_{\lambda}\psi_k|^2 \,{\rm d}z\right)^2 \,{\rm d}y.
    \end{aligned}
\end{equation*}
We conclude that 
\begin{proposition}
    Suppose $d \geq 3$ and $\psi_k$ is a smooth solution to \eqref{psi-e}. Then
    \begin{equation}\label{63}
        \int_0^t\int_{\mathbb{R}^3}\left(\chi\int_{\mathbb{T}^{d - 3}}|P_{\lambda}\psi_k|^2 \,{\rm d}z\right)^2 \,{\rm d}y{\rm d}t \lesssim |M(0) - M(t)| + \int_0^t|\mathcal{M}_{{\rm error}}| \,{\rm d}t,
    \end{equation}
    where $M$ is the interaction Morawetz potential \eqref{Mo} and $\mathcal{M}_{{\rm error}}$ is defined by \eqref{error}.
\end{proposition}

\subsubsection{The \texorpdfstring{$L_{t, y}^4L_z^4$}{} bound}\label{Morawetz}

We first note that, from the definition of $\chi$ and the periodicity of $\psi_k$, we have
\begin{equation}\label{70}
\int_0^t\int_{\mathbb{R}^3}\left(\chi\int_{\mathbb{T}^{d - 3}}|P_{\lambda}\psi_k|^2 \,{\rm d}z\right)^2 \,{\rm d}y{\rm d}t \geq \Vert P_{\lambda}\psi_k\Vert_{L_{t, y}^4L_z^2}^4.
\end{equation}
For the right-hand side of \eqref{63}, by H\"older inequality 
\begin{equation*}
    \begin{aligned}
        |M(0) - M(t)| \lesssim \sup_t\Vert P_{\lambda}\psi_k\Vert_{L_x^2}^2 \cdot \sup_t\left\Vert\int_{\mathbb{R}^3}\left(\chi\int_{\mathbb{T}^{d - 3}}\Im(\overline{P_{\lambda}\psi_k}\nabla^{\bm{C}}_{y_2}P_{\lambda}\psi_k) \,{\rm d}z \right)\nabla a(y_1 - y_2) \,{\rm d}y_2\right\Vert_{L_{y_1}^{\infty}}.
    \end{aligned}
\end{equation*}
By energy estimates 
\begin{equation*}
    \sup_t\Vert P_{\lambda}\psi_k\Vert_{L_x^2}^2 \lesssim \lambda^{-2\sigma + 2}a_{\lambda}^2.
\end{equation*}
Since $a$ is a homogenous function of degree $1$, $\nabla a$ is of degree $0$ and it can be intepreted as kernel for appropriate multipliers, which act like composition of operators $\nabla_y^{-3}$ and Riesz type operators $R_y$.
Thus, by Bernstein inequality and energy estimates 
\begin{equation*}
    \begin{aligned}
        &\sup_t\left\Vert\int_{\mathbb{R}^3}\left(\chi\int_{\mathbb{T}^{d - 3}}\Im(\overline{P_{\lambda}\psi_k}\nabla^{\bm{C}}_{y_2}P_{\lambda}\psi_k) \,{\rm d}z \right)\nabla a(y_1 - y_2) \,{\rm d}y_2\right\Vert_{L_{y_1}^{\infty}} \\ 
        \lesssim &\,\lambda^{-3}\lambda^3\left\Vert\chi\int_{\mathbb{T}^{d - 3}}\Im(\overline{P_{\lambda}\psi_k}\nabla^{\bm{C}}_{y_2}P_{\lambda}\psi_k) \,{\rm d}z \,{\rm d}y_2\right\Vert_{L_y^1} \\ 
        \lesssim &\,\Vert \Im(\overline{P_{\lambda}\psi_k}\nabla_{\bm{C}}P_{\lambda}\psi_k)\Vert_{L_x^1} \\ 
        \lesssim &\,\lambda\lambda^{-2\sigma + 2}a_{\lambda}^2.
    \end{aligned}
\end{equation*}
We conclude
\begin{equation*}
    |M(0) - M(t)| \lesssim \lambda^{-4\sigma + 5}a_{\lambda}^4.
\end{equation*} 
Controlling 
\begin{equation*}
    \int_0^t|\mathcal{M}_{{\rm error}}| \,{\rm d}t,
\end{equation*}
we first consider the term
\begin{equation*}
    \begin{aligned}
                \mathcal{J} := &\int_0^t\Bigg|\int_{\mathbb{R}^3 \times \mathbb{R}^3}\left(\chi(y_1)\int_{\mathbb{T}^{d - 3}}\Im(\overline{P_{\lambda}\psi_k}P_{\lambda}\mathcal{N}(\psi_k))(y_1) \,{\rm d}z\right)\\
        &\quad\quad\quad\quad\,\,\,\,\left(\chi(y_2)\int_{\mathbb{T}^{d - 3}}\Im(\overline{P_{\lambda}\psi_k}\nabla_{y_2}^{\bm{C}}P_{\lambda}\psi_k)(y_2) \,{\rm d}z\right) \cdot \nabla a(y_1 - y_2) \,{\rm d}y_1{\rm d}y_2\Bigg| \,{\rm d}t
    \end{aligned}
\end{equation*}
When dealing with the magentic term involved in $\Im(\overline{P_{\lambda}\psi_k}P_{\lambda}\mathcal{N}(\psi_k))$, the derivative may fall on $\nabla a$ after integartion by parts.
Note $\nabla^2a$ acts like the operator $\nabla_y^{-2}$. Then by Bernstein inequality 
\begin{equation*}
    \begin{aligned}
        &\sup_t\left\Vert\int_{\mathbb{R}^3}\left(\chi\int_{\mathbb{T}^{d - 3}}\Im(\overline{P_{\lambda}\psi_k}\nabla^{\bm{C}}_{y_2}P_{\lambda}\psi_k) \,{\rm d}z \right)\nabla^2 a(y_1 - y_2) \,{\rm d}y_2\right\Vert_{L_{y_1}^3} \\ 
        \lesssim &\,\lambda^{-2}\lambda^2\sup_t\left\Vert\chi\int_{\mathbb{T}^{d - 3}}\Im(\overline{P_{\lambda}\psi_k}\nabla^{\bm{C}}_{y_2}P_{\lambda}\psi_k) \,{\rm d}z \,{\rm d}y_2\right\Vert_{L_y^1} \\ 
        \lesssim &\,\sup_t\Vert \Im(\overline{P_{\lambda}\psi_k}\nabla_{\bm{C}}P_{\lambda}\psi_k)\Vert_{L_x^1} \\ 
        \lesssim &\,\lambda\lambda^{-2\sigma + 2}a_{\lambda}^2.
    \end{aligned}
\end{equation*}
It follows from H\"older inequality that 
\begin{equation*}
    \begin{aligned}
                &\int_0^t\Bigg|\int_{\mathbb{R}^3 \times \mathbb{R}^3}\left(\chi(y_1)\int_{\mathbb{T}^{d - 3}}\Im(\overline{P_{\lambda}\psi_k}P_{\lambda}\mathcal{N}(\psi_k))(y_1) \,{\rm d}z\right)\\
        &\quad\quad\quad\quad\,\,\,\,\left(\chi(y_2)\int_{\mathbb{T}^{d - 3}}\Im(\overline{P_{\lambda}\psi_k}\nabla_{y_2}^{\bm{C}}P_{\lambda}\psi_k)(y_2) \,{\rm d}z\right) \cdot \nabla a(y_1 - y_2) \,{\rm d}y_1{\rm d}y_2\Bigg| \,{\rm d}t \\
        \lesssim & \left\Vert\chi(y_1)\int_{\mathbb{T}^{d - 3}}\Im(\overline{P_{\lambda}\psi_k}P_{\lambda}(\sqrt{-1}\bm{\widetilde{A}} \cdot \nabla_x\psi_k))(y_1) \,{\rm d}z\right\Vert_{L_t^1L_{y_1}^{\frac{3}{2}}} \\ 
        &\cdot \sup_t\left\Vert\int_{\mathbb{R}^3}\left(\chi\int_{\mathbb{T}^{d - 3}}\Im(\overline{P_{\lambda}\psi_k}\nabla^{\bm{C}}_{y_2}P_{\lambda}\psi_k) \,{\rm d}z \right)\nabla^2 a(y_1 - y_2) \,{\rm d}y_2\right\Vert_{L_{y_1}^3} \\ 
        \lesssim &\lambda^{-2\sigma + 3}a_{\lambda}^2\left\Vert\chi(y_1)\int_{\mathbb{T}^{d - 3}}\Im(\overline{P_{\lambda}\psi_k}P_{\lambda}(\sqrt{-1}\bm{\widetilde{A}} \cdot \nabla_x\psi_k))(y_1) \,{\rm d}z\right\Vert_{L_t^1L_{y_1}^{\frac{3}{2}}},
    \end{aligned}
\end{equation*}
and by the anlaysis in Section \ref{L2L2}, 
\begin{equation*}
    \begin{aligned}
        \left\Vert\chi(y_1)\int_{\mathbb{T}^{d - 3}}\Im(\overline{P_{\lambda}\psi_k}P_{\lambda}(\sqrt{-1}\bm{\widetilde{A}} \cdot \nabla_x\psi_k))(y_1) \,{\rm d}z\right\Vert_{L_t^1L_{y_1}^{\frac{3}{2}}} &\lesssim \left\Vert \Vert \widetilde{\bm{A}}\Vert_{L_z^{\infty}}\Vert P_{\lambda}\psi_k\Vert_{L_z^2}^2 \right\Vert_{L_t^1L_y^{\frac{3}{2}}} \\ 
        &\lesssim \Vert \widetilde{\bm{A}}\Vert_{L_t^2L_x^{\infty}}\Vert P_{\lambda}\psi_k\Vert_{L_t^4L_y^3L_z^2}^2 \\ 
        &\lesssim t^{\frac{1}{2}}\lambda\Vert \widetilde{\bm{A}}\Vert_{L_t^2L_x^{\infty}}\Vert P_{\lambda}\psi_k\Vert_{L_t^{\infty}L_x^2}^2 \\ 
        &\lesssim \lambda^{-2\sigma + 3}a_{\lambda}^2.
    \end{aligned}
\end{equation*}
where we use Proposition \ref{B1} and Bernstein inequality. 
If the derivative never falls on $\nabla a$, $\mathcal{J}$ can be then bounded following a analogous approach as for $|M(0) - M(t)|$, using H\"older inequality $L^1 \cdot L^{\infty} \rightarrow L^1$ and Bernstein inequality $\lambda^3L_y^1 \rightarrow L_y^\infty$.
Therefore,
\begin{equation*}
    \mathcal{J} \lesssim \lambda^{-4\sigma + 6}a_{\lambda}^2.
\end{equation*}
Together with the fact $\chi$ and $\nabla\chi$ are uniformly bounded, the remanining terms can be controlled similarly, yielding the same bound as $\mathcal{J}$. 
We conclude
\begin{equation*}
    \int_0^t|\mathcal{M}_{{\rm error}}| \,{\rm d}t \lesssim \lambda^{-4\sigma + 6}a_{\lambda}^2.
\end{equation*}
Combining the analysis above leads to the interaction Morawetz estimates for $P_{\lambda}\psi_k$
\begin{equation*}
    \Vert P_{\lambda}\psi_k\Vert_{L_{t, y}^4L_z^2}^4 \lesssim \lambda^{-4\sigma + 6}a_{\lambda}^4.
\end{equation*}
Namely, 
\begin{equation*}
    \Vert P_{\lambda}\psi_k\Vert_{L_{t, y}^4L_z^2} \lesssim \lambda^{-\sigma + \frac{3}{2}}a_{\lambda}.
\end{equation*}

\subsection{Weak Lipschitz dependence}\label{34}

Suppose $u_0$ and $u_1$ are the solutions to \eqref{LL} with initial data $\phi_0, \phi_1 \in H^\sigma(\mathbb{T}^d; \mathbb{S}^2)$ with $\sigma > d/2 + 1/2$, respectively.
Without loss of generality we assume $\phi_0 \neq -\phi_1$.
Let $v = v(\zeta, t, x)$ satisfy ($0 < \zeta < 1$)
\begin{equation}\label{SMFa}
    \begin{cases}
        \partial_tv = v \times \Delta_xv, \\ 
        \displaystyle t = 0\colon v = \frac{(1 - \zeta)\phi_0 + \zeta\phi_1}{|(1 - \zeta)\phi_0 + \zeta\phi_1|} \in \mathbb{S}^2.
    \end{cases}
\end{equation}
It follows that $v(0, t, x) = u_0(t, x), v(1, t, x) = u_1(t, x)$, and 
\begin{equation}\label{Sta}
    u_1(t) - u_0(t) = \int_0^1\partial_{\zeta} v(t) \,{\rm d}\zeta.
\end{equation}
To prove the weak Lipschitz dependence \eqref{WLD}, from \eqref{Sta} it suffices to estimate the $H_x^{\sigma - 1}$ norm of $\partial_{\zeta}v$.

We define 
\begin{gather*}
    \psi_{\zeta} := \langle \partial_{\zeta}v, e_1\rangle + \sqrt{-1}\langle \partial_{\zeta}v, e_2\rangle, \\
    A_{\zeta} := \langle \partial_{\zeta}e_1, e_2\rangle
\end{gather*}
and the covariant derivative $D_{\zeta} := \partial_{\zeta} + \sqrt{-1}A_{\zeta}$.
It is clear that $D_{\zeta}$ share similar properties as $D_{\alpha}^{\bm{A}} = \partial_{\alpha} + \sqrt{-1}A_{\alpha}, \alpha = 0, 1, \cdots, d$ (see Section \ref{MSMF}). 
In particular, we have
\begin{equation*}
    \begin{aligned}
        D_0^{\bm{A}}\psi_{\zeta} = D_{\zeta}\psi_0 = \sqrt{-1}D_{\zeta}D_\ell^{\bm{A}}\psi_{\ell} &= \sqrt{-1}D_\ell^{\bm{A}} D_{\zeta}\psi_{\ell} + [D_{\zeta}, D_\ell^{\bm{A}}]\psi_{\ell} \\ 
        &= \sqrt{-1}D_\ell^{\bm{A}} D_{\ell}^{\bm{A}}\psi_\zeta - \Im(\psi_\ell\overline{\psi_{\zeta}})\psi_{\ell},
    \end{aligned}
\end{equation*}
which is equivalent to
\begin{equation}\label{psizeta}
\begin{aligned}
    (\sqrt{-1}\partial_t + \Delta_{\bm{C}})\psi_\zeta = -2\sqrt{-1}\widetilde{\bm{A}} \cdot \nabla_x\psi_\zeta + (A_0 + |\widetilde{\bm{A}}|^2 + 2\bm{C} \cdot \widetilde{\bm{A}})\psi_\zeta - \sqrt{-1}\bm{\psi} \cdot \Im(\psi_\zeta\overline{\bm{\psi}}).
\end{aligned}
\end{equation}
Based on \eqref{psizeta}, we now employ the same procedure in Sections \ref{EnergyEs}-\ref{Morawetzx}, combining energy estimates, the mass and momentum balance laws for $P_{\lambda}\psi_{\zeta}$, div-curl lemma and interation Morawetz estimates with bootstrap arguments, to obtain the $H_x^{\sigma - 1}$ control for $\psi_{\zeta}$.
We conclude (the details are omitted)
\begin{equation*}
    \Vert\psi_{\zeta}(t)\Vert_{H_x^{\sigma - 1}} \lesssim \Vert \psi_{\zeta}(0)\Vert_{H^{\sigma - 1}_x},
\end{equation*}
and hence
\begin{equation*}
    \Vert \partial_{\zeta}v(t)\Vert_{H_x^{\sigma - 1}} \lesssim \Vert \partial_{\zeta}v(0)\Vert_{H_x^{\sigma - 1}},
\end{equation*}
Together with Proposition \ref{AC}, we obtain 
\begin{equation*}
    \Vert u_1(t) - u_0(t)\Vert_{H_x^{\sigma - 1}} \leq \int_0^1\Vert \partial_{\zeta}v(t)\Vert_{H_x^{\sigma - 1}} \,{\rm d}\zeta \lesssim \Vert \partial_{\zeta}v(0)\Vert_{H_x^{\sigma - 1}} \lesssim \Vert \phi_1 - \phi_0\Vert_{H_x^{\sigma - 1}}.
\end{equation*}

\section{Proof of Theorem \ref{LWP1}}\label{3}

\subsection{Energy estimates}\label{31}

By integrating with respect to time from $0$ to $t$ on both sides of \eqref{E1}, we have 
\begin{align}\label{energy}
    \int_{\mathbb{T}^d}\frac{1}{2}|\nabla_xP_{\lambda}u(t)|^2 \,{\rm d}x = \int_{\mathbb{T}^d}\frac{1}{2}|\nabla_xP_{\lambda}\phi|^2 \,{\rm d}x + \mathcal{I}_1 + \mathcal{I}_2 + \mathcal{I}_3,
\end{align}
where 
\begin{align*}
    \mathcal{I}_1 &= \int_0^t\int_{\mathbb{T}^d}\langle \mathcal{R}_{\lambda}, \partial_{\ell}P_{\lambda}u\rangle \,{\rm d}x{\rm d}t \\ 
    &= \int_0^t\int_{\mathbb{T}^d}\langle [P_{\lambda}\partial_{\ell}, J(u)]D_k\partial_ku, \partial_{\ell}P_{\lambda}u \rangle \,{\rm d}x{\rm d}t + \int_0^t\int_{\mathbb{T}^d}\langle J(u)[P_{\lambda}\partial_{\ell}, D_k]\partial_ku, \partial_{\ell}P_{\lambda}u \rangle \,{\rm d}x{\rm d}t \\
    &=: \mathcal{I}_{1a} + \mathcal{I}_{1b}, \\
    \mathcal{I}_2 &= \int_0^t\int_{\mathbb{T}^d}\langle \partial_k\partial_{\ell}P_{\lambda}u, (\partial_kJ(u))\partial_{\ell}P_{\lambda}u\rangle \,{\rm d}x{\rm d}t, \\ 
    \mathcal{I}_3 &= \int_0^t\int_{\mathbb{T}^d}\left([P_{\lambda}\Delta_x, n_i(u)]\partial_{\ell}u\right)\left([P_{\lambda}, n_i(u)]J(u)\partial_{\ell}u + \langle [P_{\lambda}, J(u)]\partial_{\ell}u, n_i(u) \rangle\right) \,{\rm d}x{\rm d}t.
\end{align*}

\subsubsection{Estimates for \texorpdfstring{$\mathcal{I}_{1b}$}{}}\label{I1b}

By the definition of commutator, 
\begin{align*}
    [P_{\lambda}\partial_{\ell}, D_k]\partial_ku &= P_{\lambda}\partial_{\ell}D_k\partial_ku - D_kP_{\lambda}\partial_{\ell}\partial_ku \\
    &= P_{\lambda}\partial_{\ell}\partial_k\partial_ku - P_{\lambda}\partial_{\ell}(\langle \partial_k\partial_ku, n_i(u)\rangle n_i(u)) \\ 
    &\quad - \partial_kP_{\lambda}\partial_{\ell}\partial_ku + \langle \partial_{\ell}\partial_k\partial_kP_{\lambda}u, n_i(u)\rangle n_i(u) \\ 
    &= -P_{\lambda}\partial_{\ell}(\langle \partial_k\partial_ku, n_i(u)\rangle n_i(u)) + \langle \partial_{\ell}\partial_k\partial_kP_{\lambda}u, n_i(u)\rangle n_i(u) \\
    &= -[P_{\lambda}\partial_{\ell}, n_i(u)]\langle \partial_k\partial_ku, n_i(u)\rangle - P_{\lambda}\partial_{\ell}(\langle \partial_k\partial_ku, n_i(u)\rangle)n_i(u) \\
    &\quad + \langle \partial_{\ell}\partial_k\partial_kP_{\lambda}u, n_i(u)\rangle n_i(u) \\
    &= [P_{\lambda}\partial_{\ell}, n_i(u)]\langle \partial_ku, \partial_kn_i(u)\rangle - ([P_{\lambda}\partial_{\ell}, n_i(u)]\Delta_x u)n_i(u).
\end{align*}
Therefore, 
\begin{equation}
\begin{aligned}
    \mathcal{I}_{1b} &= \int_0^t\int_{\mathbb{T}^d}\langle J(u)[P_{\lambda}\partial_{\ell}, n_i(u)] \langle \partial_ku, \partial_kn_i(u)\rangle, \partial_{\ell}P_{\lambda}u\rangle \,{\rm d}x{\rm d}t \\
    &\quad - \int_0^t\int_{\mathbb{T}^d}([P_{\lambda}\partial_{\ell}, n_i(u)]\Delta_x u)\langle \partial_{\ell}P_{\lambda}u, J(u)n_i(u)\rangle \,{\rm d}x{\rm d}t \\
    &= -\int_0^t\int_{\mathbb{T}^d}\langle [P_{\lambda}\partial_{\ell}, n_i(u)] \langle \partial_ku, \partial_kn_i(u)\rangle, J(u)\partial_{\ell}P_{\lambda}u\rangle \,{\rm d}x{\rm d}t \\
    &\quad + \int_0^t\int_{\mathbb{T}^d}([P_{\lambda}\partial_{\ell}, n_i(u)]\Delta_x u)\langle J(u)\partial_{\ell}P_{\lambda}u, n_i(u)\rangle \,{\rm d}x{\rm d}t.
\end{aligned}
\end{equation}
Since  
\begin{equation*}
    J(u)\partial_{\ell}P_{\lambda}u = P_{\lambda}J(u)\partial_{\ell}u - [P_{\lambda}, J(u)]\partial_{\ell}u
\end{equation*}
and 
\begin{equation*}
    \langle P_{\lambda}J(u)\partial_{\ell}u, n_i(u)\rangle = -[P_{\lambda}, n_i(u)]J(u)\partial_{\ell}u,
\end{equation*}
it follows that 
\begin{equation*}
    \mathcal{I}_{1b} = \mathcal{I}_{1b}^{\alpha} + \mathcal{I}_{1b}^{\beta} + \mathcal{I}_{1b}^{\gamma},
\end{equation*}
where 
\begin{align*}
    \mathcal{I}_{1b}^{\alpha} &= -\int_0^t\int_{\mathbb{T}^d}\langle [P_{\lambda}\partial_{\ell}, n_i(u)] \langle \partial_ku, \partial_kn_i(u)\rangle, J(u)\partial_{\ell}P_{\lambda}u\rangle \,{\rm d}x{\rm d}t, \\ 
    \mathcal{I}_{1b}^{\beta} &= - \int_0^t\int_{\mathbb{T}^d}([P_{\lambda}\partial_{\ell}, n_i(u)]\Delta_x u)([P_{\lambda}, n_i(u)]J(u)\partial_{\ell}u) \,{\rm d}x{\rm d}t, \\ 
    \mathcal{I}_{1b}^{\gamma} &= - \int_0^t\int_{\mathbb{T}^d}([P_{\lambda}\partial_{\ell}, n_i(u)]\Delta_x u)\langle [P_{\lambda}, J(u)]\partial_{\ell}u, n_i(u)\rangle \,{\rm d}x{\rm d}t.
\end{align*}

We first estimate $\mathcal{I}_{1b}^{\alpha}$. Note 
\begin{equation*}
    [P_{\lambda}\partial_{\ell}, n_i(u)] = P_{\lambda}[\partial_{\ell}, n_i(u)] + [P_{\lambda}, n_i(u)]\partial_{\ell}.
\end{equation*}
It follows that       
\begin{align*}
            \mathcal{I}_{1b}^{\alpha} &= -\int_0^t\int_{\mathbb{T}^d}\langle P_{\lambda}[\partial_{\ell}, n_i(u)] \langle \partial_ku, \partial_kn_i(u)\rangle, J(u)\partial_{\ell}P_{\lambda}u\rangle \,{\rm d}x{\rm d}t \\ 
            &\quad -\int_0^t\int_{\mathbb{T}^d}\langle [P_{\lambda}, n_i(u)] \partial_{\ell}\langle \partial_ku, \partial_kn_i(u)\rangle, J(u)\partial_{\ell}P_{\lambda}u\rangle \,{\rm d}x{\rm d}t \\ 
            &= -\int_0^t\int_{\mathbb{T}^d}\langle P_{\lambda}(\langle \partial_ku, \partial_kn_i(u)\rangle\partial_{\ell}n_i(u)), J(u)\partial_{\ell}P_{\lambda}u\rangle \,{\rm d}x{\rm d}t \\ 
            &\quad -\int_0^t\int_{\mathbb{T}^d}\langle [P_{\lambda}, n_i(u)] \partial_{\ell}\langle \partial_ku, \partial_kn_i(u)\rangle, J(u)\partial_{\ell}P_{\lambda}u\rangle \,{\rm d}x{\rm d}t \\
            &=: \mathcal{J}_1^{\alpha} + \mathcal{J}_2^{\alpha}.
\end{align*}

We remark that, by treating the unit normals $n_i(u)$ as functions of $u$, we can estimate related quantities (such as $\Vert P_{\lambda}n_i(u)\Vert_{H_x^{\sigma}}$) via the paralinearization theorem (Theorem \ref{para}) and the nonlinear Bernstein inequality (Theorem \ref{Nbern}). The detailed estimates are deferred to Appendix \ref{AA}.

\begin{proposition}\label{prop3}
    There holds 
    \begin{equation*}
        \mathcal{J}_1^{\alpha}, \mathcal{J}_2^{\alpha} \lesssim t^{\frac{1}{4}}\Vert \phi\Vert_{H_x^\sigma}^4\lambda^{-2\sigma + 2}a_{\lambda}^2,
    \end{equation*}
    and hence 
    \begin{equation*}
        \mathcal{I}_{1b}^{\alpha} \lesssim t^{\frac{1}{4}}\Vert \phi\Vert_{H_x^\sigma}^4\lambda^{-2\sigma + 2}a_{\lambda}^2.
    \end{equation*}
    \begin{proof}
        For $\mathcal{J}_2^{\alpha}$, by Lemma \ref{BONYB} 
        \begin{align*}
            \mathcal{J}_2^{\alpha} \lesssim \mathcal{J}_{{\rm LH}}^{\alpha} + \mathcal{J}_{{\rm HL}}^{\alpha} + \mathcal{J}_{{\rm HH}}^{\alpha},
        \end{align*}
        where\footnote{
            In the following the notation $\nabla_x$ will be used interchangeably to represent the gradient and the Fourier multiplier with the symbol $\xi \mapsto |\xi|$, depenging on the context.
        }
        \begin{align*}
            \mathcal{J}_{{\rm LH}}^{\alpha} &:=\lambda^{-1}\sum_{\lambda_1 \ll \lambda_2 \sim \lambda}\int_0^t\int_{\mathbb{T}^d}|\langle L(\nabla_xP_{\lambda_1}n_i(u), P_{\lambda_2}\partial_{\ell}\langle \partial_ku, \partial_kn_i(u)\rangle), J(u)\partial_{\ell}P_{\lambda}u\rangle| \,{\rm d}x{\rm d}t, \\ 
            \mathcal{J}_{{\rm HL}}^{\alpha} &:= \sum_{\lambda_2 \ll \lambda_1 \sim \lambda}\int_0^t\int_{\mathbb{T}^d}|\langle L(P_{\lambda_1}n_i(u), P_{\lambda_2}\partial_{\ell}\langle \partial_ku, \partial_kn_i(u)\rangle), \partial_kn_i(u)\rangle, J(u)\partial_{\ell}P_{\lambda}u\rangle| \,{\rm d}x{\rm d}t, \\  
            \mathcal{J}_{{\rm HH}}^{\alpha} &:= \sum_{\lambda_1 \sim \lambda_2 \gtrsim \lambda}\int_0^t\int_{\mathbb{T}^d}|\langle L(P_{\lambda_1}n_i(u), P_{\lambda_2}\partial_{\ell}\langle \partial_ku, \partial_kn_i(u)\rangle), \partial_kn_i(u)\rangle, J(u)\partial_{\ell}P_{\lambda}u\rangle| \,{\rm d}x{\rm d}t.
        \end{align*}

        For the term $\mathcal{J}_{{\rm LH}}^{\alpha}$, by H\"older inequality, Theorem \ref{Comu}, bootstrap assumptions \eqref{boot1}-\eqref{boot2} and Proposition \ref{prop2}
            \begin{align*}
            \mathcal{J}_{{\rm LH}}^{\alpha} &\lesssim t^{\frac{1}{2}}\sum_{\lambda_1 \ll \lambda}\Vert L(P_{\lambda_1}n_i(u), P_{\lambda}\langle \nabla_xu, \nabla_xn_i(u)\rangle)\Vert_{L_{t, x}^2}\Vert \nabla_xP_{\lambda}u\Vert_{L_t^{\infty}L_x^2} \\
            &\lesssim t^{\frac{1}{2}}\Vert P_{\lambda}\langle \nabla_xu, \nabla_xn_i(u)\rangle\Vert_{L_t^3L_x^2}\Vert \nabla_xP_{\lambda}u\Vert_{L_t^{\infty}L_x^2}\sum_{\lambda_1 \ll \lambda}\Vert P_{\lambda_1}\nabla_xn_i(u)\Vert_{L_t^6L_x^{\infty}} \\ 
            &\lesssim t^{\frac{1}{2} - \frac{1}{24} - \frac{1}{12}}\Vert \phi\Vert_{H_x^\sigma}^2\lambda^{-\sigma + 1}a_{\lambda}\Vert P_{\lambda}\langle \nabla_xu, \nabla_xn_i(u)\rangle\Vert_{L_t^3L_x^2} \\ 
            &= t^{\frac{3}{8}}\Vert \phi\Vert_{H_x^\sigma}^2\lambda^{-\sigma + 1}a_{\lambda}\Vert P_{\lambda}\langle \nabla_xu, \nabla_xn_i(u)\rangle\Vert_{L_t^3L_x^2}.
            \end{align*}
        Since 
            \begin{align*}
                            \Vert P_{\lambda}\langle \nabla_xu, \nabla_xn_i(u)\rangle\Vert_{L_t^3L_x^2} &\lesssim \left(\sum_{\lambda_1 \ll \lambda_2 \sim \lambda} + \sum_{\lambda_2 \ll \lambda_1 \sim \lambda} + \sum_{\lambda_1 \sim \lambda_2 \gtrsim \lambda}\right)\Vert \langle P_{\lambda_1}\nabla_xu, P_{\lambda_2}\nabla_xn_i(u)\rangle\Vert_{L_t^3L_x^2} \\
            &=: \,\mathcal{K}_{\mathrm{LH}} + \mathcal{K}_{\mathrm{HL}} + \mathcal{K}_{\mathrm{HH}},
            \end{align*}
        it follows from H\"older inequality, Proposition \ref{prop1}-\ref{prop2} and Theorem \ref{Nbern} that  
        \begin{equation*}
            \begin{aligned}
                             \mathcal{K}_{\mathrm{LH}} &\lesssim \Vert P_{\lambda}\nabla_xn_i(u)\Vert_{L_t^6L_x^2}\sum_{\lambda_1 \ll \lambda}\Vert P_{\lambda_1}\nabla_xu\Vert_{L_t^6L_x^{\infty}} \\
            &\lesssim t^{-\frac{1}{12} - \frac{1}{24}}\Vert \phi\Vert_{H_x^\sigma}\lambda^{-\sigma + 1}a_{\lambda}\sum_{\lambda_1 \ll \lambda}\lambda_1^{-\sigma + \frac{d}{2} + \frac{5}{6}}a_{\lambda_1} \\ 
            &\lesssim t^{-\frac{1}{8}}\Vert \phi\Vert_{H_x^\sigma}^2\lambda^{-\sigma + 1}a_{\lambda}, \\ 
            \mathcal{K}_{\mathrm{HL}} &\lesssim t^{\frac{1}{6}}\Vert \nabla_xP_{\lambda}u\Vert_{L_t^{\infty}L_x^2}\sum_{\lambda_2 \ll \lambda}\Vert P_{\mu}\nabla_xn_i(u)\Vert_{L_t^6L_x^{\infty}} \\
            &\lesssim t^{\frac{1}{6} - \frac{1}{24} - \frac{1}{12}}\Vert \phi\Vert_{H_x^\sigma}^2\lambda^{-\sigma + 1}a_{\lambda} \\ 
            &= t^{\frac{1}{24}}\Vert \phi\Vert_{H_x^\sigma}^2\lambda^{-\sigma + 1}a_{\lambda}
            \end{aligned}
        \end{equation*}
        and 
        \begin{align*} 
                                                            \mathcal{K}_{\mathrm{HH}} &\lesssim t^{\frac{1}{6}}\sum_{\lambda_1 \gtrsim \lambda}\Vert P_{\lambda_1}\nabla_xu\Vert_{L_t^{\infty}L_x^2}\Vert P_{\lambda_1}\nabla_xn_i(u)\Vert_{L_t^6L_x^{\infty}} \\ 
            &\lesssim t^{\frac{1}{6} - \frac{1}{24}}\sum_{\lambda_1 \gtrsim \lambda}\lambda_1^{-\sigma + 1}a_{\lambda_1}\Vert \nabla_xu\Vert_{L_t^6L_x^{\infty}} \\
            &\lesssim t^{\frac{1}{6} - \frac{1}{24} - \frac{1}{24}}\Vert \phi\Vert_{H_x^\sigma}\sum_{\lambda_1 \gtrsim \lambda}\lambda_1^{-\sigma + 1}a_{\lambda_1}\left(\frac{\lambda_1}{\lambda}\right)^{\delta} \\
            &\lesssim t^{\frac{1}{12}}\Vert \phi\Vert_{H_x^\sigma}\lambda^{-\sigma + 1}a_{\lambda}\sum_{\lambda_1 \gtrsim \lambda}\left(\frac{\lambda_1}{\lambda}\right)^{-\sigma + 1 + \delta} \\
            &\lesssim t^{\frac{1}{12}}\Vert \phi\Vert_{H_x^\sigma}\lambda^{-\sigma + 1}a_{\lambda}.
        \end{align*}
        Note $\delta = \delta(\sigma, d)$ can be chosen to be sufficiently small such that $-\sigma + 1 + \delta < 0$.
        Therefore, 
        \begin{equation*}
            \Vert P_{\lambda}\langle \nabla_xu, \nabla_xn_i(u)\rangle\Vert_{L_t^3L_x^2} \lesssim t^{-\frac{1}{8}}\Vert \phi\Vert_{H_x^\sigma}^2\lambda^{-\sigma + 1}a_{\lambda},
        \end{equation*}
        and it follows that  
        \begin{equation*}
            \mathcal{J}_{{\rm LH}}^{\alpha} \lesssim t^{\frac{3}{8} - \frac{1}{8}}\Vert \phi\Vert_{H_x^\sigma}^{2 + 2}\lambda^{-\sigma + 1 - s + 1}a_{\lambda}^2 = t^{\frac{1}{4}}\Vert \phi\Vert_{H_x^\sigma}^4\lambda^{-2\sigma + 2}a_{\lambda}^2.
        \end{equation*}

        Similarly, we have 
            \begin{align*}
            \mathcal{J}_{\alpha}^{{\rm HL}} &\lesssim t^{\frac{2}{3}}\Vert P_{\lambda}n_i(u)\Vert_{L_t^6L_x^2}\Vert \nabla_xP_{\lambda}u\Vert_{L_t^{\infty}L_x^2}\sum_{\lambda_2 \ll \lambda}\Vert P_{\lambda_2}\nabla_x\langle \nabla_xu, \nabla_xn_i(u)\rangle\Vert_{L_t^6L_x^{\infty}} \\
            &\lesssim t^{\frac{2}{3}}\Vert \nabla_xP_{\lambda}n_i(u)\Vert_{L_t^6L_x^2}\Vert \nabla_xP_{\lambda}u\Vert_{L_t^{\infty}L_x^2}\sum_{\lambda_2 \ll \lambda}\lambda^{-1}\lambda_2\Vert P_{\lambda_2}\langle \nabla_xu, \nabla_xn_i(u)\rangle\Vert_{L_t^6L_x^{\infty}}  \\ 
            &\lesssim t^{\frac{2}{3}}(t^{-\frac{1}{12}}\Vert \phi\Vert_{H_x^\sigma}\lambda^{-\sigma + 1}a_{\lambda}) \cdot (t^{-\frac{1}{24}}\lambda^{-\sigma + 1}a_{\lambda}) \cdot (t^{-\frac{1}{8}}\Vert \phi\Vert_{H_x^\sigma}^3) \\
            &= t^{\frac{5}{12}}\Vert \phi\Vert_{H_x^\sigma}^4\lambda^{-2\sigma + 2}a_{\lambda}^2
            \end{align*}
            and 
            \begin{align*}
            \mathcal{J}_{\alpha}^{{\rm HH}} &\lesssim t^{\frac{1}{2}}\Vert \nabla_xP_{\lambda}u\Vert_{L_t^{\infty}L_x^2}\sum_{\lambda_1 \gtrsim \lambda}\Vert P_{\lambda_1}n_i(u)\Vert_{L_t^6L_x^{\infty}}\Vert P_{\lambda_1}\nabla_x\langle \nabla_xu, \nabla_xn_i(u)\rangle\Vert_{L_t^3L_x^2} \\
            &\lesssim t^{\frac{1}{2}}\Vert \nabla_xP_{\lambda}u\Vert_{L_t^{\infty}L_x^2}\sum_{\lambda_1 \gtrsim \lambda}\lambda_1\lambda^{-1}\Vert P_{\lambda_1}\nabla_xn_i(u)\Vert_{L_t^6L_x^{\infty}}\Vert P_{\lambda_1}\langle \nabla_xu, \nabla_xn_i(u)\rangle\Vert_{L_t^3L_x^2} \\ 
            &\lesssim t^{\frac{1}{4}}\Vert \phi\Vert_{H_x^\sigma}^4\lambda^{-2\sigma + 2}a_{\lambda}^2.
            \end{align*}
        In conclusion, 
        \begin{equation*}
            \mathcal{J}_2^{\alpha} \lesssim (t^{\frac{1}{4}} + t^{\frac{5}{12}})\Vert \phi\Vert_{H_x^\sigma}^4\lambda^{-2\sigma + 2}a_{\lambda}^2 \lesssim t^{\frac{1}{4}}\Vert \phi\Vert_{H_x^\sigma}^4\lambda^{-2\sigma + 2}a_{\lambda}^2.
        \end{equation*}
        
        For $\mathcal{J}_1^{\alpha}$, in view of Bony decomposition we see that that the term 
        \begin{equation*}
            P_{\lambda}(\langle \partial_ku, \partial_kn_i(u)\rangle\partial_{\ell}n_i(u))
        \end{equation*}
        involves one derivative at scale $\lambda$, which is exactly the same number of derivatives as in the term
        \begin{equation*}
            [P_{\lambda}, n_i(u)]\partial_{\ell}\langle \partial_ku, \partial_kn_i(u)\rangle,
        \end{equation*}
        where the commutator term should be expanded in the form of \eqref{BONY}.
        Therefore, repeating the same analysis above we can also obtain 
        \begin{equation*}
            \mathcal{J}_1^{\alpha} \lesssim t^{\frac{1}{4}}\Vert \phi\Vert_{H_x^\sigma}^4\lambda^{-2\sigma + 2}a_{\lambda}^2.
        \end{equation*}
        This completes the proof.
    \end{proof}
\end{proposition}

For the estimates of $\mathcal{I}_{1b}^{\beta}$, we decompose the commutator term as follows:
\begin{align*}
            [P_{\lambda}\partial_{\ell}, n_i(u)]\Delta_x u &= P_{\lambda}[\partial_{\ell}, n_i(u)]\Delta_x u + [P_{\lambda}, n_i(u)]\partial_{\ell}\Delta_x u \\ 
    &= P_{\lambda}\langle \Delta_x u, \partial_{\ell}n_i(u)\rangle + [P_{\lambda}, n_i(u)]\partial_{\ell}\Delta_x u,
\end{align*}
and we see that this term contains two derivatives at scale $\lambda$. The term $[P_{\lambda}, n_i(u)]J(u)\partial_{\ell}u$ costs no derivative.
Following the arguments in the proof of Proposition \ref{prop3}, we have 
\begin{equation*}
    \mathcal{I}_{1b}^{\beta} \lesssim t^{\frac{1}{4}}\Vert \phi\Vert_{H_x^\sigma}^4\lambda^{-2\sigma + 2}a_{\lambda}^2.
\end{equation*}
Analogously,
\begin{equation*}
    \mathcal{I}_{1b}^{\gamma} \lesssim t^{\frac{1}{4}}\Vert \phi\Vert_{H_x^\sigma}^4\lambda^{-2\sigma + 2}a_{\lambda}^2.
\end{equation*}

To conclude, 
\begin{equation*}
    \mathcal{I}_{1b} \lesssim t^{\frac{1}{4}}\Vert \phi\Vert_{H_x^\sigma}^4\lambda^{-2\sigma + 2}a_{\lambda}^2.
\end{equation*}

\subsubsection{Estimates for \texorpdfstring{$\mathcal{I}_{1a}$}{}}\label{I1a}

The compatiblility of $J$ and connection yields
\begin{equation*}
    D_{\ell}J(u) = 0, 
\end{equation*}
and hence 
\begin{equation}\label{J}
    \partial_{\ell}J_m(u) = \langle \partial_{\ell}J_m(u), n_i(u)\rangle n_i(u), \qquad  m = 1, \cdots, N,
\end{equation}
where $\partial_{\ell}J = (\partial_{\ell}J_1, \cdots, \partial_{\ell}J_N) \in \mathbb{R}^{N \times N}, \partial_{\ell}J_m(u) \in \mathbb{R}^{N \times 1}$. 

By Plancherel theorem, we have 
\begin{equation*}
    \mathcal{I}_{1a} = -\int_0^t\int_{\mathbb{T}^d}\langle D_k\partial_ku, [P_{\lambda}\partial_{\ell}, J(u)]\partial_{\ell}P_{\lambda}u\rangle \,{\rm d}x{\rm d}t,
\end{equation*}
and 
    \begin{align*}
            [P_{\lambda}\partial_{\ell}, J(u)]\partial_{\ell}P_{\lambda}u &= P_{\lambda}\partial_{\ell}J(u)\partial_{\ell}P_{\lambda}u - J(u)P_{\lambda}\partial_{\ell}\partial_{\ell}P_{\lambda}u \\ 
    &= P_{\lambda}(\partial_{\ell}J(u))\partial_{\ell}P_{\lambda}u + P_{\lambda}J(u)\partial_{\ell}\partial_{\ell}P_{\lambda}u - J(u)P_{\lambda}\partial_{\ell}\partial_{\ell}P_{\lambda}u. 
    \end{align*}
From \eqref{J}, 
\begin{equation}\label{J1}
    (\partial_{\ell}J(u))\partial_{\ell}P_{\lambda}u = (\langle \partial_{\ell}J_m(u), n_i(u)\rangle\partial_{\ell}P_{\lambda}u_m)n_i(u), \qquad m = 1, \cdots, N,
\end{equation}
where $\partial_{\ell}P_{\lambda}u = (\partial_{\ell}P_{\lambda}u_1, \cdots, \partial_{\ell}P_{\lambda}u_N)^T \in \mathbb{R}^{N \times 1}, \partial_{\ell}P_{\lambda}u_m \in \mathbb{R}$.
It follows that 
\begin{align*}
            \mathcal{I}_{1a} &= -\int_0^t\int_{\mathbb{T}^d}\langle D_k\partial_ku, P_{\lambda}(\langle \partial_{\ell}J_m(u), n_i(u)\rangle\partial_{\ell}P_{\lambda}u_mn_i(u)) \rangle \,{\rm d}x{\rm d}t \\
    &\quad - \int_0^t\int_{\mathbb{T}^d}\langle D_k\partial_ku, P_{\lambda}J(u)\partial_{\ell}\partial_{\ell}P_{\lambda}u - J(u)P_{\lambda}\partial_{\ell}\partial_{\ell}P_{\lambda}u\rangle \,{\rm d}x{\rm d}t \\
    &=: \mathcal{I}_{1a}^{\alpha} + \mathcal{I}_{1a}^{\beta}.
\end{align*}

For $\mathcal{I}_{1a}^{\alpha}$, 
\begin{align*}
    \mathcal{I}_{1a}^{\alpha} &= -\int_0^t\int_{\mathbb{T}^d}\langle P_{\lambda}D_k\partial_ku, \langle \partial_{\ell}J_m(u), n_i(u)\rangle\partial_{\ell}P_{\lambda}u_mn_i(u) \rangle \,{\rm d}x{\rm d}t \\ 
    &= -\int_0^t\int_{\mathbb{T}^d}(\langle \partial_{\ell}J_m(u), n_i(u)\rangle\partial_{\ell}P_{\lambda}u_m)\langle P_{\lambda}D_k\partial_ku, n_i(u)\rangle \,{\rm d}x{\rm d}t \\
    &= \int_0^t\int_{\mathbb{T}^d}(\langle \partial_{\ell}J_m(u), n_i(u)\rangle\partial_{\ell}P_{\lambda}u_m)([P_{\lambda}, n_i(u)]D_k\partial_ku) \,{\rm d}x{\rm d}t.
\end{align*}
Note 
\begin{equation*}
    [P_{\lambda}, n_i(u)]D_k\partial_ku = [P_{\lambda}, n_i(u)]\Delta_x u - [P_{\lambda}, n_i(u)](\langle \Delta_x u, n_j(u)\rangle n_j(u)),
\end{equation*}
and 
\begin{align*}
    &[P_{\lambda}, n_i(u)](\langle \Delta_x u, n_j(u)\rangle n_j(u)) \\
    = &P_{\lambda}(\langle \Delta_x u, n_j(u)\rangle\langle n_i(u), n_j(u)\rangle) - \langle P_{\lambda}(\langle \Delta_x u, n_j(u)\rangle n_j(u)), n_i(u)\rangle \\ 
    = &P_{\lambda}\langle \Delta_x u, n_i(u)\rangle - \langle P_{\lambda}(\langle \Delta_x u, n_j(u)\rangle n_j(u)), n_i(u)\rangle \\
    = &P_{\lambda}\langle \Delta_x u, n_i(u)\rangle - \langle [P_{\lambda}, n_j(u)]\langle \Delta_x u, n_j(u)\rangle, n_i(u)\rangle - \langle n_j(u), n_i(u)\rangle P_{\lambda}\langle \Delta_x u, n_j(u)\rangle\\ 
    = &-\langle [P_{\lambda}, n_j(u)]\langle \Delta_x u, n_j(u)\rangle, n_i(u)\rangle \\ 
    = &\langle [P_{\lambda}, n_j(u)]\langle \partial_ku, \partial_kn_j(u)\rangle, n_i(u)\rangle,
\end{align*}
where we use the property 
\begin{equation}\label{normal}
    \langle n_i(u), n_j(u)\rangle = \delta_{ij}.
\end{equation}
Therefore,
\begin{equation}\label{z}
\begin{aligned}
    &[P_{\lambda}, n_i(u)]D_k\partial_ku \\
    = &\,[P_{\lambda}, n_i(u)]\Delta_x u + P_{\lambda}\langle \partial_ku, \partial_kn_i(u)\rangle - \langle [P_{\lambda}, n_j(u)]\langle \partial_ku, \partial_kn_j(u)\rangle, n_i(u)\rangle,
\end{aligned}
\end{equation}
and it follows that this term contains one derivative at scale $\lambda$.

Estimating $\mathcal{I}_{1a}^{\alpha}$, by H\"older inequality we have 
\begin{align*}
    \mathcal{I}_{1a}^{\alpha} \lesssim \Vert \nabla_xJ(u)\Vert_{L_t^3L_x^{\infty}}\Vert \nabla_xP_{\lambda}u\Vert_{L_t^{\infty}L_x^2}\Vert [P_{\lambda}, n_i(u)]D_k\partial_ku\Vert_{L_t^{\frac{3}{2}}L_x^2}.
\end{align*}
On the one hand, from the assumptions on $J$ and \eqref{boot1}-\eqref{boot2}
\begin{align*}
            \Vert \nabla_xJ(u)\Vert_{L_t^3L_x^{\infty}} = \Vert J'(u)\nabla_xu\Vert_{L_t^3L_x^{\infty}} &\leq \Vert J'(u)\Vert_{L_t^6L_x^{\infty}}\Vert \nabla_xu\Vert_{L_t^6L_x^{\infty}} \\ 
    &\lesssim \Vert u\Vert_{L_t^6L_x^{\infty}}\Vert \nabla_xu\Vert_{L_t^6L_x^{\infty}} \\ 
    &\lesssim t^{-\frac{1}{12}}\Vert \phi\Vert_{H_x^\sigma}^2,
\end{align*}
On the other hand, combining \eqref{z} with Proposition \ref{prop1}-\ref{prop2}, we have (see also the proof of Proposition \ref{prop3})
\begin{align*}
    \Vert [P_{\lambda}, n_i(u)]D_k\partial_ku \Vert_{L_t^{\frac{3}{2}}L_x^2} \lesssim t^{\frac{3}{2} - \frac{1}{6} - \frac{1}{24} - \frac{1}{12}}\Vert \phi\Vert_{H_x^\sigma}^2\lambda^{-\sigma + 1}a_{\lambda} = t^{\frac{3}{8}}\Vert \phi\Vert_{H_x^\sigma}^2\lambda^{-\sigma + 1}a_{\lambda}.
\end{align*}
Consequently, 
\begin{equation*}
    \mathcal{I}_{1a}^{\alpha} \lesssim (t^{-\frac{1}{12}}\Vert \phi\Vert_{H_x^\sigma}^2) \cdot (t^{-\frac{1}{24}}\lambda^{-\sigma + 1}a_{\lambda}) \cdot (t^{\frac{3}{8}}\Vert \phi\Vert_{H_x^\sigma}^2\lambda^{-\sigma + 1}a_{\lambda}) = t^{\frac{1}{4}}\Vert \phi\Vert_{H_x^\sigma}^4\lambda^{-2\sigma + 2}a_{\lambda}^2.
\end{equation*}

For $\mathcal{I}_{1a}^{\beta}$, let $k_{\lambda}$ be the convolution kernel of $P_{\lambda}$, namely, $P_{\lambda}u = k_{\lambda} \ast u$. 
We have 
\begin{align*}
            &P_{\lambda}J(u)\partial_{\ell}\partial_{\ell}P_{\lambda}u - J(u)P_{\lambda}\partial_{\ell}\partial_{\ell}P_{\lambda}u \\
    = &(J(u)\partial_{\ell}\partial_{\ell}P_{\lambda}u) \ast k_{\lambda} - J(u)(\partial_{\ell}\partial_{\ell}P_{\lambda}u \ast k_{\lambda}) \\ 
    = &\int_{\mathbb{T}^d}J(u(x - y))\partial_{\ell}\partial_{\ell}P_{\lambda}u(x - y)k_{\lambda}(y) \,{\rm d}y - \int_{\mathbb{T}^d}J(u(x))\partial_{\ell}\partial_{\ell}P_{\lambda}u(x - y)k_{\lambda}(y) \,{\rm d}y \\ 
    = &\int_{\mathbb{T}^d}(J(u(x - y)) - J(u(x)))\partial_{\ell}\partial_{\ell}P_{\lambda}u(x - y)k_{\lambda}(y) \,{\rm d}y \\ 
    = &-\int_{\mathbb{T}^d}\partial_{\ell}\partial_{\ell}P_{\lambda}u(x - y)k_{\lambda}(y) \,{\rm d}y\int_0^1y \cdot \nabla J(u(x - s_1y)) \,{\rm d}s_1 \\ 
    = &\int_{\mathbb{T}^d}\partial_{\ell}\partial_{\ell}P_{\lambda}u(x - y)k_{\lambda}(y) \,{\rm d}y\int_0^1y \cdot \nabla J(u(x - s_1y)) \,{\rm d}s_1.
\end{align*}
Inserting \eqref{J} into the equality above, we obtain 
    \begin{align*}
    &P_{\lambda}J(u)\partial_{\ell}\partial_{\ell}P_{\lambda}u - J(u)P_{\lambda}\partial_{\ell}\partial_{\ell}P_{\lambda}u \\
    = &\int_{\mathbb{T}^d}\partial_{\ell}\partial_{\ell}P_{\lambda}u_m(x - y)k_{\lambda}(y) \,{\rm d}y\int_0^1y_a(\langle \partial_kJ_m, n_i\rangle n_i(u(x - s_1y))) \,{\rm d}s_1 \\ 
    = &\langle \partial_aJ_m(u), n_i(u)\rangle\int_{\mathbb{T}^d}\int_0^1\partial_{\ell}\partial_{\ell}P_{\lambda}u_m(x - y)y_ak_{\lambda}(y)n_i(u(x - s_1y)) \,{\rm d}s_1\,{\rm d}y.
\end{align*}
Since $\langle D_k\partial_ku, n_i(u)\rangle = 0$, it follows that 
\begin{align*}
            &\mathcal{I}_{1a}^{\beta} \\
    = &\langle \partial_aJ_m(u), n_i(u)\rangle\\
    &\cdot \left\langle D_k\partial_ku, \int_{\mathbb{T}^d}\int_0^1\partial_{\ell}\partial_{\ell}P_{\lambda}u_m(x - y)y_ak_{\lambda}(y)(n_i(u(x - s_1y)) - n_i(u(x))) \,{\rm d}s_1\,{\rm d}y\right\rangle \\
    = &\langle \partial_aJ_m(u), n_i(u)\rangle\\
    &\cdot \left\langle D_k\partial_ku, \int_{\mathbb{T}^d}\int_0^1\int_0^1\partial_{\ell}\partial_{\ell}P_{\lambda}u_m(x - y)s_1y_ay_bk_{\lambda}(y)\partial_bn_i(u(x - s_1s_2y)) \,{\rm d}s_2\,{\rm d}s_1\,{\rm d}y\right\rangle.
\end{align*}
Note
\begin{equation*}
    \int_{\mathbb{T}^d}|y|^2|k_{\lambda}(y)| \,{\rm d}y \lesssim \lambda^{-2}.
\end{equation*}
This $O(\lambda^{-2})$ bound, consequently, allows us to remove two derivatives at scale $\lambda$ from the term $\partial_{\ell}\partial_{\ell}P_{\lambda}u_m$.
Moreover, 
\begin{equation*}
    D_k\partial_ku = \Delta_x u - \langle \Delta_x u, n_i(u)\rangle n_i(u) = \Delta_x u + \underbrace{\langle \partial_ku, \partial_kn_i(u)\rangle n_i(u)}_{{\rm l.o.t}},
\end{equation*}
we see that $D_k\partial_ku$ is a second-order term. Therefore, $\mathcal{I}_{1a}^{\beta}$ can be estimated in a similar manner to $\mathcal{I}_{1a}^{\alpha}$, leading to
\begin{equation*}
    \mathcal{I}_{1a}^{\beta} \lesssim t^{\frac{1}{4}}\Vert \phi\Vert_{H_x^\sigma}^4\lambda^{-2\sigma + 2}a_{\lambda}^2.
\end{equation*}

To conclude, 
\begin{equation*}
    \mathcal{I}_{1a} \lesssim t^{\frac{1}{4}}\Vert \phi\Vert_{H_x^\sigma}^4\lambda^{-2\sigma + 2}a_{\lambda}^2.
\end{equation*}

\subsubsection{Estimates for \texorpdfstring{$\mathcal{I}_2$}{}}

We rewrite $\mathcal{I}_2$ as 
    \begin{align*}
            \mathcal{I}_2 &= \int_0^t\int_{\mathbb{T}^d}\langle \partial_k\partial_{\ell}P_{\lambda}u, \langle \partial_kJ_m(u), n_i(u)\rangle\partial_{\ell}P_{\lambda}u_mn_i(u)\rangle \,{\rm d}x{\rm d}t \\ 
    &= -\int_0^t\int_{\mathbb{T}^d}(\langle \partial_kJ_m(u), n_i(u)\rangle \partial_{\ell}P_{\lambda}u_m)([P_{\lambda}\partial_{\ell}, n_i(u)]\partial_ku) \,{\rm d}x{\rm d}t,
    \end{align*}
and 
    \begin{align*}
            [P_{\lambda}\partial_{\ell}, n_i(u)]\partial_ku &= P_{\lambda}[\partial_{\ell}, n_i(u)]\partial_ku + [P_{\lambda}, n_i(u)]\partial_{\ell}\partial_ku \\ 
    &= P_{\lambda}\langle \partial_ku, \partial_{\ell}u\rangle + [P_{\lambda}, n_i(u)]\partial_{\ell}\partial_ku.
    \end{align*}
Again, we see that the terms $\langle \partial_kJ_m(u), n_i(u)\rangle \partial_{\ell}P_{\lambda}u_m$ and $[P_{\lambda}\partial_{\ell}, n_i(u)]\partial_ku$ contain both only one derivative at scale $\lambda$.
Applying the same techniques when controlling $\mathcal{I}_{1b}^{\alpha}$ (in Section \ref{I1b}) and $\mathcal{I}_{1a}^{\alpha}$ (in Section \ref{I1a}), we obtain 
\begin{equation*}
    \mathcal{I}_2 \lesssim t^{\frac{1}{4}}\Vert \phi\Vert_{H_x^\sigma}^4\lambda^{-2\sigma + 2}a_{\lambda}^2.
\end{equation*}

\subsubsection{Estimates for \texorpdfstring{$\mathcal{I}_3$}{}}

Note 
\begin{align*}
            [P_{\lambda}\Delta, n_i(u)]\partial_{\ell}u &= P_{\lambda}\partial_k[\partial_k, n_i(u)]\partial_{\ell}u + [P_{\lambda}\partial_k, n_i(u)]\partial_k\partial_{\ell}u \\ 
    &= P_{\lambda}\partial_k\langle \partial_{\ell}u, \partial_kn_i(u)\rangle + P_{\lambda}[\partial_k, n_i(u)]\partial_k\partial_{\ell}u + [P_{\lambda}, n_i(u)]\Delta_x\partial_{\ell}u \\ 
    &= P_{\lambda}\partial_k\langle \partial_{\ell}u, \partial_kn_i(u)\rangle + P_{\lambda}\langle \partial_k\partial_{\ell}u, \partial_kn_i(u)\rangle + [P_{\lambda}, n_i(u)]\Delta_x\partial_{\ell}u.
    \end{align*}
We see that this commutator term contains two derivatives at scale $\lambda$. Since the terms 
\begin{equation*}
    [P_{\lambda}, n_i(u)]J(u)\partial_{\ell}u, \quad \langle [P_{\lambda}, J(u)]\partial_{\ell}u, n_i(u)\rangle
\end{equation*}
have no derivative, we conclude 
\begin{equation*}
    \mathcal{I}_3 \lesssim t^{\frac{1}{4}}\Vert \phi\Vert_{H_x^\sigma}^4\lambda^{-2\sigma + 2}a_{\lambda}^2.
\end{equation*}

\subsubsection{The \texorpdfstring{$L_t^{\infty}L_x^2$}{} bound}

We have proved that 
\begin{equation}
    \mathcal{I}_1, \mathcal{I}_2, \mathcal{I}_3 \lesssim t^{\frac{1}{4}}\Vert \phi\Vert_{H_x^\sigma}^4\lambda^{-2\sigma + 2}a_{\lambda}^2.
\end{equation}
By \eqref{energy}, 
\begin{equation}
    \Vert \nabla_xP_{\lambda}u(t)\Vert_{L_x^2}^2 \lesssim \Vert \nabla_xP_{\lambda}\phi\Vert_{L_x^2}^2 + t^{\frac{1}{4}}\Vert \phi\Vert_{H_x^\sigma}^4\lambda^{-2\sigma + 2}a_{\lambda}^2 \lesssim (1 + t^{\frac{1}{4}}\Vert \phi\Vert_{H_x^\sigma}^4)\lambda^{-2\sigma + 2}a_{\lambda}^2,
\end{equation}
and hence we can choose $T^* > 0$, depending only on $\Vert \phi\Vert_{H_x^\sigma}$, such that  
\begin{equation}
    \Vert \nabla_xP_{\lambda}u(t)\Vert_{L_x^2} \lesssim \lambda^{-\sigma + 1}a_{\lambda}, \qquad \forall t \in [0, T^*].
\end{equation}

\subsection{\texorpdfstring{$L_t^6L_x^{\infty}$}{} \emph{a priori} estimate}\label{Linfty}

Applying div-curl lemma \eqref{div2} to the system consisting of \eqref{dc11}-\eqref{dc22} yields
\begin{equation}\label{divcurl}
    {\rm LHS} \lesssim {\rm RHS},
\end{equation}
where 
\begin{align*}
    {\rm LHS} &:= L_1 + L_2, \\
    {\rm RHS} &:= R_1 \cdot R_2 + R_3
\end{align*}
and 
\begin{align*}
    L_1 &= \int_0^t\int_{\mathbb{T}}\left(\int_{\mathbb{T}^{d - 1}}\frac{1}{2}|\nabla_xP_{\lambda}u|^2 \,{\rm d}\widehat{x}_1\right)\left(\int_{\mathbb{T}^{d - 1}}\left(|\nabla_x\partial_1P_{\lambda}u|^2 - \langle \partial_{\ell} P_{\lambda}u, \partial_1^2\partial_{\ell}P_{\lambda}u\rangle\right) \,{\rm d}\widehat{x}_1\right)\,{\rm d}x_1{\rm d}t,\\ 
    L_2 &= -\int_0^t\int_{\mathbb{T}}\left(\int_{\mathbb{T}^{d - 1}}\langle \partial_1\partial_{\ell}P_{\lambda}u, J(u)\partial_{\ell}P_{\lambda}u\rangle \,{\rm d}\widehat{x}_1\right)^2\,{\rm d}x_1{\rm d}t,\\
    R_1 &= \int_{\mathbb{T}^d}\frac{1}{2}|\nabla_xP_{\lambda}\phi|^2 \,{\rm d}x + \sup_t\int_{\mathbb{T}^d}\frac{1}{2}|\nabla_xP_{\lambda}u(t)|^2 \,{\rm d}x + \int_0^t\int_{\mathbb{T}}|\mathcal{G}_{\lambda}| \,{\rm d}x_1{\rm d}t, \\
    R_2 &= \int_{\mathbb{T}^d}|\langle \partial_1\partial_{\ell}P_{\lambda}\phi, J(u)\partial_{\ell}P_{\lambda}\phi\rangle| \,{\rm d}x + \sup_t\int_{\mathbb{T}^d}|\langle \partial_1\partial_{\ell}P_{\lambda}u(t), J(u)\partial_{\ell}P_{\lambda}u(t)\rangle| \,{\rm d}x + \int_0^t\int_{\mathbb{T}}|\mathcal{H}_{\lambda}| \,{\rm d}x_1{\rm d}t, \\ 
    R_3 &= \int_0^t\left(\int_{\mathbb{T}^d}\frac{1}{2}|\nabla_xP_{\lambda}u|^2 \,{\rm d}x\right)\left(\int_{\mathbb{T}^d}\left(|\nabla_x\partial_1P_{\lambda}u|^2 - \langle \partial_{\ell} P_{\lambda}u, \partial_1^2\partial_{\ell}P_{\lambda}u\rangle\right) \,{\rm d}x\right) \,{\rm d}t \\ 
    &\quad- \int_0^t\left(\int_{\mathbb{T}^d}\langle \partial_1\partial_{\ell}P_{\lambda}u, J(u)\partial_{\ell}P_{\lambda}u\rangle \,{\rm d}x\right)^2 \,{\rm d}t.
\end{align*}

For ${\rm LHS}$, we first note
\begin{equation*}
    \langle \partial_{\ell} P_{\lambda}u, \partial_1^2\partial_{\ell}P_{\lambda}u\rangle = \partial_1^2\frac{1}{2}|\nabla_xP_{\lambda}u|^2 - |\nabla_x\partial_1P_{\lambda}u|^2.
\end{equation*}
Then integration by parts yields 
    \begin{align*}
            L_1 &= \int_0^t\int_{\mathbb{T}}\left(\partial_1\int_{\mathbb{T}^{d - 1}}\frac{1}{2}|\nabla_xP_{\lambda}u|^2 \,{\rm d}\widehat{x}_1\right)^2 \,{\rm d}x_1{\rm d}t \\
    &\quad + \int_0^t\int_{\mathbb{T}}\left(\int_{\mathbb{T}^{d - 1}}|\nabla_xP_{\lambda}u|^2 \,{\rm d}\widehat{x}_1\right)\left(\int_{\mathbb{T}^{d - 1}}|\nabla_x\partial_1P_{\lambda}u|^2 \,{\rm d}\widehat{x}_1\right) \,{\rm d}x_1{\rm d}t.
    \end{align*}
For $L_2$, by H\"older inequality  
\begin{align*}
    L_2 \geq -\int_0^t\int_{\mathbb{T}}\left(\int_{\mathbb{T}^{d - 1}}|\nabla_xP_{\lambda}u|^2 \,{\rm d}\widehat{x}_1\right)\left(\int_{\mathbb{T}^{d - 1}}|\nabla_x\partial_1P_{\lambda}u|^2 \,{\rm d}\widehat{x}_1\right) \,{\rm d}x_1{\rm d}t.
\end{align*}
Therefore, 
\begin{equation}\label{LHS1}
    {\rm LHS} \geq \int_0^t\int_{\mathbb{T}}\left(\partial_1\int_{\mathbb{T}^{d - 1}}\frac{1}{2}|\nabla_xP_{\lambda}u|^2 \,{\rm d}\widehat{x}_1\right)^2 \,{\rm d}x_1{\rm d}t.
\end{equation}

Turning to the estimates for RHS, firstly from the energy estimates we have 
\begin{equation*}
   R_1 \lesssim t^{\frac{1}{4}}\Vert \phi\Vert_{H_x^\sigma}^4\lambda^{-2\sigma + 2}a_{\lambda}^2,
\end{equation*} 
 and also 
\begin{gather*}
            \int_{\mathbb{T}^d}|\langle \partial_1\partial_{\ell}P_{\lambda}\phi, J(u)\partial_{\ell}P_{\lambda}\phi\rangle| \,{\rm d}x \lesssim \Vert \nabla_x\partial_1P_{\lambda}\phi\Vert_{L_x^2}\Vert \nabla_xP_{\lambda}\phi\Vert_{L_x^2} \lesssim \lambda^{-2\sigma + 3}a_{\lambda}^2, \\ 
    \sup_t\int_{\mathbb{T}^d}|\langle \partial_1\partial_{\ell}P_{\lambda}u(t), J(u)\partial_{\ell}P_{\lambda}u(t)\rangle| \,{\rm d}x \lesssim t^{\frac{1}{4}}\lambda^{-2\sigma + 3}a_{\lambda}^2.
\end{gather*}
For $R_3$, from H\"older's inequality and \eqref{boot1}-\eqref{boot2} we have 
\begin{align*}
    R_3 \lesssim t\lambda^2\Vert \nabla_xP_{\lambda}u\Vert_{L_t^{\infty}L_x^2}^4 \lesssim t^{\frac{5}{6}}\lambda^{-4\sigma + 6}a_{\lambda}^4.
\end{align*}
To control the term 
\begin{equation*}
    \int_0^t\int_{\mathbb{T}}|\mathcal{H}_{\lambda}| \,{\rm d}x_1{\rm d}t,
\end{equation*}
we decompose the commutator terms in $\mathcal{H}_{\lambda}$ to investigate the number of derivatives involved at scale $\lambda$.
We recall that 
\begin{equation*}
    \mathcal{H}_{\lambda} = \int_{\mathbb{T}^{d - 1}}(\mathcal{H}_{\lambda}^1 + \mathcal{H}_{\lambda}^2) \,{\rm d}\widehat{x}_1
\end{equation*}
and $\mathcal{H}_{\lambda}^1, \mathcal{H}_{\lambda}^2$ are defined in Section \ref{Balance}.
For the first term of $\mathcal{H}_{\lambda}^1$, we have
\begin{align*}
            &[\partial_1\partial_{\ell}, n_i(u)]([P_{\lambda}\partial_k, n_i(u)]\partial_ku) \\
    = &\,\partial_1[\partial_{\ell}, n_i(u)]([P_{\lambda}\partial_k, n_i(u)]\partial_ku) + [\partial_1, n_i(u)](\partial_{\ell}([P_{\lambda}\partial_k, n_i(u)]\partial_ku)) \\
    = &\,\partial_1(([P_{\lambda}\partial_k, n_i(u)]\partial_ku)\partial_{\ell}n_i(u)) + (\partial_{\ell}([P_{\lambda}\partial_k, n_i(u)]\partial_ku))\partial_1n_i(u) \\ 
    = &\,\partial_1((P_{\lambda}[\partial_k, n_i(u)]\partial_ku + [P_{\lambda}, n_i(u)]\Delta_x u)\partial_{\ell}n_i(u)) \\
    &+ (\partial_{\ell}(P_{\lambda}[\partial_k, n_i(u)]\partial_ku + [P_{\lambda}, n_i(u)]\Delta_x u))\partial_1n_i(u) \\
    = &\,\partial_1((P_{\lambda}\langle \partial_ku, \partial_kn_i(u)\rangle)\partial_{\ell}n_i(u)) + \partial_1(([P_{\lambda}, n_i(u)]\Delta_x u)\partial_{\ell}n_i(u)) \\
    &+ (P_{\lambda}\partial_{\ell}\langle \partial_ku, \partial_kn_i(u)\rangle)\partial_1n_i(u) + (\partial_{\ell}[P_{\lambda}, n_i(u)]\Delta_x u)\partial_1n_i(u) \\ 
    = &\,(P_{\lambda}\partial_1\langle \partial_ku, \partial_kn_i(u)\rangle)\partial_{\ell}n_i(u) + (P_{\lambda}\langle \partial_ku, \partial_kn_i(u)\rangle)\partial_1\partial_{\ell}n_i(u) \\
    &+ (\partial_1[P_{\lambda}, n_i(u)]\Delta_x u)\partial_{\ell}n_i(u) + ([P_{\lambda}, n_i(u)]\partial_1\Delta_x u)\partial_{\ell}n_i(u) \\
    &+ (P_{\lambda}\partial_{\ell}\langle \partial_ku, \partial_kn_i(u)\rangle)\partial_1n_i(u) + (\partial_{\ell}[P_{\lambda}, n_i(u)]\Delta_x u)\partial_1n_i(u) \\
    =: &\,\mathcal{P}^{1a}_{\lambda} + \mathcal{P}^{1b}_{\lambda} + \mathcal{P}^{1c}_{\lambda} + \mathcal{P}^{1d}_{\lambda} + \mathcal{P}^{1e}_{\lambda} + \mathcal{P}^{1f}_{\lambda}. 
\end{align*}
It follows that 
    \begin{align*}
            \langle \mathcal{P}^{1a}_{\lambda}, \partial_{\ell}P_{\lambda}u\rangle &= (P_{\lambda}\partial_1\langle \partial_ku, \partial_kn_i(u)\rangle)\langle \partial_{\ell}P_{\lambda}u, \partial_1n_i(u)\rangle \\
    &= -(P_{\lambda}\partial_1\langle \partial_ku, \partial_kn_i(u)\rangle)([P_{\lambda}, \partial_1n_i(u)]\partial_{\ell}u) \\
    &\quad + (P_{\lambda}\partial_1\langle \partial_ku, \partial_kn_i(u)\rangle)(P_{\lambda}\langle \partial_{\ell}u, \partial_1n_i(u)\rangle), \\
    \langle \mathcal{P}^{1b}_{\lambda}, \partial_{\ell}P_{\lambda}u\rangle &= (P_{\lambda}\langle \partial_ku, \partial_kn_i(u)\rangle)\langle\partial_{\ell}P_{\lambda}u, \partial_1\partial_{\ell}n_i(u)\rangle \\ 
    &= -(P_{\lambda}\langle \partial_ku, \partial_kn_i(u)\rangle)([P_{\lambda}, \partial_1\partial_{\ell}n_i(u)]\partial_{\ell}u) \\ 
    &\quad + (P_{\lambda}\langle \partial_ku, \partial_kn_i(u)\rangle)(P_{\lambda}\langle \partial_{\ell}u, \partial_1\partial_{\ell}n_i(u)\rangle), \\ 
    \langle \mathcal{P}^{1d}_{\lambda}, \partial_{\ell}P_{\lambda}u\rangle &= ([P_{\lambda}, n_i(u)]\partial_1\Delta_x u)\langle \partial_{\ell}P_{\lambda}u, \partial_1n_i(u)\rangle \\
    &= ([P_{\lambda}, n_i(u)]\partial_1\Delta_x u)(-[P_{\lambda}, \partial_1n_i(u)]\partial_{\ell}P_{\lambda}u + P_{\lambda}\langle \partial_{\ell}u, \partial_1n_i(u)\rangle).
\end{align*}
The terms $\langle \mathcal{P}_{\lambda}^{1e}, \partial_{\ell}P_{\lambda}u\rangle, \langle \mathcal{P}^{1f}_{\lambda}, \partial_{\ell}P_{\lambda}u\rangle$ can be dealt with exactly the same as $\langle \mathcal{P}_{\lambda}^{1a}, \partial_{\ell}P_{\lambda}u\rangle$ and $\langle \mathcal{P}_{\lambda}^{1c}, \partial_{\ell}P_{\lambda}u\rangle$, respectively.
For $\mathcal{P}_{\lambda}^{1c}$, note the identity 
\begin{equation*}
    \partial_1[P_{\lambda}, n_i(u)] = [P_{\lambda}, \partial_1n_i(u)] + [P_{\lambda}, n_i(u)]\partial_1.
\end{equation*}
We have 
\begin{equation*}
    \begin{aligned}
     \langle \mathcal{P}^{1c}_{\lambda}, \partial_{\ell}P_{\lambda}u\rangle &= ([P_{\lambda}, \partial_1n_i(u)]\Delta_x u + [P_{\lambda}, n_i(u)]\partial_1\Delta_x u)\langle \partial_{\ell}P_{\lambda}u, \partial_{\ell}n_i(u)\rangle \\
    &= ([P_{\lambda}, \partial_1n_i(u)]\Delta_x u + [P_{\lambda}, n_i(u)]\partial_1\Delta_x u)(-[P_{\lambda}, \partial_{\ell}n_i(u)]\partial_{\ell}u + P_{\lambda}\langle \partial_{\ell}u, \partial_{\ell}n_i(u)\rangle).
    \end{aligned}
\end{equation*}
The calculations above shows that the first term of $\mathcal{H}_{\lambda}^1$ contains three derivatives at $\lambda$ scale.
Thus, 
\begin{align*}
    \int_0^t\int_{\mathbb{T}}\left|\int_{\mathbb{T}^{d - 1}}\langle [\partial_1\partial_{\ell}, n_i(u)]([P_{\lambda}\partial_k, n_i(u)]\partial_ku), \partial_{\ell}P_{\lambda}u\rangle \,{\rm d}\widehat{x}_1\right| \,{\rm d}x_1{\rm d}t \lesssim t^{\frac{1}{4}}\Vert \phi\Vert_{H_x^\sigma}^4\lambda^{-2\sigma + 3}a_{\lambda}^2.
\end{align*} 
The other components of $\mathcal{H}_{\lambda}^1$ can be estimated similarly, except for the term
\begin{equation}\label{abs}
    \langle \partial_1\partial_{\ell}[P_{\lambda}, D_k]\partial_ku, \partial_{\ell}P_{\lambda}u\rangle,
\end{equation}
we need to expand the commutator $[P_{\lambda}, D_k]$ to see that the highest-order term $\partial_1\partial_{\ell}\partial_k\partial_kP_{\lambda}u$ vanishes since $P_{\lambda}$ and $\partial$ commute:
\begin{align*}
    \partial_1\partial_{\ell}[P_{\lambda}, D_k]\partial_ku &= \partial_1\partial_{\ell}P_{\lambda}\partial_k\partial_ku - \partial_1\partial_{\ell}P_{\lambda}(\langle \partial_k\partial_ku, n_i(u)\rangle n_i(u)) \\
    &\quad - \partial_1\partial_{\ell}\partial_kP_{\lambda}\partial_ku + \partial_1\partial_{\ell}(\langle \partial_kP_{\lambda}\partial_ku, n_i(u)\rangle n_i(u)) \\
    &= - \partial_1\partial_{\ell}P_{\lambda}(\langle \partial_k\partial_ku, n_i(u)\rangle n_i(u)) + \partial_1\partial_{\ell}(\langle \partial_kP_{\lambda}\partial_ku, n_i(u)\rangle n_i(u)).
\end{align*}
Moreover, thanks to the orthogonality of unit normals and tangent vectors, 
the inner products of the remainders with the term $\partial_{\ell}P_{\lambda}u$ exhibits a certain commutator structure, and hence \eqref{abs} also involves three derivatives at the scale $\lambda$ (see the proof for Proposition \ref{prop3} for more details).
We conclude that 
\begin{equation*}
    \int_0^t\int_{\mathbb{T}^{d - 1}}|\mathcal{H}_{\lambda}^1| \,{\rm d}x_1{\rm d}t \lesssim t^{\frac{1}{4}}\Vert \phi\Vert_{H_x^\sigma}^4\lambda^{-2\sigma + 3}a_{\lambda}^2.
\end{equation*}

For the estimates of $\mathcal{H}_{\lambda}^2$, we note for the term
\begin{equation}\label{AB}
    \langle (\partial_tJ(u))\partial_{\ell}P_{\lambda}u - (\partial_{\ell}J(u))\partial_tP_{\lambda}u, \partial_1\partial_{\ell}P_{\lambda}u\rangle,
\end{equation}
we need to use the compatiblility condition $D_kJ = D_tJ = 0$ to see that \eqref{AB} costs three derivative at scale $\lambda$, where $\partial_tP_{\lambda}u$ should be expanded as follows:
\begin{equation*}
    \partial_tP_{\lambda}u = [P_{\lambda}, J(u)]D_k\partial_ku + J(u)[P_{\lambda}, D_k]\partial_ku + J(u)D_k\partial_kP_{\lambda}u.
\end{equation*}
The estimates for the remaining terms of $\mathcal{H}_{\lambda}^2$ are completely analogous to those of preceding terms.
Therefore, 
\begin{equation*}
    \int_0^t\int_{\mathbb{T}^{d - 1}}|\mathcal{H}_{\lambda}^2| \,{\rm d}x_1{\rm d}t \lesssim t^{\frac{1}{4}}\Vert \phi\Vert_{H_x^\sigma}^4\lambda^{-2\sigma + 3}a_{\lambda}^2.
\end{equation*}

To conclude, 
\begin{equation}\label{RHS111}
    {\rm RHS} \lesssim \lambda((1 + t^{\frac{1}{4}})\Vert \phi\Vert_{H_x^\sigma}^4\lambda^{-2\sigma + 2}a_{\lambda}^2)^2 + t^{\frac{5}{6}}\lambda^{-4\sigma + 6}a_{\lambda}^4.
\end{equation}
Inserting \eqref{LHS1} and \eqref{RHS111} into \eqref{divcurl}, we obtain
\begin{equation}\label{divcurl1}
    \begin{aligned}
        \int_0^t\int_{\mathbb{T}}\left(\partial_1\left(\int_{\mathbb{T}^{d - 1}}\frac{1}{2}|\nabla_xP_{\lambda}u|^2 \,{\rm d}\widehat{x}_1\right)\right)^2 \,{\rm d}x_1{\rm d}t \lesssim \lambda((1 + t^{\frac{1}{4}})\Vert \phi\Vert_{H_x^\sigma}^4\lambda^{-2\sigma + 2}a_{\lambda}^2)^2 + t^{\frac{5}{6}}\lambda^{-4s + 6}a_{\lambda}^4.
    \end{aligned}
\end{equation}

Finally, we set 
\begin{equation*}
    f(x_1) := \int_{\mathbb{T}^{d - 1}}\frac{1}{2}|\nabla_xP_{\lambda}u|^2 \,{\rm d}\widehat{x}_1.
\end{equation*}
By Gagliardo-Nirenberg-Sobolev inequality
\begin{equation*}
    \Vert f\Vert_{L_{x_1}^{\infty}(\mathbb{T})} \lesssim \Vert \partial_1f\Vert_{L_{x_1}^1(\mathbb{T})}^{\frac{2}{3}}\Vert f\Vert_{L_{x_1}^2(\mathbb{T})}^{\frac{1}{3}} + \Vert f \Vert_{L_{x_1}^1(\mathbb{T})},
\end{equation*}
we have 
\begin{equation*}
    \begin{aligned}
            &\sup_{x_1}\int_{\mathbb{T}^{d - 1}}\frac{1}{2}|\nabla_xP_{\lambda}u|^2 \,{\rm d}\widehat{x}_1 \\
    \lesssim &\left(\int_{\mathbb{T}}\left(\partial_1\int_{\mathbb{T}^{d - 1}}\frac{1}{2}|\nabla_xP_{\lambda}u|^2 \,{\rm d}\widehat{x}_1\right)^2 \,{\rm d}x_1\right)^{\frac{1}{3}}\left(\int_{\mathbb{T}^d}\frac{1}{2}|\nabla_xP_{\lambda}u|^2 \,{\rm d}x\right)^{\frac{1}{3}} + \int_{\mathbb{T}^{d}}\frac{1}{2}|\nabla_xP_{\lambda}u|^2 \,{\rm d}x.
    \end{aligned}
\end{equation*}
Namely, 
\begin{equation*}
    \begin{aligned}
            &\left(\sup_{x_1}\int_{\mathbb{T}^{d - 1}}\frac{1}{2}|\nabla_xP_{\lambda}u|^2 \,{\rm d}\widehat{x}_1\right)^3 \\
    \lesssim &\left(\int_{\mathbb{T}}\left(\partial_1\int_{\mathbb{T}^{d - 1}}\frac{1}{2}|\nabla_xP_{\lambda}u|^2 \,{\rm d}\widehat{x}_1\right)^2 \,{\rm d}x_1\right)\left(\int_{\mathbb{T}^d}\frac{1}{2}|\nabla_xP_{\lambda}u|^2 \,{\rm d}x\right) + \left(\int_{\mathbb{T}^{d}}\frac{1}{2}|\nabla_xP_{\lambda}u|^2 \,{\rm d}x\right)^3.
    \end{aligned}
\end{equation*}
Integrating from $0$ to $t$ on both sides of the inequality above, we get
\begin{equation}
    \begin{aligned}
            &\int_0^t\left(\sup_{x_1}\int_{\mathbb{T}^{d - 1}}\frac{1}{2}|\nabla_xP_{\lambda}u|^2 \,{\rm d}\widehat{x}_1\right)^3 \,{\rm d}t \\
    \lesssim &\int_0^t\left(\int_{\mathbb{T}}\left(\partial_1\int_{\mathbb{T}^{d - 1}}\frac{1}{2}|\nabla_xP_{\lambda}u|^2 \,{\rm d}\widehat{x}_1\right)^2 \,{\rm d}x_1\right)\left(\int_{\mathbb{T}^d}\frac{1}{2}|\nabla_xP_{\lambda}u|^2 \,{\rm d}x\right) \,{\rm d}t + \int_0^t\left(\int_{\mathbb{T}^{d}}\frac{1}{2}|\nabla_xP_{\lambda}u|^2 \,{\rm d}x\right)^3 \,{\rm d}t.
    \end{aligned}
\end{equation}
By H\"older inequality, 
    \begin{align*}
            \Vert \nabla_xP_{\lambda}u\Vert_{L_t^6L_{x_1}^{\infty}L_y^2}^6 \lesssim \Vert \nabla_xP_{\lambda}u\Vert_{L_t^{\infty}L_x^2}^2\int_0^t\int_{\mathbb{T}}\left(\partial_1\left(\int_{\mathbb{T}^{d - 1}}\frac{1}{2}|\nabla_xP_{\lambda}u|^2 \,{\rm d}\widehat{x}_1\right)\right)^2 \,{\rm d}x_1{\rm d}t  + \Vert \nabla_xP_{\lambda}u\Vert_{L_t^6L_x^2}^6, 
    \end{align*}
and it follows from energy estimates and \eqref{divcurl1} that 
    \begin{align*}
        \Vert \nabla_xP_{\lambda}u\Vert_{L_t^6L_{x_1}^{\infty}L_y^2}^6 \lesssim (\lambda^{-\sigma + 1}a_{\lambda})^2(\lambda((1 + t^{\frac{1}{4}})\Vert \phi\Vert_{H_x^\sigma}^4\lambda^{-2\sigma + 2}a_{\lambda}^2)^2 + t^{\frac{5}{6}}\lambda^{-4\sigma + 6}a_{\lambda}^4) + t(\lambda^{-\sigma + 1}a_{\lambda})^6.
    \end{align*}
Hence there exists $T^* = T^*(\Vert \phi\Vert_{H_x^\sigma}) > 0$, such that
\begin{equation*}
    \Vert \nabla_xP_{\lambda}u\Vert_{L_t^6L_{x_1}^{\infty}L_y^2} \lesssim \lambda^{\frac{1}{6}(-2\sigma + 2 - 4\sigma+ 6)}a_{\lambda} = \lambda^{-\sigma + \frac{4}{3}}a_{\lambda}, \qquad t \in [0, T^*].
\end{equation*}
By Bernstein inequality,
\begin{equation*}
    \Vert \nabla_xP_{\lambda}u\Vert_{L_t^6L_x^{\infty}} \lesssim \lambda^{\frac{d - 1}{2}}\Vert \nabla_xP_{\lambda}u\Vert_{L_t^6L_{x_1}^{\infty}L_y^2} \lesssim \lambda^{-\sigma + \frac{d}{2} + \frac{5}{6}}a_{\lambda}.
\end{equation*}

\appendix

\section{Estimates for \texorpdfstring{$\widetilde{\bm{A}}$}{}}\label{ABC}

We recall that 
\begin{equation*}
    \widetilde{\bm{A}} = \nabla_x^{-1}\sum_{\ell = 1}^d\mathcal{R}_{\ell}(\Im(\bm{\psi}\overline{\psi_{\ell}})),
\end{equation*}
where $\nabla_x^{-1}$ is the Fourier multiplier defined by $\xi \mapsto |\xi|^{-1}$ and $\mathcal{R}_{\ell}$ is the Riesz-type operator.

\begin{proposition}\label{B1}
    Suppose $\sigma > (d + 1)/2$ and \eqref{13}-\eqref{20} hold. We have 
    \begin{equation*}
        \Vert \widetilde{\bm{A}}\Vert_{L_t^2H_x^{\frac{d + 1}{2}}} \lesssim t^{-\sqrt{\frac{\varepsilon}{2}}}\Vert \bm{\psi}(0)\Vert_{H_x^{\sigma - 1}}^2.
    \end{equation*}
    \begin{proof}
        The proof for this proposition is standard, utilizing Littlewood-Paley decomposition, and we include it for completeness.
        Expressing the $H_x^{\sigma}$ norm using dyadic decomposition, we get 
        \begin{equation*}
            \begin{aligned}
            \Vert \widetilde{\bm{A}}\Vert_{L_t^2H_x^{\frac{d + 1}{2}}}^2 &\lesssim \sum_{\lambda \geq 1}\lambda^{d - 1}\Vert P_{\lambda}(\bm{\psi} \cdot \bm{\psi})\Vert_{L_{t, x}^2}^2 \\ 
            &\lesssim \left(\sum_{\lambda_1 \ll \lambda_2 \sim \lambda} + \sum_{\lambda_2 \ll \lambda_1 \sim \lambda} + \sum_{\lambda_1 \sim \lambda_2 \gtrsim \lambda}\right)\lambda^{d - 1}\Vert P_{\lambda_1}\bm{\psi} \cdot P_{\lambda_2}\bm{\psi}\Vert_{L_{t, x}^2}^2 \\ 
            &=: \mathcal{A}_1 + \mathcal{A}_2 + \mathcal{A}_3.
            \end{aligned}
        \end{equation*}
        For the estimates of $\mathcal{A}_1$ and $\mathcal{A}_2$ we use $L_{t, x}^2$ bilinear estimates \eqref{16}. For instance, 
        \begin{equation*}
            \begin{aligned}
                \mathcal{A}_1 &\lesssim \sum_{\lambda \geq 1}\sum_{\mu \ll \lambda}\sum_{i = 1}^d \lambda^{d - 1}\mu^{d - 1}\left\Vert \Vert P_{\mu}\bm{\psi}\Vert_{L_{\widehat{x}_i}^2}\Vert P_{\lambda, e_i}\bm{\psi}\Vert_{L_{\widehat{x}_i}^2}\right\Vert_{L_{t, x_i}^2}^2 \\ 
                &\lesssim t^{-\frac{\varepsilon}{4}}\sum_{\lambda \geq 1}\sum_{\mu \ll \lambda} \lambda^{-2\sigma + d + 1}\mu^{-2\sigma + d + 1}a_{\mu}^2a_{\lambda}^2 \\ 
                &\lesssim t^{-\frac{\varepsilon}{4}}\Vert \bm{\psi}(0)\Vert_{H_x^{\sigma - 1}}^4
            \end{aligned}
        \end{equation*}
        since $\sigma > (d + 1)/2$. We also conclude 
        \begin{equation*}
            \mathcal{A}_2 \lesssim t^{-\frac{\varepsilon}{4}}\Vert \bm{\psi}(0)\Vert_{H_x^{\sigma - 1}}^4.
        \end{equation*}
        Controlling $\mathcal{A}_3$, by \eqref{17} 
        \begin{equation*}
            \begin{aligned}
            \mathcal{A}_3 \lesssim \sum_{\mu \gtrsim \lambda \geq 1}\lambda^{d - 1}\Vert P_{\mu}\bm{\psi}\Vert_{L_{t, x}^4}^4 \lesssim t^{-\frac{\varepsilon}{2}}\sum_{\mu \gtrsim \lambda \geq 1}\lambda^{d - 1}\mu^{-4\sigma + d + 3}a_{\mu}^4 \lesssim t^{-\frac{\varepsilon}{2}}\Vert \bm{\psi}(0)\Vert_{H_x^{\sigma - 1}}^4.
            \end{aligned}
        \end{equation*}
        The proof is complete.
    \end{proof}
\end{proposition}

\section{Estimates for Homotopy Initial Data}

Suppose $\phi_0, \phi_1 \in H^{\sigma}(\mathbb{T}^d; \mathbb{S}^2)$ with $\sigma > d/2 + 1/2$.
For $0 < \zeta < 1$ we define 
\begin{equation*}
    \phi_{\zeta} = (1 - \zeta)\phi_0 + \zeta\phi_1, \quad \Phi_{\zeta} := \frac{\phi_{\zeta}}{|\phi_{\zeta}|}.
\end{equation*}
Moreover we assume $\phi_0 \neq -\phi_1$, and then there exists $C_0 > 0$ such that 
\begin{equation}\label{assu}
    \inf_{x \in \mathbb{T}^d, \zeta \in [0, 1]}|\phi_{\zeta}| \geq C_0.
\end{equation}

\begin{proposition}\label{AC}
    For $d \geq 2$ and $\sigma > d/2$
    \begin{equation*}
        \Vert \partial_{\zeta}\Phi_{\zeta}\Vert_{H_x^{\sigma - 1}} \lesssim \Vert \phi_0 - \phi_1\Vert_{H_x^{\sigma - 1}}.
    \end{equation*}
    \begin{proof}
        We first have
        \begin{equation*}
            \partial_{\zeta}\Phi_{\zeta} = \frac{\phi_1 - \phi_0}{|\phi_{\zeta}|} - \frac{\phi_{\zeta} \cdot (\phi_1 - \phi_0)}{|\phi_{\zeta}|^2}\phi_{\zeta}.
        \end{equation*}
        Note $|\phi_0| = |\phi_1| = 1$. It follows that 
        \begin{equation*}
            \begin{aligned}
                |\phi_{\zeta}|^2 &= (1 - \zeta)^2|\phi_0|^2 + 2\zeta(1 - \zeta)\phi_1 \cdot \phi_0 + \zeta^2|\phi_0|^2 \\ 
                &= (1 - \zeta)^2 + \zeta^2 + 2\zeta(1 - \zeta)\phi_1 \cdot \phi_0 \\ 
                &= 1 - 2\zeta(1 - \zeta)(1 - \phi_1 \cdot \phi_0) \\ 
                &= 1 - \zeta(1 - \zeta)|\phi_1 - \phi_0|^2, \\
                \phi_{\zeta} \cdot (\phi_1 - \phi_0) &= ((1 - \zeta)\phi_0 + \zeta\phi_1) \cdot (\phi_1 - \phi_0) \\ 
                &= -(1 - \zeta)|\phi_0|^2 + (1 - \zeta - \zeta)\phi_1 \cdot \phi_0 + \zeta|\phi_1|^2 \\ 
                &= 2\zeta - 1 - (2\zeta - 1)\phi_1 \cdot \phi_0 \\ 
                &= (2\zeta - 1)(1 - \phi_1 \cdot \phi_0) \\ 
                &= \left(\zeta - \frac{1}{2}\right)|\phi_1 - \phi_0|^2.
            \end{aligned}
        \end{equation*}
        Thus, 
        \begin{equation*}
            \partial_{\zeta}\Phi_{\zeta} = \frac{\phi_1 - \phi_0}{(1 - \zeta(1 - \zeta)|\phi_1 - \phi_0|^2)^{1/2}} - \frac{(\zeta - 1/2)|\phi_1 - \phi_0|^2}{(1 - \zeta(1 - \zeta)|\phi_1 - \phi_0|^2)^{3/2}}\phi_{\zeta} =: \Phi_1 + \Phi_2.
        \end{equation*}
        We set $\widetilde{\phi} := \phi_1 - \phi_0$. On the one hand, 
        \begin{equation*}
            \Phi_1 = \widetilde{\phi} + g(\widetilde{\phi})\widetilde{\phi}, \qquad g(x) = \frac{1}{(1 - \zeta(1 - \zeta)|x|^2)^{1/2}} - 1.
        \end{equation*}
        Note $g(0) = 0$. By paralinearization theorem (Theorem \ref{para}) and Sobolev embedding
        \begin{equation*}
            \Vert g(\widetilde{\phi})\Vert_{H_x^{\sigma - 1}} \lesssim \Vert \widetilde{\phi}\Vert_{H_x^{\sigma - 1}},
        \end{equation*}
        and in view of \eqref{assu}, 
        \begin{equation*}
            \Vert g(\widetilde{\phi})\Vert_{L_x^{\infty}} \lesssim \Vert \widetilde{\phi}\Vert_{L_x^{\infty}}.
        \end{equation*}
        Then from Kato-Ponce inequality \cite{A,KP}
        \begin{equation*}
            \Vert g(\widetilde{\phi})\widetilde{\phi}\Vert_{H_x^{\sigma - 1}} \lesssim \Vert g(\widetilde{\phi})\Vert_{H_x^{\sigma - 1}}\Vert \widetilde{\phi}\Vert_{L_x^{\infty}} + \Vert g(\widetilde{\phi})\Vert_{L_x^{\infty}}\Vert \widetilde{\phi}\Vert_{H_x^{\sigma - 1}} \lesssim \Vert \widetilde{\phi}\Vert_{L_x^{\infty}}\Vert \widetilde{\phi}\Vert_{H_x^{\sigma - 1}} \lesssim \Vert \widetilde{\phi}\Vert_{H_x^{\sigma - 1}},
        \end{equation*}
        where the last inequality follows from the fact $H_x^{\sigma} \hookrightarrow L_x^{\infty}$, and hence
        \begin{equation*}
            \Vert \widetilde{\phi}\Vert_{L_x^{\infty}} \leq \Vert \phi_1\Vert_{L_x^{\infty}} + \Vert \phi_0\Vert_{L_x^{\infty}} \lesssim \Vert \phi_1\Vert_{H_x^{\sigma}} + \Vert \phi_0\Vert_{H_x^{\sigma}} \leq C.
        \end{equation*}
        Therefore, 
        \begin{equation*}
            \Vert \Phi_1\Vert_{H_x^{\sigma - 1}} \lesssim \Vert \widetilde{\phi}\Vert_{H_x^{\sigma - 1}}.
        \end{equation*}
        In a similar manner we can also obtain 
        \begin{equation*}
            \Vert \Phi_2\Vert_{H_x^{\sigma - 1}} \lesssim \Vert \widetilde{\phi}\Vert_{H_x^{\sigma - 1}}.
        \end{equation*}
        The proof is concluded.
    \end{proof}
\end{proposition}

\section{Estimates for Unit Normals}\label{AA}

\begin{proposition}\label{prop1}
    Under assumptions \eqref{boot1}-\eqref{boot2} there holds 
    \begin{equation*}
        \Vert P_{\lambda}\nabla_xn_i(u)\Vert_{L_t^6L_x^2} \lesssim t^{-\frac{1}{12}}\Vert \phi\Vert_{H_x^\sigma}\lambda^{-\sigma + 1}a_{\lambda}.
    \end{equation*}
    \begin{proof}
    By Bony decomposition, 
    \begin{equation}\label{Bony1}
    \begin{aligned}
        \Vert P_{\lambda}\nabla_xn_i(u)\Vert_{L_t^6L_x^2} &= \Vert P_{\lambda}(n_i'(u)\nabla_xu)\Vert_{L_t^6L_x^2} \\
        &\lesssim \left(\sum_{\lambda_1 \ll \lambda_2 \sim \lambda} + \sum_{\lambda_2 \ll \lambda_1 \sim \lambda} + \sum_{\lambda_1 \sim \lambda_2 \gtrsim \lambda}\right)\Vert P_{\lambda_1}n_i'(u)P_{\lambda_2}\nabla_xu\Vert_{L_t^6L_x^2}.
    \end{aligned}
\end{equation}
    Here $n_i'(u) \in \mathbb{R}^{N \times N}$ and $\nabla_xu \in \mathbb{R}^{N \times 1}$.
    For the second term on the right-hand side of \eqref{Bony1}, from H\"older inequality, \eqref{boot2} and \eqref{fre1} we have
    \begin{align*}
            \sum_{\lambda_2 \ll \lambda_1 \sim \lambda}\Vert P_{\lambda_1}n_i'(u)P_{\lambda_2}\nabla_xu\Vert_{L_t^6L_x^2} &\leq \Vert P_{\lambda}n_i'(u)\Vert_{L_t^{\infty}L_x^2}\sum_{\lambda_2 \ll \lambda}\Vert P_{\lambda_2}\nabla_xu\Vert_{L_t^6L_x^{\infty}} \\ 
        &\lesssim t^{-\frac{1}{24}}\Vert P_{\lambda}n_i'(u)\Vert_{L_t^{\infty}L_x^2}\sum_{\lambda_2 \ll \lambda}\lambda_2^{-\sigma + \frac{d}{2} + \frac{5}{6}}a_{\lambda_2} \\ 
        &\lesssim t^{-\frac{1}{24}}\Vert P_{\lambda}n_i'(u)\Vert_{L_t^{\infty}L_x^2}\sum_{\lambda_2 \ll \lambda}\lambda_2^{-\sigma + \frac{d}{2} + \frac{5}{6}}\left(\frac{\lambda}{\lambda_2}\right)^{\delta}a_{\lambda}.
    \end{align*}
    Note 
    \begin{equation*}
        \Vert P_{\lambda}n_i'(u)\Vert_{L_t^{\infty}L_x^2} = \lambda^{-\sigma}\lambda^\sigma\Vert P_{\lambda}n_i'(u)\Vert_{L_t^{\infty}L_x^2} \lesssim \lambda^{-\sigma}\Vert n_i'(u)\Vert_{L_t^{\infty}H_x^\sigma}
    \end{equation*}
    It follows from Theorem \ref{para} and \eqref{boot1} that 
    \begin{align*}
    \Vert P_{\lambda}n_i'(u)\Vert_{L_t^{\infty}L_x^2} \lesssim \lambda^{-\sigma}\Vert n_i'(u)\Vert_{L_t^{\infty}H_x^\sigma} \lesssim t^{-\frac{1}{24}}\lambda^{-\sigma}\Vert \phi\Vert_{H_x^\sigma}.
\end{align*}
Since $\sigma > d/2 + 5/6$, we have 
    \begin{align*}
            \sum_{\lambda_2 \ll \lambda_1 \sim \lambda}\Vert P_{\lambda_1}n_i'(u)P_{\lambda_2}\nabla_xu\Vert_{L_t^6L_x^2} &\lesssim t^{-\frac{1}{12}}\Vert \phi\Vert_{H_x^\sigma}\lambda^{-\sigma + \delta}a_{\lambda}\sum_{\lambda_2 \ll \lambda}\lambda_2^{-\sigma + \frac{d}{2} + \frac{5}{6} - \delta}\\
    &\lesssim t^{-\frac{1}{12}}\Vert \phi\Vert_{H_x^\sigma}\lambda^{-\sigma + \delta}a_{\lambda},
\end{align*}
where $\delta = \delta(\sigma, d)$ can be chosen to satisfy $0 < \delta < 1$.  

For the first and the third term, by H\"older inequality, Theorem \ref{Nbern} and \eqref{boot1}-\eqref{boot2}
\begin{align*}                \sum_{\lambda_1 \ll \lambda_2 \sim \lambda}\Vert P_{\lambda_1}n_i'(u)P_{\lambda_2}\nabla_xu\Vert_{L_t^6L_x^2} &\leq \Vert P_{\lambda}\nabla_xu\Vert_{L_t^{\infty}L_x^2}\sum_{\lambda_1 \ll \lambda}\Vert P_{\lambda_1}n_i'(u)\Vert_{L_t^6L_x^{\infty}} \\ 
    &\lesssim t^{-\frac{1}{24}}\lambda^{-\sigma + 1}a_{\lambda}\sum_{\lambda_1 \ll \lambda}\lambda_1^{-1}\Vert \nabla_xu\Vert_{L_t^6L_x^{\infty}} \\ 
    &\lesssim t^{-\frac{1}{12}}\Vert \phi\Vert_{H_x^\sigma}\lambda^{-\sigma + 1}a_{\lambda}, \\
    \sum_{\lambda_1 \sim \lambda_2 \gtrsim \lambda}\Vert P_{\lambda_1}n_i'(u)P_{\lambda_2}\nabla_xu\Vert_{L_t^6L_x^2} &\lesssim \sum_{\lambda_1 \gtrsim \lambda}\Vert P_{\lambda_1}n_i'(u)\Vert_{L_t^6L_x^{\infty}}\Vert P_{\lambda_1}\nabla_xu\Vert_{L_t^{\infty}L_x^2} \\
    &\lesssim t^{-\frac{1}{24}}\sum_{\lambda_1 \gtrsim \lambda}\lambda_1^{-1}\Vert \nabla_xu\Vert_{L_t^6L_x^{\infty}}\lambda_1^{-\sigma + 1}a_{\lambda_1} \\ 
    &\lesssim t^{-\frac{1}{12}}\Vert \phi\Vert_{H_x^\sigma}\lambda^{-\sigma + 1}a_{\lambda}\sum_{\lambda_1 \gtrsim \lambda}\lambda_1^{-1}\left(\frac{\lambda_1}{\lambda}\right)^{-\sigma + 1 + \delta} \\ 
    &\lesssim t^{-\frac{1}{12}}\Vert \phi\Vert_{H_x^\sigma}\lambda^{-\sigma + 1}a_{\lambda} \qquad (-\sigma + 1 + \delta < 0).
\end{align*}
    \end{proof}
\end{proposition}

\begin{proposition}\label{prop2}
    Under assumptions \eqref{boot1}-\eqref{boot2} there holds 
    \begin{equation*}
        \sum_{\mu \ll \lambda}\Vert P_{\mu}\nabla_xn_i(u)\Vert_{L_t^6L_x^{\infty}} \lesssim t^{-\frac{1}{12}}\Vert \phi\Vert_{H_x^\sigma}^2.
    \end{equation*}
    \begin{proof}
    Similar to the proof for Proposition \ref{prop1}, we first write 
        \begin{align*}
    \sum_{\mu \ll \lambda}\Vert P_{\mu}\nabla_xn_i(u)\Vert_{L_t^6L_x^{\infty}} \lesssim \sum_{\mu \ll \lambda}\left(\sum_{\mu_1 \ll \mu_2 \sim \mu} + \sum_{\mu_2 \ll \mu_1 \sim \mu} + \sum_{\mu_1 \sim \mu_2 \gtrsim \mu}\right)\Vert P_{\mu_1}n_i'(u)P_{\mu_2}\nabla_xu\Vert_{L_t^6L_x^{\infty}}.
\end{align*}
Note $\sigma > d/2 + 5/6$. It follows that 
\begin{align*}
                            \sum_{\mu \ll \lambda}\sum_{\mu_1 \ll \mu_2 \sim \mu}\Vert P_{\mu_1}n_i'(u)P_{\mu_2}\nabla_xu\Vert_{L_t^6L_x^{\infty}} &\lesssim \sum_{\mu \ll \lambda}\Vert P_{\mu}\nabla_xu\Vert_{L_t^6L_x^{\infty}}\sum_{\mu_1 \ll \mu}\Vert P_{\mu_1}n_i'(u)\Vert_{L_t^{\infty}L_x^{\infty}} \\ 
    &\lesssim \sum_{\mu \ll \lambda}\Vert P_{\mu}\nabla_xu\Vert_{L_t^6L_x^{\infty}}\sum_{\mu_1 \ll \mu}\mu_1^{\frac{d}{2}}\Vert P_{\mu_1}n_i'(u)\Vert_{L_t^{\infty}L_x^2} \\ 
    &= \sum_{\mu \ll \lambda}\Vert P_{\mu}\nabla_xu\Vert_{L_t^6L_x^{\infty}}\sum_{\mu_1 \ll \mu}\mu_1^{\frac{d}{2} - \sigma}\mu_1^\sigma\Vert P_{\mu_1}n_i'(u)\Vert_{L_t^{\infty}L_x^2} \\ 
    &\lesssim \Vert n_i'(u)\Vert_{L_t^{\infty}H_x^\sigma}\sum_{\mu \ll \lambda}\Vert P_{\mu}\nabla_xu\Vert_{L_t^6L_x^{\infty}} \\ 
    &\lesssim t^{-\frac{1}{24}}\Vert \phi\Vert_{H_x^\sigma}\sum_{\mu \ll \lambda}\mu^{-\sigma + \frac{d}{2} + \frac{5}{6}}a_{\mu} \\ 
    &\lesssim t^{-\frac{1}{12}}\Vert \phi\Vert_{H_x^\sigma}^2, \\ 
    \sum_{\mu \ll \lambda}\sum_{\mu_2 \ll \mu_1 \sim \mu}\Vert P_{\mu_1}n_i'(u)P_{\mu_2}\nabla_xu\Vert_{L_t^6L_x^{\infty}} &\lesssim \sum_{\mu \ll \lambda}\Vert P_{\mu}n_i'(u)\Vert_{L_t^6L_x^{\infty}}\sum_{\mu_2 \ll \mu} \Vert P_{\mu_2}\nabla_xu\Vert_{L_t^{\infty}L_x^{\infty}} \\
    &\lesssim \sum_{\mu \ll \lambda}\Vert P_{\mu}n_i'(u)\Vert_{L_t^6L_x^{\infty}}\sum_{\mu_2 \ll \mu}\mu_2^{\frac{d}{2}} \Vert P_{\mu_2}\nabla_xu\Vert_{L_t^{\infty}L_x^2} \\
    &\lesssim t^{-\frac{1}{24}}\sum_{\mu \ll \lambda}\Vert P_{\mu}n_i'(u)\Vert_{L_t^6L_x^{\infty}}\sum_{\mu_2 \ll \mu}\mu_2^{\frac{d}{2}}\mu_2^{-\sigma + 1}a_{\mu_2} \\
    &\lesssim t^{-\frac{1}{24}}\sum_{\mu \ll \lambda}\mu^{\frac{1}{6}}\Vert P_{\mu}n_i'(u)\Vert_{L_t^6L_x^{\infty}}\sum_{\mu_2 \ll \mu}\mu_2^{-\sigma + \frac{d}{2} + \frac{5}{6}}a_{\mu_2} \\ 
    &\lesssim t^{-\frac{1}{24}}\Vert \phi\Vert_{H_x^\sigma}\sum_{\mu \ll \lambda}\mu^{\frac{1}{6}}\Vert P_{\mu}n_i'(u)\Vert_{L_t^6L_x^{\infty}} \\
    &\lesssim t^{-\frac{1}{24}}\Vert \phi\Vert_{H_x^\sigma}\sum_{\mu \ll \lambda}\mu^{\frac{1}{6}}\mu^{-1}\Vert \nabla_xu\Vert_{L_t^6L_x^{\infty}} \\
    &\lesssim t^{-\frac{1}{12}}\Vert \phi\Vert_{H_x^\sigma}^2 
\end{align*}
and
\begin{equation*}
    \begin{aligned}
\sum_{\mu \ll \lambda}\sum_{\mu_1 \sim \mu_2 \gtrsim \mu}\Vert P_{\mu_1}n_i'(u)P_{\mu_2}\nabla_xu\Vert_{L_t^6L_x^{\infty}} &\lesssim \sum_{\mu \ll \lambda}\sum_{\mu_1 \gtrsim \mu}\Vert P_{\mu_1}n_i'(u)\Vert_{L_t^6L_x^{\infty}}\Vert P_{\mu_1}\nabla_xu\Vert_{L_t^{\infty}L_x^{\infty}} \\
    &\lesssim \sum_{\mu \ll \lambda}\sum_{\mu_1 \gtrsim \mu}\mu_1^{\frac{d}{2}}\Vert P_{\mu_1}n_i'(u)\Vert_{L_t^6L_x^{\infty}}\Vert P_{\mu_1}\nabla_xu\Vert_{L_t^{\infty}L_x^2} \\ 
    &\lesssim t^{-\frac{1}{24}}\sum_{\mu \ll \lambda}\sum_{\mu_1 \gtrsim \mu}\mu_1^{\frac{d}{2}}\mu_1^{-1}\Vert \nabla_xu\Vert_{L_t^6L_x^{\infty}}\mu_1^{-\sigma + 1}a_{\mu_1} \\ 
    &\lesssim t^{-\frac{1}{12}}\Vert \phi\Vert_{H_x^\sigma}\sum_{\mu \ll \lambda}\mu^{-\frac{5}{6}}\sum_{\mu_1 \gtrsim \mu}\mu_1^{-\sigma + \frac{d}{2} + \frac{5}{6}}a_{\mu_1} \\
    &\lesssim t^{-\frac{1}{12}}\Vert \phi\Vert_{H_x^\sigma}^2.
    \end{aligned}
\end{equation*}
This completes the proof.
    \end{proof}
\end{proposition}

\vspace{3mm}

\noindent \textbf{Acknowledgements.}
L. Tu would like to express his gratitude to Dr. Sheng Wang for his continuous and insightful discussions during the early stages of this work, and Dr. Zhixiu Fan, Jie Shao and Zexian Zhang for their valuable comments.
The authors would also like to thank Prof. Naqing Xie for discussions on differential geometry.
This work was supported by the National Natural Science Foundation of China [No. 12571231, No. 12171097], the Key Laboratory of Mathematics for Nonlinear Sciences (Fudan University), the Ministry of Education of China and Shanghai Key Laboratory for Contemporary Applied Mathematics.

\vspace{3mm}

\noindent\textbf{Data availability.} The manuscript has no associated data.

\vspace{6mm}

\noindent\textbf{\Large Declaration}

\vspace{3mm}

\noindent\textbf{Conflict of interest.} The authors state that there is no conflict of interest.

\bibliographystyle{abbrv}
\bibliography{refs}

\end{document}